\documentclass[12pt]{amsbook}

\usepackage{imakeidx}
\makeindex[columns=2, title=Index of mathematical terms] 

\usepackage{seealso}

\usepackage{geometry}
\usepackage{thmtools, thm-restate}
\usepackage{amssymb,xfrac,relsize}
\usepackage{amsmath}
\usepackage{amsthm}
\usepackage{mathrsfs}
\usepackage{enumerate,enumitem}
\usepackage[dvipsnames]{xcolor}
\usepackage[urlcolor=black, colorlinks=true, citecolor=ForestGreen, linkcolor=Plum]{hyperref}
\usepackage{comment}
\usepackage{soul} 
\usepackage{units} 

\usepackage{wrapfig}
\usepackage{tikz, tikz-cd}
\usetikzlibrary{decorations.pathreplacing,calc,shapes.geometric,arrows,matrix,arrows.meta}

\tikzset{
mybrace/.style={decorate,decoration={brace,aspect=#1}}
}

\usepackage{chngcntr}
\counterwithout{figure}{chapter}

\newcommand{\hgline}[2]{
\pgfmathsetmacro{\thetaone}{#1}
\pgfmathsetmacro{\thetatwo}{#2}
\pgfmathsetmacro{\theta}{(\thetaone+\thetatwo)/2}
\pgfmathsetmacro{\phi}{abs(\thetaone-\thetatwo)/2}
\pgfmathsetmacro{\close}{less(abs(\phi-90),0.0001)}
\ifdim \close pt = 1pt
    \draw[blue] (\theta+180:1) -- (\theta:1);
\else
    \pgfmathsetmacro{\R}{tan(\phi)}
    \pgfmathsetmacro{\distance}{sqrt(1+\R^2)}
    \draw[blue] (\theta:\distance) circle (\R);
\fi
}

\newcommand{\hide}[1]{}
\newcommand{\script}{\mathcal}
\newcommand{\parentheses}[1]{{\left( {#1} \right)}}
\newcommand{\sequence}[1]{{\langle {#1} \rangle}}
\newcommand{\p}{\parentheses}
\newcommand{\of}{\parentheses}
\newcommand{\closure}[1]{\overline{#1}}

\newcommand{\interior}[1]{\mathrm{int}\of{#1}}

\newcommand{\Set}[1]{{\left\lbrace {#1} \right\rbrace}}
\newcommand{\singleton}{\Set}

\newcommand{\cardinality}[1]{{\left\lvert {#1} \right\rvert}}

\def\set#1:#2{\Set{{#1} \colon {#2}}}

\newcommand{\diam}[1]{\textnormal{diam}{\left({#1} \right)}}
\newcommand{\mesh}[1]{\textnormal{mesh}{\left({#1} \right)}}

\newcommand{\cat}{^\frown}
\def\Sequence#1:#2{\left( {#1} \colon {#2}  \right)}
\newcommand{\core}[1]{\textnormal{core}\p{#1}}
\newcommand{\ground}[1]{\mathfrak{G}\p{#1}}
\newcommand{\Gmesh}[1]{\mathfrak{G}\textnormal{-mesh}{\left({#1} \right)}}

\renewcommand{\subset}{\subseteq}

\theoremstyle{plain}
\newtheorem{theorem}{Theorem}[section]
\newtheorem*{theorem*}{Theorem}
\newtheorem{lemma}[theorem]{Lemma}
\newtheorem{cor}[theorem]{Corollary}

\newtheorem*{claim*}{Claim}

\newtheorem{blueprint}[theorem]{Blueprint}
\newtheorem{conj}{Conjecture}
\newtheorem{mysubclaim}{Subclaim}

\theoremstyle{definition}
\newtheorem{defn}[theorem]{Definition}

\newtheorem{remark}[theorem]{Remark}

\newcommand{\R}{\mathbb{R}}
\newcommand{\N}{\mathbb{N}}

\makeatletter
\renewcommand{\tocsubsection}[3]{%
  \indentlabel{\@ifnotempty{#2}{\ignorespaces#1 #2\quad}}#3}

\newcommand\@dotsep{4.5}
\def\@tocline#1#2#3#4#5#6#7{\relax
  \ifnum #1>\c@tocdepth 
  \else
    \par \addpenalty\@secpenalty\addvspace{#2}%
    \begingroup \hyphenpenalty\@M
    \@ifempty{#4}{%
      \@tempdima\csname r@tocindent\number#1\endcsname\relax
    }{%
      \@tempdima#4\relax
    }%
    \parindent\z@ \leftskip#3\relax \advance\leftskip\@tempdima\relax
    \rightskip\@pnumwidth plus1em \parfillskip-\@pnumwidth
    #5\leavevmode\hskip-\@tempdima{#6}\nobreak
    \leaders\hbox{$\m@th\mkern \@dotsep mu\hbox{.}\mkern \@dotsep mu$}\hfill
    \nobreak
    \hbox to\@pnumwidth{\@tocpagenum{\ifnum#1=1 \fi#7}}\par
    \nobreak
    \endgroup
  \fi}
\AtBeginDocument{%
\expandafter\renewcommand\csname r@tocindent0\endcsname{0pt}
}
\def\l@subsection{\@tocline{2}{0pt}{2.5pc}{5pc}{}}
\makeatother

\title{Eulerian Spaces}

\author{Paul Gartside and Max Pitz}
\address{Department
    of Mathematics, University of Pittsburgh, Pittsburgh, PA~15260, USA}
\email{gartside@math.pitt.edu} 

\address{Department of Mathematics, Universität Hamburg, Bundesstra\ss e 55, 20146 Hamburg, Germany}
\email{max.pitz@uni-hamburg.de}

\keywords{Eulerian map, edge-wise Eulerian map, topological Euler tour, strongly irreducible map, almost injective map, 1-dimensional continua, brick partitions.}

\subjclass[2010]{Primary: 54F15, 54C10 Secondary: 05C45, 05C63, 57M15, 54F50}

\begin{document}

\begin{abstract}
We develop a unified theory of Eulerian spaces by combining the combinatorial theory of infinite, locally finite Eulerian graphs as introduced by Diestel and K\"uhn with the topological theory of Eulerian continua defined as irreducible images of the circle, as proposed by Bula, Nikiel and Tymchatyn.

First, we clarify the notion of an \emph{Eulerian} space and establish that all competing definitions in the literature are in fact equivalent. Next, responding to an unsolved problem of Treybig and Ward from 1981, we formulate a combinatorial conjecture for characterising the Eulerian spaces, in a manner that naturally extends the characterisation for finite Eulerian graphs. Finally, we present far-reaching results in support of our conjecture which together subsume and extend all known results  about the Eulerianity of infinite graphs and continua to date. In particular, we characterise all one-dimensional Eulerian spaces. 

\end{abstract}

\maketitle

\tableofcontents

\chapter{Introduction}

\section{The Eulerian Problem}

An old, well-known quest in graph theory is to find a natural generalisation for the concept of Eulerian walks to infinite graphs. An equally old problem in topology is to find a theory that allows additional control over space-filling curves from the circle in the form of \emph{strongly irreducible maps}. We show in this paper that these seemingly unrelated strands of research represent two sides of the same coin, and develop a general theory of Eulerian spaces that combines these combinatorial and topological research efforts into a single, unified framework.

There are two main motivations for investigating generalised Eulerian spaces. First, the combinatorial one: recall that a finite multi-graph is \emph{Eulerian}\index{Eulerian graph|textbf} if it admits a  \emph{combinatorial Euler tour}, a closed walk that contains every edge of the graph precisely once. Euler showed, in what is commonly considered the first theorem of graph theory and foreshadowing topology, that a finite connected multi-graph is Eulerian if and only if {it is an \emph{even} graph, i.e.} every vertex has even degree. \index{even graph|textbf} See \cite{graphhistory} for a historical account of Euler's work on this problem. An equivalent characterisation of connected Eulerian graphs, the importance of which was first realised by Nash-Williams \cite{NW}, is that every edge cut is even. An \emph{edge cut}\index{edge cut (graph)|textbf} of a graph $G = (V,E)$ is the set of edges $F\subseteq E$ crossing a bipartition $(A,V \setminus A)$ of the vertices $V$, in other words, the set of all edges with one endvertex in $A$ and the other outside $A$. 

There have been numerous attempts to generalise these results to infinite graphs, see for example \cite{Erdos, NW, NW2, Sabidussi, rothschild, Laviolette}. Since combinatorial Euler tours are inherently finite objects, these attempts focused rather on constructing decompositions of such graphs into cycles or collections of two-way infinite walks, sacrificing the intuitive appeal that an Euler tour should 
return to its start vertex. 
However, around 2000, Diestel and his group started a programme that has the potential to restore this intuitive appeal for locally finite graphs~$G$. Taking as infinite analogues of finite paths and cycles the topological arcs and circles in their Freudenthal compactification~$|G|$, they were able to show that a number of standard results from finite graph theory generalise smoothly to this topological context even when their verbatim infinite analogues fail, see \cite[\S 8]{Diestel} and \cite{DSurv}. Already in the first paper of this programme, Diestel and K\"uhn  \cite{infinitecycles} proved a topological version of Euler's theorem for locally finite graphs: Recall that every graph $G$ naturally turns into a topological space by interpreting each edge as an arc between its endpoints, and each combinatorial Euler tour corresponds naturally to a continuous surjection from the circle $S^1$ to the space $G$ which continuously traverses through the edge-arcs in the order prescribed by the combinatorial walk, henceforth called an \emph{edge-wise Eulerian} map. Diestel and K\"uhn now call an infinite, locally finite (multi-)graph \emph{Eulerian}, if there is such an edge-wise Eulerian surjection from $S^1$ onto the Freudenthal compactification of the graph (formalising the idea that if the Euler tour disappears in some direction towards infinity, then it should again return from that very direction). In this setting, they were able to show that a connected multi-graph is Eulerian if and only if each of its finite edge cuts is even, thus generalising the second of the characterising conditions from the finite case to infinite, locally finite graphs. 

Looking at this result, it seems natural to wonder about Eulerianity in other naturally occurring compactifications of locally finite graphs, which give a more refined meaning for a `direction towards infinity', for example Gromov compactifications of locally finite hyperbolic graphs, or metric completions of edge-length graphs \cite{agelosedgelength}, and the work presented here started out investigating whether for instance compactifications of locally finite graphs with a circle as boundary at infinity are Eulerian in this sense.

\begin{figure}[h!]
\begin{center}
\begin{tikzpicture}[scale=2]

\draw[white] (135:1) .. controls (150:1.3) and  (120:1.3) .. (135:1);
   \draw[white] (-45:1) .. controls (-60:1.3) and  (-30:1.3) .. (-45:1);

\draw (0,0) circle (1);
\clip (0,0) circle (1);

\hgline{90}{270}
\hgline{0}{180}

\foreach \x in {0,90,...,360}{
\hgline{\x-30}{\x+30}
}

\foreach \x in {0,30,...,360}{
\hgline{\x-10}{\x+10}
}

\foreach \x in {0,10,...,360}{
\hgline{\x-3.33}{\x+3.33}
}

\foreach \x in {0,10,...,1080}{
\pgfmathsetmacro\y{\x/3}
\hgline{\y-1.111}{\y+1.111}
}
\end{tikzpicture}\quad%
\begin{tikzpicture}[scale=2]

\draw[white] (135:1) .. controls (150:1.3) and  (120:1.3) .. (135:1);
   \draw[white] (-45:1) .. controls (-60:1.3) and  (-30:1.3) .. (-45:1);

\draw (0,0) circle (1);
\clip (0,0) circle (1);


\foreach \x in {0,45,...,360}{
\hgline{\x-45}{\x+45}
}

\foreach \x in {0,15,...,360}{
\hgline{\x-15}{\x+15}
}

\foreach \x in {0,5,...,360}{
\hgline{\x-5}{\x+5}
}

\foreach \x in {0,5,...,1080}{
\pgfmathsetmacro\y{\x/3}
\hgline{\y-1.666}{\y+1.666}
}
\end{tikzpicture}\quad%
\begin{tikzpicture}[scale=2]

\draw[white] (135:1) .. controls (150:1.3) and  (120:1.3) .. (135:1);
   \draw[white] (-45:1) .. controls (-60:1.3) and  (-30:1.3) .. (-45:1);

\draw (0,0) circle (1);
\clip (0,0) circle (1);


\foreach \x in {0,90,...,360}{
\hgline{\x-60}{\x+60}
}

\foreach \x in {0,30,...,360}{
\hgline{\x-20}{\x+20}
}

\foreach \x in {0,10,...,360}{
\hgline{\x-6.666}{\x+6.666}
}

\foreach \x in {0,10,...,1080}{
\pgfmathsetmacro\y{\x/3}
\hgline{\y-2.222}{\y+2.222}
}
\end{tikzpicture}%
\end{center}
\caption{Three hyperbolic Eulerian structures.}
\end{figure}
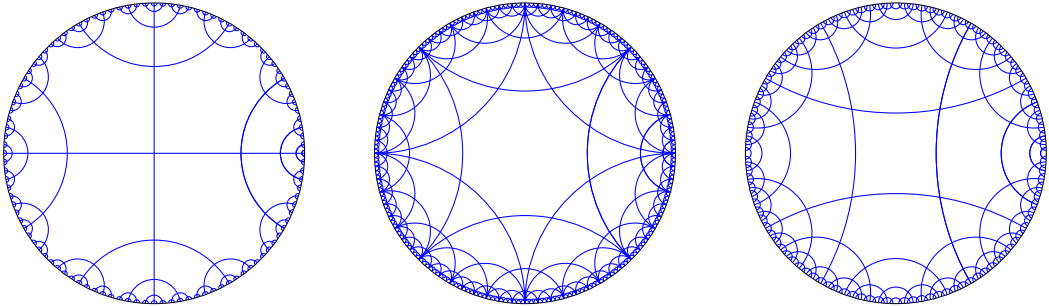
Here we meet our second, topological motivation:  
by the Hahn-Mazurkiewicz Theorem, a space is the continuous image of the circle if and only if it is a \emph{Peano continuum}\index{Peano continuum|textbf} -- a compact, metrisable, connected and locally connected space. Originating with Hilbert's observation (1891) \cite{Hilbert} that the square is a continuous image of the circle so that each point is visited at most three times, the natural question arises which properties beyond `Peano' are needed to guarantee the existence of well-behaved such continuous surjections. Achieving additional control over the surjections from the circle, however, is a notorious open problem in continuum theory discussed, for example, in N\"obling (1933) \cite{noebling33}, Harrold (1940, 1942) \cite{harrold,harrold2}, Ward (1977) \cite{ward}, Treybig \& Ward (1981) \cite[\S4]{treybig},  Treybig (1983) \cite{finiteoscillation}, and Bula, Nikiel \& Tymchatyn (1994) \cite{Koenigsberg}. 
The latter six authors were particularly interested in the existence of \emph{strongly irreducible}\index{strongly irreducible map|textbf} maps from the circle,  continuous surjections $g\colon S^1 \to X$ such that for any proper closed subset $A \subsetneq S^1$ we have $g(A) \subsetneq g(S^1)$. It may not be immediately clear how the property of being strongly irreducible is related to Eulerianity. But using the intermediate value theorem, it is an easy exercise to verify that a strongly irreducible map from $S^1$ onto a finite multi-graph $G$ must sweep through each edge of the graph precisely once without stopping or turning. Hence, a finite graph is Eulerian if and only if it is a strongly irreducible image of the circle. This suggests a second natural candidate for calling an arbitrary Peano continuum Eulerian, namely if it is the strongly irreducible image of the circle.

In this paper we achieve the following goals:
\begin{enumerate}
\item formalise the notion of an \emph{Eulerian} continuum -- all competing definitions in the literature are fortunately shown to be equivalent;
\item formulate a conjecture for characterising the Eulerian Peano continua, in a manner that naturally extends Nash-Williams's condition, and which can be extended to a characterisation in the spirit of Euler; and
\item present far-reaching results in support of our conjecture, confirming it in particular for all one-dimensional Peano continua. 
\end{enumerate}

%
\subsection{Eulerianity} Taking our cue from Bula, Nikiel and Tymchatyn \cite{Koenigsberg}, we say a space $X$ is \emph{Eulerian}\index{Eulerian space|textbf} if it is a strongly irreducible image of the circle, so there is a continuous surjection $g\colon S^1 \to X$ such that for any proper closed subset $A \subsetneq S^1$, we have $g(A) \subsetneq g(S^1)=X$. We also refer to such a map as an \emph{Eulerian map}\index{Eulerian map|textbf}.

Extending Diestel \& K\"{u}hn's definition \cite{infinitecycles}, let us say a space $X$ is \emph{edge-wise Eulerian}\index{edge-wise Eulerian space|textbf} if there is a continuous map of $S^1$ onto $X$ which sweeps through each free arc of $X$ exactly once. Here a \emph{free arc}\index{free arc|textbf} is any inclusion-maximal open subset homeomorphic to $(0,1)$, and by `sweeps once through a free arc' we mean a map such that the preimage of every point in a free arc is a singleton. We also refer to such a map as an \emph{edge-wise Eulerian map}\index{edge-wise Eulerian map|textbf}. 

As remarked earlier, every Eulerian map from $S^1$ onto a space $X$ is edge-wise Eulerian. The converse, however, does not hold on the level of individual functions. Still, as our main result in Chapter~\ref{chapter_Eulerianmaps}, we establish that a space is edge-wise Eulerian if and only if it is Eulerian. The added flexibility of edge-wise Eulerian over Eulerian maps is convenient for constructions, and Chapter~\ref{chap_Eulerdecomp} continues with the development of a versatile framework to establish their existence, which we call \emph{approximating sequences of Eulerian decompositions}. 
Overall, our main results on the different concepts of  Eulerian spaces can be summarised as follows.

\begin{theorem}
\label{thm_MainEquivalence}
For a Peano continuum $X$, the following are equivalent:
\begin{enumerate}[label=(\roman*)]
    \item\label{romani} $X$ is Eulerian, 
    \item\label{romanii} $X$ is edge-wise Eulerian, and
    \item\label{romaniii} $X$ admits an approximating sequence of Eulerian decompositions.
\end{enumerate}
\end{theorem}

The first equivalence $\ref{romani} \Leftrightarrow \ref{romanii}$ is the topic of Chapter~\ref{chapter_Eulerianmaps}, 
and relies on a function space Baire category argument. %
The second equivalence $\ref{romanii} \Leftrightarrow \ref{romaniii}$ is the topic of Chapter~\ref{chap_Eulerdecomp},
and combines the classical strategy of the Hahn-Mazurkiewicz Theorem with inverse limit methods developed by Espinoza and the authors in \cite{euleriangraphlike}. 

%
\subsection{The conjecture} \label{subsec_conj} Let $X$ be a Peano continuum. 
As above a \emph{free arc} is an inclusion-maximal open subset of $X$ homeomorphic to $(0,1)$. We think of free arcs as being the `edges' of $X$. Write $E = E(X)$ for the collection of edges of $X$.\index{edge set|textbf}
\index{E(X)@$E(X)$|see {edge set} \textbf}
For a subset $F \subset E$, we write for brevity $X - F := X \setminus \bigcup F$\index{X-F@$X - F$|textbf}. The \emph{ground-space}\index{ground space|textbf} of $X$ is the (compact metrisable) space $\ground{X} := X - E$\index{G(X)@$\ground{X}$|see {ground space} \textbf}. 
%
Every edge of a Peano continuum has two \emph{endpoints}\index{endpoint|textbf}, which may agree, in which case the edge is a \emph{loop}\index{loop|textbf}. An \emph{edge cut}\index{edge cut (space)|textbf} of a Peano continuum $X$ is a non-empty set $F \subset E(X)$ of {all} edges crossing a partition $A \oplus B$ of $ \ground{X}$ into two disjoint clopen subsets $A$ and $B$. In this case, we also write $F = E(A,B) = E_X(A,B)$\index{E(A,B)@$E(A,B)$|see {edge cut (space)} \textbf}\index{E(A,B)@$E_X(A,B)$|see {edge cut (space)} \textbf}. Every edge cut of a Peano continuum is finite. (See Section~\ref{edge_cuts_degree} for a record of basic results on edge cuts.) With this set-up, we conjecture that Nash-Williams's edge cut characterisation of finite Eulerian graphs extends to all Peano continua: 

\begin{conj}[The Eulerianity Conjecture]
\label{conj_eulerian} \ 

A Peano continuum $X$ is Eulerian if and only if every edge cut of $X$ is even.
\end{conj}

We also say that $X$ satisfies the \emph{even-cut condition}\index{even-cut condition|textbf} or has the \emph{even-cut property}. The condition that an Eulerian continuum has the even-cut property is clearly necessary: if $g$ is an (edge-wise) Eulerian map for $X$, and $F$ is the set of edges crossing a disconnection $A \oplus B$ of $ \ground{X}$, then consider $g$ as a `path' with start and endpoint in $A$, and observe that $g$ must sweep through the edges of $F$ in pairs, from $A$ to $B$ and then back. Also note that an affirmative answer to the conjecture implies the truth of  $\ref{romani} \Leftrightarrow \ref{romanii}$ in Theorem~\ref{thm_MainEquivalence}.

When $X$ is the space underlying a finite multi-graph $G$, then, suppressing vertices of degree two, the edges of $X$ (free arcs) correspond to edges of $G$, and the ground space of $X$ corresponds to the vertex set of $G$. Hence our conjecture naturally encompasses the second characterisation for finite Eulerian graphs. Also, Diestel and K\"uhn's Eulerianity result \cite[Theorem~7.2]{infinitecycles} for the Freudenthal compactification $FG$ of a connected, locally finite graph $G$ mentioned above falls under the scope of Conjecture~\ref{conj_eulerian}: the ground space of $FG$ consists of all vertices and ends of $G$, and edge cuts of $FG$ correspond precisely to the finite edge cuts of $G$.\footnote{For every finite edge cut $E(A,V \setminus A)$ of the graph $G$, the properties of the Freudenthal compactification guarantee that $A$ and $V \setminus A$ have disjoint closures in $FG$, and so $E_G(A,V \setminus A) = E_{FG}(\closure{A}, \closure{V \setminus A})$.} The same holds for Georgakopoulos's \cite{agelos} extension of this result to standard subspaces of Freudenthal compactifications of locally finite graphs.

For Peano continua, Harrold \cite{harrold} showed in 1940 that every Peano continuum without free arcs is Eulerian,\footnote{To be precise, Harrold has shown in \cite{harrold}
that Peano continua in which \emph{the non-local separating points are dense} are strongly irreducible images of  $I$ and $S^1$. However, this condition is equivalent to not having free arcs, as remarked in Harrold's later paper \cite{harrold2}.
}
and in 1994, Bula, Nikiel and Tymchatyn \cite[Theorem~3, Example~2]{Koenigsberg} showed that every Peano continuum obtained by adding a dense collection of free arcs to a Peano continuum is Eulerian.\footnote{As stated, \cite[Theorem~3]{Koenigsberg} excludes edges which are loops, but this assumption is unnecessary.} Both results  are in line with Conjecture~\ref{conj_eulerian}, as with connected ground spaces, these examples have no edge cuts whatsoever, and so the even-cut condition is trivially satisfied. 
In the same paper, Bula, Nikiel and Tymchatyn settled when so-called `completely regular' continua are Eulerian. Call a continuum \emph{graph-like}\index{graph-like space|textbf}\footnote{This notion of `graph-like', by now firmly established in graph theory, is not to be confused with the notion of arc-like, tree-like and graph-like in continuum theory, which we shall not use in this paper.} if its ground space is zero-dimensional, see \cite{infinitematroids, graphlikeplanar, thomassenvella}. In \cite{euleriangraphlike}, Espinoza and the authors showed that a continuum is graph-like if and only if it is completely regular, and equivalently, if and only if it is a standard subspace of the Freudenthal compactification of a locally finite, connected graph. Hence,  these spaces fall also under Conjecture~\ref{conj_eulerian}. 

%
\subsection{Towards the Eulerianity conjecture} All previously known cases for Conjecture~\ref{conj_eulerian} fall under the dichotomy that there are either no free arcs at all, or the free arcs are dense. Our first result towards Conjecture~\ref{conj_eulerian}, which we call the `reduction theorem', clears the middle ground: the problem of establishing the Eulerianity Conjecture for a given space can always be reduced to a space with the same ground space in which the edges are dense. For brevity, such a Peano continuum in which the edges are dense will also be called a \emph{Peano graph}\index{Peano graph|textbf}. Note that Peano graphs are precisely the spaces that can be obtained as Peano compactifications of countable, locally finite graphs.

\begin{theorem}[Reduction Result]
\label{thmIntroReduction}
If the Conjecture~\ref{conj_eulerian} holds for all \textnormal{[}loopless\textnormal{]} Peano graphs, then it holds in general.
\end{theorem}

This result is proved in Theorems~\ref{thm_reduction} and \ref{thmReduction2}. The class of Peano graphs is still surprisingly complex: in Theorem~\ref{thm_nadler} we observe that there is no restriction on the possible ground spaces of an (Eulerian) Peano graph. Our remaining results establish Conjecture~\ref{conj_eulerian} for three large classes of Peano continua, which together subsume and extend every result known to date about the Eulerianity of infinite graphs and  continua.

%
%

\begin{theorem}
\label{thm_crossingfinitearcs}
Conjecture~\ref{conj_eulerian} holds for every Peano continuum whose ground space
\begin{enumerate}[label=(\Alph*)]
    \item\label{THMA} consists of finitely many Peano continua, or
    \item\label{THMB} is homeomorphic to a product $V \times P$, where $V$ is zero-dimensional and $P$ a Peano continuum, or
    \item\label{THMC} is at most one-dimensional.\footnote{Equivalently: the Eulerianity Conjecture holds for all one-dimensional Peano continua.}
\end{enumerate}
\end{theorem}




Indeed, the main results of Harrold \cite{harrold} and Bula-Nikiel-Tymchatyn \cite[Theorem 3]{Koenigsberg} follow either from \ref{THMA} (where the ground space is a single Peano component, and the free arcs are either absent or dense) or \ref{THMB} (by taking $V$ to be a singleton). Diestel and K\"uhn's results for Freudenthal compactifications of graphs, and the results about graph-like spaces from \cite{euleriangraphlike} are covered either by \ref{THMB} (by taking $P$ to be a singleton) or indeed \ref{THMC}. 


However, \ref{THMC} goes significantly beyond these results. Consider for example hyperbolic groups with one-dimensional boundaries, whose Gromov boundaries, provided the groups are one-ended, are either homeomorphic to $S^1$, the Sierpinski carpet, or to the Menger curve \cite[Theorem~4]{Kapovich}. 
Interestingly, `generic' finitely presentable groups are hyperbolic and have the Menger curve as boundary \cite{Champetier}, thus falling once again under \ref{THMC}.  
%
A geometrically interesting class of spaces with $S^1$ boundary is given by the regular tessellations $T(n,k)$ of the hyperbolic plane where precisely $k$ regular $n$-gons surround each vertex 
(for $\nicefrac{1}{k} + \nicefrac{1}{n} < \nicefrac{1}{2}$). 
Since $S^1$ is connected, edge cuts in these spaces can only contain finitely many vertices on one side, so \ref{THMC} implies that $T(n,k)$ is Eulerian if and only if $k$ is even.


    \begin{figure}[h!]
\begin{tikzpicture}[thick,scale=4,x={(1,0)},y={(0,0.8)}]

\draw (0,0) -- (0,1);
\draw (1/2,0) -- (1/2,1);
\draw[red] (5/16,1/2) circle (3/16);

\draw (-1/4,0) -- (-1/4,1);
\draw[red] (-1/8,1/4) circle (4/32);
\draw[red] (-1/8,3/4) circle (4/32);

\draw (7/8,0) -- (7/8,1);
\draw[red] (6/8,1/4) circle (4/32);
\draw[red] (6/8,3/4) circle (4/32);

\draw (-3/8,0) -- (-3/8,1);

\draw (-11/32,0) -- (-11/32,1);
\draw (-9/32,0) -- (-9/32,1);

\foreach \s in {0,...,3} 
{
\pgfmathsetmacro\t{(4*\s + 2)/16}
  \draw[red] (-10/32,\t) circle (1/32);
}

\draw (1,0) -- (1,1);
\draw (31/32,0) -- (31/32,1);
\draw (29/32,0) -- (29/32,1);

\foreach \s in {0,...,3} 
{
\pgfmathsetmacro\t{(4*\s + 2)/16}
  \draw[red] (30/32,\t) circle (1/32);
}

\draw (1/8,0) -- (1/8,1);
\draw (3/32,0) -- (3/32,1);
\draw (1/32,0) -- (1/32,1);
\foreach \s in {0,...,3} 
{
\pgfmathsetmacro\t{(4*\s + 2)/16}
  \draw[red] (2/32,\t) circle (1/32);
}

\draw (5/8,0) -- (5/8,1);
\draw (19/32,0) -- (19/32,1);
\draw (17/32,0) -- (17/32,1);

\foreach \s in {0,...,3} 
{
\pgfmathsetmacro\t{(4*\s + 2)/16}
  \draw[red] (18/32,\t) circle (1/32);
}
\end{tikzpicture}\hspace{48pt} 
\begin{tikzpicture}[thick,scale=4,x={(1,0)},y={(0,0.8)}]

\draw (0,0) -- (0,1);
\draw (1/2,0) -- (1/2,1);
\draw[red] (1/8,1/4) -- (1/2,1/4);
\draw[red] (1/8,3/4) -- (1/2,3/4);

\draw (-1/4,0) -- (-1/4,1);
\draw[red] (-1/4,1/8) -- (0,1/8);
\draw[red] (-1/4,3/8) -- (0,3/8);
\draw[red] (-1/4,5/8) -- (0,5/8);
\draw[red] (-1/4,7/8) -- (0,7/8);

\draw (7/8,0) -- (7/8,1);
\draw[red] (7/8,1/8) -- (5/8,1/8);
\draw[red] (7/8,3/8) -- (5/8,3/8);
\draw[red] (7/8,5/8) -- (5/8,5/8);
\draw[red] (7/8,7/8) -- (5/8,7/8);

\draw (-3/8,0) -- (-3/8,1);

\draw (-11/32,0) -- (-11/32,1);
\draw (-9/32,0) -- (-9/32,1);

\foreach \s in {0,...,7} 
{
\pgfmathsetmacro\t{(2*\s + 1)/16}
  \draw[red] (-11/32,\t) -- (-9/32,\t);
}

\draw (1,0) -- (1,1);
\draw (31/32,0) -- (31/32,1);
\draw (29/32,0) -- (29/32,1);

\foreach \s in {0,...,7} 
{
\pgfmathsetmacro\t{(2*\s + 1)/16}
  \draw[red] (29/32,\t) -- (31/32,\t);
}

\draw (1/8,0) -- (1/8,1);
\draw (3/32,0) -- (3/32,1);
\draw (1/32,0) -- (1/32,1);
\foreach \s in {0,...,7} 
{
\pgfmathsetmacro\t{(2*\s + 1)/16}
  \draw[red] (3/32,\t) -- (1/32,\t);
}

\draw (5/8,0) -- (5/8,1);
\draw (19/32,0) -- (19/32,1);
\draw (17/32,0) -- (17/32,1);

\foreach \s in {0,...,7} 
{
\pgfmathsetmacro\t{(2*\s + 1)/16}
  \draw[red] (17/32,\t) -- (19/32,\t);
}
\end{tikzpicture} 
\caption{The spaces $X$ and $Y$ with ground-space in black and edges in red.}
\label{figureCtimesI}
\end{figure}
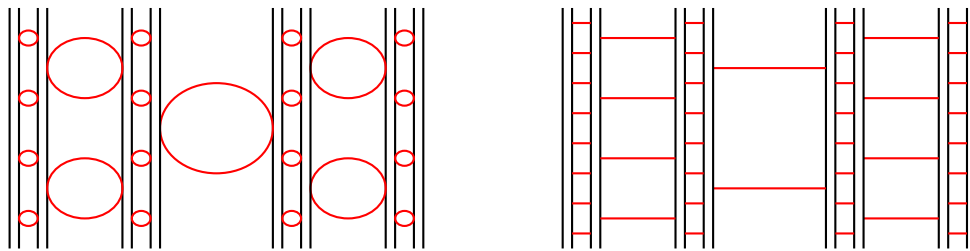
Our result \ref{THMB} answers an open question in the literature, namely (a variant of) \cite[Problem~3]{Koenigsberg}. Its strength lies in supporting Conjecture~\ref{conj_eulerian} by providing non-trivial affirmative examples in all dimensions. To illustrate \ref{THMB}, consider the `fractal' spaces $X$ and $Y$ with ground-space $\ground{X} = \ground{Y} = C \times [0,1]$ in Figure~\ref{figureCtimesI}. 
Both spaces $X$ and $Y$ clearly satisfy the even-cut condition and so are Eulerian by \ref{THMB}. Alternatively, due to the fractal nature of these specific examples, it is possible in both cases to give a geometric, recursive definition of an (edge-wise) Eulerian map in the spirit of Hilbert \cite{Hilbert}. For a different example in which the free arcs are not necessarily dense, consider a Peano continuum $X$ with ground-space a convergent sequence of unit squares, $\ground{X} = (\omega +1 ) \times I^2$, satisfying the even-cut condition.

\begin{figure}[ht!]
\centering
\begin{tikzpicture}[scale=.2,thick]

\def \NZero {6}
\def \NOne {10}	
\def \NThree {10}

\begin{scope}[x={(0.6cm,0.4cm)},y={(0,1)}] 

\draw [fill=black, opacity=.85] (-\NZero-.3,-\NZero+.7) -- (\NZero+.3,-\NZero+.7)  -- (\NZero+.3,\NZero+.3) -- (-\NZero-.3,\NZero+.3) -- (-\NZero-.3,-\NZero+.7);

\node at (\NZero-4,\NZero-4) (gn1start) {};
\node at (\NZero-10,\NZero-8) (gn1start2) {};
\end{scope}

\begin{scope}[shift={(14,0)},x={(0.6cm,0.4cm)},y={(0,1)}]
\node at (\NZero-4,\NZero-4) (gn1end) {};

\node at (\NZero-5,\NZero-2)(gn2start) {};
\node at (\NZero-8,\NZero-8)(gn2start2) {};
\node at (\NZero-2,\NZero-10)(gn2start3) {};
\end{scope}

\begin{scope}[shift={(26,0)},x={(0.6cm,0.4cm)},y={(0,1)}]
\node at (\NZero-6,\NZero-5) (gn2end) {};
\node at (\NZero-2,\NZero-5) (gn2end2) {};
\node at (\NZero-8,\NZero-8)(gn1end2) {};
\node at (\NZero-2,\NZero-9)(gn2end3) {};

\node at (\NZero-6,\NZero-2)(gn3start) {};
\node at (\NZero-10,\NZero-10)(gn3start2) {};
\end{scope}

\begin{scope}[shift={(37,0)},x={(0.6cm,0.4cm)},y={(0,1)}]
\node at (\NZero-10,\NZero-5) (gn3end) {};
\node at (\NZero-2,\NZero-5) (gn3end2) {};
\end{scope}



\draw[red] (gn1start) to[out=30,in=150] (gn1end);
\draw[red] (gn1start2) to[out=-40,in=-140] (gn1end2);


\begin{scope}[shift={(14,0)},x={(0.6cm,0.4cm)},y={(0,1)}]


\draw [fill=black, opacity=.85] (-\NZero-.3,-\NZero+.7) -- (\NZero+.3,-\NZero+.7)  -- (\NZero+.3,\NZero+.3) -- (-\NZero-.3,\NZero+.3) -- (-\NZero-.3,-\NZero+.7);
\end{scope}

\draw[red] (gn2start) to[out=30,in=120] (gn2end2);
\draw[red] (gn2start2) to[out=50,in=150] (gn2end);
\draw[red] (gn2start3) to[out=50,in=150] (gn2end3);

\begin{scope}[shift={(26,0)},x={(0.6cm,0.4cm)},y={(0,1)}]


\draw [fill=black, opacity=.85]  (-\NZero-.3,-\NZero+.7) -- (\NZero+.3,-\NZero+.7)  -- (\NZero+.3,\NZero+.3) -- (-\NZero-.3,\NZero+.3) -- (-\NZero-.3,-\NZero+.7);
\end{scope}

\draw[red] (gn3start) to[out=80,in=100] (gn3end);
\draw[red] (gn3start2) to[out=90,in=90] (gn3end2);

\begin{scope}[shift={(37,0)},x={(0.6cm,0.4cm)},y={(0,1)}]


\draw [fill=black, opacity=.85]  (-\NZero-.3,-\NZero+.7) -- (\NZero+.3,-\NZero+.7)  -- (\NZero+.3,\NZero+.3) -- (-\NZero-.3,\NZero+.3) -- (-\NZero-.3,-\NZero+.7);
\end{scope}

\node at (45,1) {$\ldots$};

\begin{scope}[shift={(53,0)},x={(0.6cm,0.4cm)},y={(0,1)}]


\draw [fill=black, opacity=.85]  (-\NZero-.3,-\NZero+.7) -- (\NZero+.3,-\NZero+.7)  -- (\NZero+.3,\NZero+.3) -- (-\NZero-.3,\NZero+.3) -- (-\NZero-.3,-\NZero+.7);
\end{scope}

\end{tikzpicture}
\caption{A Peano continuum satisfying the even-cut condition with ground space a convergent sequence of squares. Local connectedness implies that endpoints of edges are dense in the right limit square.}
\label{fig_comcos}
\end{figure}
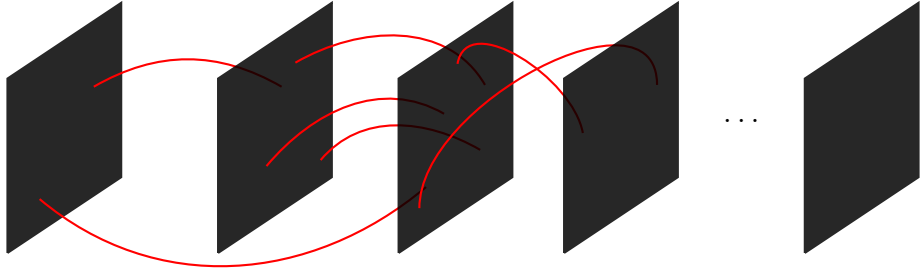


All three results in Theorem~\ref{thm_crossingfinitearcs} rely on our earlier equivalences for Eulerianity given in Theorem~\ref{thm_MainEquivalence}. First, \ref{THMA} follows from an appealing application of the equivalence $\ref{romani} \Leftrightarrow \ref{romanii}$ in Theorem~\ref{thm_crossingfinitearcs}, and will be given, after introducing a modicum of notation, right at the end of the introduction in Section~\ref{subsec_proofTHMA}.

The other two results, \ref{THMB} and \ref{THMC}, utilise the implication $\ref{romaniii} \Rightarrow \ref{romani}$ of Theorem~\ref{thm_MainEquivalence}, and, being rather more involved,  occupy the final two chapters of this paper, Chapter~\ref{chapter_ProductRemainders} and \ref{chapter_1DimRemainders}. As indicated, for both cases the objective is, relying on nothing but the even-cut property, to construct an approximating sequence of Eulerian decompositions for these spaces, in other words, to show that the even-cut condition implies property $\ref{romaniii}$. Carrying out this program requires a combination of powerful techniques from both topology and graph theory. Topologically, we rely on Bing's \cite{Brickpartitions,partitioningold,partitioning} and Anderson's \cite{anderson} theory of Peano partitions, widely regarded as the single most effective structural tool in the theory of Peano continua.
Combinatorially, we rely on the the cycle space theory for locally finite graphs developed in the past 15 years by Diestel {and his group}, see \cite{DSurv} for a survey, and its extension to graph-like spaces developed in \cite{infinitematroids, euleriangraphlike}. 
Roughly, these ingredients are then combined as follows: first, Peano partitions are used to supply a preliminary decomposition of our spaces, whose parts are then carefully modified using combinatorial tools in order to gain control over the edge cuts of the individual parts. 

\medskip

\subsection{Open problems} The main open problem is to establish  Conjecture~\ref{conj_eulerian} for all Peano continua. Motivated by the naturally occurring examples of hyperbolic boundaries, interesting partial results may be about Peano compactifications of locally finite graphs with boundary homeomorphic to $S^2$, $S^3$ and generally $S^n$, and we hope that these examples can also be approached using our theory of \emph{approximating sequences of Eulerian decompositions}. Slightly more general, a result saying that all $2$-dimensional Peano graphs satisfy Conjecture~\ref{conj_eulerian} would be welcome, and might be in reach once the $S^2$-boundary case has been settled.

%
%

\section{Related Conjectures for the Eulerian Problem}

\subsection{Equivalent conjectures}

While calling the free arcs of a Peano continuum $X$ `edges', the points of $\ground{X} = X - E(X)$ should generally not be considered the `vertices' of $X$. 
Instead the `vertices' of $X$ correspond to the connected components of $\ground{X}$. 
Let $X_\sim$\index{X_@$X_\sim$|textbf} denote the quotient of $X$ where we collapse, one by one, each component of the ground space $\ground{X}$ to a point. 
Note that $X_\sim$ is a continuum with 
zero-dimensional ground-space. In other words, the continuum $X_\sim$ is a graph-like Peano continuum. Moreover, the edge cuts of $X$ correspond closely to the edge cuts of $X_\sim$ and vice versa, {and so $X$ satisfies the even-cut condition if and only if $X_\sim$ does, see Lemma~\ref{lem_EvenCutIndAdmissible} for details}. Since we know from \cite{euleriangraphlike} that graph-like continua are Eulerian if and only if they satisfy the even-cut condition, the following is equivalent to the Conjecture~\ref{conj_eulerian}:

\begin{conj}
\label{conj_eulerian2} \ 

A Peano continuum $X$ is Eulerian if and only if $X_\sim$ is Eulerian.
\end{conj}

Since points in a Peano continuum other than a finite graph may have infinite order, the definition of when a point has `even degree' is problematic. Note that these difficulties for generalising Euler's characterisation of Eulerian graphs occur already in the case of locally finite graphs, cf.\ \cite[Fig.~2]{BruhnStein} and \cite{standard}. Nevertheless, from \cite{euleriangraphlike} we know that a graph-like continuum $Y$ is Eulerian if and only if every point $y \in \ground{Y}$ has \emph{even degree} in the sense that there exists a clopen neighbourhood $A$ of $y$ in $\ground{Y}$ such that for every clopen subset $B$ of $\ground{Y}$ with $y \in B \subset A$, the edge cut $E(B,\ground{Y} \setminus B)$ is even. Thus another equivalent version of Conjecture~\ref{conj_eulerian} is that:

\begin{conj}
\label{conj_eulerian3} \ 

A Peano continuum $X$ is Eulerian if and only if every vertex of $X_\sim$ has even degree.
\end{conj}

\subsection{Circle decompositions}
\label{subsec_cycledecomp}

Recall that another classical characterisation of finite Eulerian multi-graphs, due to Veblen, is that the edge set of the graph can be decomposed into edge-disjoint cycles, see \cite[1.9.1]{Diestel}. Accordingly, let us say that the edge set of a Peano continuum $X$ can be \emph{decomposed into edge-disjoint circles}\index{decomposition into edge-disjoint circles|textbf} if there is a collection of edge-disjoint copies of $S^1$ contained in $X$ such that each edge of $X$ is contained in precisely one of them. Generalising the corresponding equivalence for graphs due to Nash-Williams \cite{NW}, we shall prove in Theorem~\ref{thm_circledecomposition} that a Peano continuum has the even-cut property if and only if its edge set can be decomposed into edge-disjoint circles. Consequently, another equivalent version of Conjecture~\ref{conj_eulerian} is that:

\begin{conj}
\label{conj_eulerianDec} \ 
A Peano continuum is Eulerian if and only if its edge set can be decomposed into edge-disjoint circles.
\end{conj}

\subsection{Open Eulerian spaces}

A finite multi-graph is \emph{open Eulerian} if there is a walk starting and ending at distinct vertices, using every edge of the graph precisely once. The open Eulerian multi-graphs are precisely the connected graphs for which all but two vertices have even degree. A Peano continuum $X$ is \emph{open Eulerian} if it is the strongly irreducible image of a map from the unit interval $I=[0,1]$. Let $x \neq y \in X$, and let $X_{xy}$ denote the Peano continuum where we add a new free arc from $x$ to $y$. Then $X$ is open Eulerian from $x$ to $y$ if and only if $X_{xy}$ is Eulerian. Thus, Conjecture~\ref{conj_eulerian} may be used to characterise open Eulerian spaces. 
Moreover, applying the degree characterisation from \cite{euleriangraphlike} when a graph-like continuum is open Eulerian, the following is again equivalent, via the $X_{xy}$ construction, to Conjecture~\ref{conj_eulerian}:

\begin{conj}
\label{conj_eulerian4} \ 

A Peano continuum $X$ is open Eulerian if and only if all but two vertices of $X_\sim$ have even degree. 
\end{conj}

To our knowledge, this conjecture is the first attempt to put forward a proposal for the characterisation of open Eulerian continua and, if correct, would provide a complete answer to \cite[Problem~3]{treybig}. Interestingly, if a Peano continuum $X$ is open Eulerian from $x$ to $y$ for $x,y \in \ground{X}$, then Conjecture~\ref{conj_eulerian} predicts that $X$ is also open Eulerian from $x'$ to $y'$ for all $x'$ (respectively $y'$) that lie in the same component of $\ground{X}$ as $x$ (respectively $y$).

\subsection{The Bula-Nikiel-Tymchatyn conjecture}\label{sect_BNT} 

Our conjecture is not the only contender to characterise Eulerian continua. Bula et al \cite{Koenigsberg} have proposed an alternative,
which is, however, difficult to verify in concrete cases, and implied by Conjecture~\ref{conj_eulerian}.


 A point $x$ of a Peano continuum $X$ is said to be \emph{locally separating} if there is a connected open subset $U$ of $X$ such that $U \setminus \Set{x}$ is disconnected. The set $N(X)$ denotes the set of all $x$ in $X$ such that $x$ is \emph{not} locally separating in $X$. By $Y_X$ denote the quotient of $X$ where we collapse every component of $\closure{N(X)}$ to a single point. By \cite[Theorem~2]{Koenigsberg}, if $Y_X$ is non-trivial then it is a (cyclically completely regular) Peano continuum, and if $X$ is Eulerian then so is $Y_X$. The following is from \cite[Problem~1]{Koenigsberg}:

\begin{conj}[Bula, Nikiel \& Tymchatyn]
\label{conj_bula} \ 

A Peano continuum $X$ is Eulerian if and only if $Y_X$ is Eulerian. 
\end{conj}


Since interior points of edges are locally separating, and $\ground{X}$ is closed, we have $\closure{N(X)} \subseteq \ground{X}$, and hence $(Y_X)_\sim = X_\sim$. In particular, edge cuts of $Y_X$ are in bijective correspondence with edge cuts of $X$, and hence the truth of  Conjecture~\ref{conj_eulerian2} implies the truth of  Conjecture~\ref{conj_bula}.
%
Furthermore, the difference between the two conjectures is not simply formal, as the two quotient spaces $Y_X$ and $X_\sim$ may differ: fix a finite {graph-theoretic} tree $T$, {consider it as a 1-complex} and add to it a dense, zero-sequence of loops. Denote the resulting Peano continuum by $X$, and note that $\ground{X} = T$. Since $T$ is connected, $X_\sim$ is a Hawaiian earring. However, as every point of $T$ apart from the finitely many leaves remains locally separating in $X$, we have $X = Y_X$. 
For a more interesting example where $Y_X$ and $X_\sim$ differ, consider the $\sin(1/x)$-continuum $Z$, {i.e.\ the closure of the $\sin(1/x)$-graph in the plane over $0<x\leq 1$}. Form a Peano continuum $X$ with $\ground{X} = Z$ by first adding a dense, {zero-sequence} of loops to $Z$ (to guarantee $\ground{X} = Z$), and then also adding a nowhere dense, {zero-sequence} of free arcs between points on the {$\sin(1/x)$}-graph and points on the y-axis of $Z$ (to make $X$ locally connected). Again, $X_\sim$ is the Hawaiian earring, but $Y_X$ is an interval with a dense collection of free arcs, since $\closure{N(X)}$ corresponds precisely to the y-axis of $Z$.


\subsection{Further consequences}

Harrold has shown, generalising a result by N\"obling \cite{noebling33}, that every Peano continuum $X$ is the image of a map $g \colon S^1 \to X$ that sweeps through every free arc at most twice, \cite[Theorem~1 ff.]{harrold2}. We observe here that this result
is implied by Conjecture~\ref{conj_eulerian}: for an arbitrary Peano continuum $X$, let $\hat{X}$ denote the space where we add for each edge $e$ of $X$ one additional parallel edge $\hat{e}$. Then $\hat{X}$ is again a Peano continuum (compare with  Lemma~\ref{lem_addingzerosequences} below) which now satisfies the even-cut condition. Hence, there is an Eulerian map $\hat{g} \colon S^1 \to \hat{X}$ that sweeps through every free arc of $\hat{X}$ precisely once. But then it is clear that $\hat{g}$ naturally induces a map $g \colon S^1 \to X$ that uses the original edge $e$ a second time instead of $\hat{e}$ for each $e \in E(X)$. By construction, $g$ has the desired property that it sweeps through every free arc of $X$ precisely twice.

{Conversely, one may strengthen Harrold's result directly to the assertion that every Peano continuum $X$ is the image of a map $g \colon S^1 \to X$ that sweeps through every free arc \emph{precisely} twice, as follows: Ward \cite{ward} has shown that every Peano continuum $X$ is the strongly irreducible image of some dendrite $T$ under a map $\varphi$ say. A moment's reflection shows that we may assume that for each edge $e \in E(X)$ (where $X \neq S^1$), its preimage $\varphi^{-1}(e)$ is an open arc contained in some edge of $T$. Next, since every dendrite $T$ is the image of a map $h \colon S^1 \to T$ that sweeps through every free arc of $T$ precisely twice (this is trivial for finite trees, and lifts canonically to dendrites using any inverse limits), the composition $g = \varphi \circ h$ witnesses the strengthened assertion. But then it is clear that every space with doubled edges, i.e.\ every space of the form $\hat{X}$ for some Peano continuum $X$, is Eulerian and so supports the Eulerianity conjecture.}


\section{Notation and Essentials}

Throughout this paper, all topological spaces are metrisable, and all maps are continuous.  {If $A$ and $B$ are spaces then we denote the topological \emph{disjoint sum}\index{disjoint sum (of spaces)|textbf} of $A$ and $B$ by $A \oplus B$\index{$\oplus$|see {disjoint sum (of spaces)} \textbf}.} A \emph{continuum} is a compact connected metrisable space, a \emph{Peano continuum} is a continuum which is locally connected, and a \emph{Peano graph} is a Peano continuum in which the edges are dense. We write $\N = \Set{0,1,2,\ldots}$ and $[n] = \Set{1,2,\ldots,n}$\index{n@$[n]$|textbf} for $n \in \N$. If $A$ is a subset of the domain of a function $g$, then we denote by $g \restriction A$ the restriction of $g$ to $A$. 

Let $(X,d)$ be a metric space,  $A,B \subset X$ and $\script{A}$ a family of subsets of $X$. {A subset $S$ of $X$ \emph{separates} \index{separator|textbf} $A$ from $B$ in $X$ if each connected component of the subspace $X \setminus S$ intersects at most one of $A$ or $B$. In this case, we also call $S$ an \emph{$A{-}B$-separator}.\index{A-B-separator@$A{-}B$-separator|textbf} Clearly, $A \cap B \subset S$.} We use $A \sqcup B$ to denote disjoint union. A \emph{clopen partition}\index{clopen partition|textbf}  of a space $V$ is a partition of $V$ into pairwise disjoint clopen subsets. If $V$ is compact, then any clopen partition is finite, and we denote by $\Pi(V)$\index{Pi(V)@$\Pi(V)$|see {clopen partition} \textbf} the collection of clopen partitions of $V$. For $\varepsilon > 0$, let $B_\varepsilon(x)$ denote the open $\varepsilon$ ball around $x$, {and set $B_\varepsilon(A) := \bigcup_{x \in A} B_\varepsilon(x)$.} Further, we write $\operatorname{dist}(A,B) = \inf \set{d(a,b)}:{a \in A, \, b \in B}$\index{dist(A,B)@$\operatorname{dist}(A,B)$|textbf}, $\diam{A} := \sup \set{d(a,b)}:{a,b \in A}$\index{diam(A)@$\diam{A}$|textbf}, and $\mesh{\script{A}} := \sup \set{\diam{A}}:{A \in \script{A}}$\index{mesh(A)@$\mesh{\script{A}}$|textbf}. Let $X$ be a metrisable compactum. Then $\script{A}$ is said to be a \emph{null-family}\index{null-family|textbf}, if for any $\varepsilon > 0$, the collection $\set{A \in \script{A}}:{\diam{A} > \varepsilon}$ is finite. By compactness, this does not depend on the metric for $X$. Any null-family $\script{A}$ contains only countably many non-singleton sets. A countable null-family $\script{A}$ is said to be a \emph{zero-sequence}\index{zero-sequence|textbf}. This is equivalent to saying that whenever an enumeration $\script{A} = \Set{A_1,A_2,\ldots}$ is chosen, then $\diam{A_n} \to 0$ as $n \to \infty$.


Let $A,B \subset X$ be disjoint closed subsets. An $A{-}B$-arc\index{A-B-arc)@$A{-}B$-arc|textbf} in $X$ is an arc whose first endpoint lies in $A$, whose last endpoint lies in $B$, and which is otherwise disjoint from $A\cup B$. Finally, a subset $A \subset X$ is \emph{regular closed}\index{regular closed|textbf} if $A = \closure{\interior{A}}$.

\subsection{Edge cuts in Peano continua}\label{edge_cuts_degree}

Free arcs in Peano continua behave much the same as edges in finite graphs, and statements to this effect can be found for example in \cite{Koenigsberg} or \cite{Nikiel}. To make this paper accessible for readers with more of a combinatorial background, we offer brief indications how to prove these basic facts with a minimal topological background, relying only on the fact that Peano continua are (locally) arc-connected.

If $e$ is an edge of $X$, then any point in $\partial e = \closure{e} \setminus e$ is called an \emph{endpoint}\index{endpoint} of $e$. Moreover, with some fixed homeomorphism $e \cong (0,1)$ in mind, we write $e(x) \in e$ for $x \in (0,1)$ to mean the corresponding interior point on $e$, and also write $[a,b]_e$ for the set $\set{e(x)}:{x \in [a,b]}$ and similar for other subsets of the interval.

\begin{lemma}
\label{lem_freearcsS1}
Edges of a Peano continuum $X$ are pairwise disjoint, unless $X = S^1$.
\end{lemma}

\begin{proof}
Suppose $e$ and $f$ are two distinct {edges} which intersect. Since each free arc is maximal with respect to set-inclusion {and is a neighborhood of each of its points}, this amounts to the statement that all $e\setminus f$, $f \setminus e$ and $e \cap f$ are non-empty. 
Let $A$ be a component of $e \cap f$. Then $A$ is a proper subinterval of $e$, and so one endpoint $a$ of $A$ lies in $e \setminus f$. Now if there was a half-open interval $[a,a+\varepsilon)_e \subset e \setminus f$, then this contradicts maximality of $f$. But then connectedness of $f$ implies that $e \setminus \Set{a} \subset f$. However, it follows that $e \cup f = \closure{f} = f \cup \singleton{a}$ is homeomorphic to $S^1$, and is clopen in $X$. So by connectedness, $X = S^1$.
\end{proof}

For the remainder of this paper, when investigating Conjecture~\ref{conj_eulerian} for a space $X$ we always implicitly assume that $X$ is not a simple closed curve, implying that the edge set\index{edge set} $E(X)$ consists of disjoint open sets and that $\ground{X}$ is non-empty.

\begin{lemma}
\label{lem_removing edges} Let $X$ be a Peano continuum.
\begin{enumerate}[label=(\alph*)]
    \item Every edge (free arc) in $X$ contains {at least one and }at most two endpoints.
    \item Removing an edge from $X$ creates at most two connected components which are again Peano continua. Thus, removing $k$ edges from a Peano continuum results in at most $k+1$ components, all of which are again Peano.
    \item If $X \neq S^1$, the edges $E(X)$ form a zero-sequence of disjoint open subsets.
    \item Every edge cut of $X$ is finite.
\end{enumerate}

\end{lemma}

\begin{proof}
(a) Consider a free arc $e \cong (0,1)$ of a Peano continuum $X$. Write for the moment $e(0) = \closure{(0,\frac{1}{2}]} \setminus e$ and $e(1) = \closure{[\frac{1}{2},1)} \setminus e$. By symmetry, it suffices to show that $e(0)$ is a singleton. By compactness, it is certainly non-empty.
Next, since $X$ is locally arc-connected, there exists an $\singleton{\frac12} -e(0)$-arc $\alpha$ in $X$ so that $(0,\frac{1}{2}] \subset \alpha$, and so $\alpha \setminus (0,\frac{1}{2}]$ is precisely the second endpoint of $\alpha$. However, compactness of  $\alpha$ gives $\closure{(0,\frac{1}{2}]} \subset \alpha$, from which it is clear that $e(0)$ consists of at most one point.\footnote{The assumption on local connectedness in (a) is necessary, as witnessed by the unique free arc of the topological sine curve, \cite[1.5]{Nadler}. {Note that we made no assertion whether $e(0) \neq e(1)$; if they agree, we say that the edge $e$ is a \emph{loop}.}}


(b) 
Otherwise, for some edge $e$ the space $X-e$ has a partition into three non-empty, pairwise disjoint compact subsets $A,B,C$. By (a), it follows that one of them, say $A$, does not contain an endpoint of $e$. But then $A$ against $B \cup C \cup \closure{e}$ forms a partition of $X$ into two non-empty, pairwise disjoint compact subsets, contradicting connectedness of $X$.\footnote{Alternatively, assertion (b) can be concluded from the \emph{boundary bumping lemma} \cite[5.7]{Nadler}.}

(c) As a collection of disjoint open subsets (Lemma~\ref{lem_freearcsS1}) in a compact metrisable space, $E(X)$ must be countable, \cite[4.1.15]{Engelking}. Now if $E(X)$ does not form a zero-sequence, then there is $\varepsilon>0$ and infinitely many distinct edges $\Set{e_1,e_2,e_3,\ldots} \subset E(X)$ each containing three successive points $x^1_n<x^2_n<x^3_n \in e_n$ such that $d(x^i_n,x^j_n) \geq \varepsilon$ for all $i \neq j \in [3]$ and $n \in \N$. By moving to convergent subsequences and relabelling, we may assume that {there are $x^1,x^2,x^3$} in $X$ with $x^i_n \to x^i$ for all $i \in [3]$ as $n \to \infty$, and so $d(x^i,x^j) \geq \varepsilon$ for all $i \neq j \in [3]$. However, by local arc-connectedness, for large enough $n$  there exist arcs from $x^2$ to $x^2_n$ of diameter less than $\varepsilon$, a contradiction.\footnote{Alternatively, assertion (c) follows from \emph{compactness of the hyperspace}  \cite[4.14]{Nadler}.}

(d) Trivial for $X=S^1$. Otherwise, the assertion follows from (c) since the sets of any topological disconnection $A \oplus B$ of $\ground{X}$ are disjoint compact, so have $\operatorname{dist}(A,B) > 0$.\footnote{Alternatively, for a proof that does not rely on (c), use normality to find disjoint open sets $U,V \subset X$ separating $A$ from $B$, forming together with $E(A,B)$ an open cover of the compact $X$.}
\end{proof}

From now on, if $e$ is an edge in a Peano continuum $X$, let $e(0),e(1) \in \ground{X}$ denote the two endpoints of that edge.  {Recall we write $e(x)$ for $x \in (0,1)$ to mean the corresponding interior point on $e$, and note we can choose our parametrisation so that $e(x)$ is continuous for $x \in [0,1]$ and a homeomorphism on $(0,1)$.}
If $x$ is an endpoint of an edge $e$, we also write $x \sim e$, or write $e=xy$ to mean that $e(i)=x$ and $e(1-i)=y$ for $i=0$ or $i=1$.  Next, recall from the introduction that for a subset $F \subset E(X)$, we write for brevity $X - F := X \setminus \bigcup F$, and so $\ground{X} := X - E(X)$. If $F = \Set{f}$ is a singleton, we write $X - f$ instead of $X - \Set{f}$. 
Let $X[F] =\closure{\bigcup F} \subset X$\index{X[U]@$X[F]$ \& $X[U]$|see {induced subspace} \textbf} be the subspace of $X$ \emph{induced by $F$}\index{induced subspace|textbf}. Similarly, for $U \subset \ground{X}$, write $E(U) = \set{e=xy \in E(X)}:{\Set{x,y} \subset U}$\index{E(U)@$E(U)$|see {induced edge set} \textbf} for the \emph{induced edge set}\index{induced edge set|textbf} of $U$, and set $X[U] = U \cup E(U)$. Finally, an edge set $F \subset E(X)$ is called \emph{sparse (in $X$)}\index{edge set!sparse|textbf} if $X[F]$ is a graph-like compactum. This notion will be of crucial importance in the final two chapters. Note that if $F$ is sparse, then so is every $F' \subset F$.

A subspace $Y$ of a Peano continuum $X$ is a \emph{standard subspace}\index{standard subspace} if $Y$ contains every edge from $X$ it intersects. Finally, two standard subspaces $Y_1,Y_2$ of $X$ are \emph{edge-disjoint}\index{edge-disjoint|textbf} if every edge of $X$ is contained in at most one $Y_i$.


\subsection{Waiting times for maps from the circle}
\label{sec_waitingtimes}

A map $g \colon I \to X$ or $g \colon S^1 \to X$ which is nowhere constant is also called \emph{light}\index{light map|textbf}. 
The first part of the next lemma is about `avoiding waiting times': given a map $g \colon I \to X$, by contracting all non-trivial intervals in $g^{-1}(x)$ for each $x \in X$, one obtains an associated map that traces out the same path but is, by construction, nowhere constant. The second part describes, in a sense, the converse operation, and says that given a map $g \colon I \to X$, we may add a countable list of waiting intervals, so that the resulting map still traces out the same path. 

\begin{lemma}
\label{lem_waitingtimes}
Let $X$ be a non-trivial Peano continuum.
\begin{enumerate}[label=(\alph*)]
\item\label{waitingA} For every continuous surjection $g \colon I \to X$, there is a continuous light surjection $\hat{g} \colon I \to X$ and a monotonically increasing $m \colon I \to I$ such that $g = \hat{g} \circ m$.
\item\label{waitingB} For every surjection $g \colon I \to X$ and any sequence $(x_0,x_1,\ldots)$ in $X$, there is a  zero-sequence $(J_0,J_1,\ldots)$ of non-trivial disjoint closed intervals of $I$ and monotonically increasing $m \colon I \to I$ such that $\tilde{g}=g \circ m \colon I \to X$ maps each $J_n$ to $x_n$.
\end{enumerate}
Furthermore, the same assertions hold mutatis mutandis for maps $g \colon S^1 \to X$.
\end{lemma}

\begin{proof}
Assertion (a) follows from the \emph{monotone-light-factorisation} \cite[13.3]{Nadler}, and relies on the fact that a quotient of $I$ over closed intervals and points is again homeomorphic to $I$, cf.\ \cite[13.4 \& 8.22]{Nadler}. For (b), pick points $y_n \in g^{-1}(x_n)$ and construct a uniformly converging sequence of monotone surjections $m_n \colon I \to I$ such that $m_n^{-1}(y_i)$ contains a non-trivial interval $J_i$ for $i \in [n]$. The furthermore-part follows by viewing maps $g \colon S^1 \to X$ as maps $g \colon I \to X$ with $f(0) = f(1)$.
\end{proof}

We first illustrate the use of Lemma~\ref{lem_waitingtimes}\ref{waitingB} in following well-known fact.

\begin{lemma}
\label{lem_addingzerosequences}
Suppose $X$ is a compact metrisable space, and $Y,Y_1,Y_2,\ldots$ a zero-sequence of Peano subcontinua of $X$ such that $Y \cap Y_n \neq \emptyset$ for all $ n \in \N$. Then
$ Y' := Y \cup \bigcup_{n \in \N} Y_n \subset X$
is a Peano continuum.
\end{lemma}
\begin{proof}
Pick $y_n \in Y_n \cap Y$ for each $n \in \N$. By Lemma~\ref{lem_waitingtimes}\ref{waitingB}, there is a surjection $h \colon I \to Y$ and non-trivial disjoint closed intervals $J_n \subset I$ such that $h(J_n) = \Set{y_n}$. Fix surjections $h_n \colon I \to Y_n$ such that $h_n(0) = h_n(1) = y_n$. Construct surjections $g_n \colon I \to Y \cup \bigcup_{i \in [n]} Y_i$ by replacing $h \restriction J_i$ by $h_i$ for $i \in [n]$. Then $g_n$ converges uniformly to a continuous surjection $g \colon I \to Y'$ as desired.
\end{proof}

Our second illustration of Lemma~\ref{lem_waitingtimes}\ref{waitingB} lets us combine edge-wise Eulerian maps:

\begin{lemma}
\label{lem_pastingedgewiseEulermaps}
Let $X$ be a Peano continuum and suppose that $Y,Y_1,Y_2,\ldots$ is a zero-sequence of edge-disjoint standard Peano subcontinua of $X$ with $X = Y \cup \bigcup_{n \in \N} Y_n$ such that $Y_n \cap Y \neq \emptyset$. If $Y$ and all $Y_n$ are edge-wise Eulerian, then so is $X$.
\end{lemma}

\begin{proof}
Follow the same proof as in Lemma~\ref{lem_addingzerosequences}, but start with edge-wise Eulerian surjections $h \colon S^1 \to Y$ and $h_n \colon I \to Y_n$. 
\end{proof}

\subsection{An application of the equivalence for edge-wise Eulerianity}
\label{subsec_proofTHMA}

We conclude our introduction with a proof of  Theorem~\ref{thm_crossingfinitearcs}\ref{THMA}. Indeed, given $\ref{romanii} \Rightarrow \ref{romani}$ of Theorem~\ref{thm_MainEquivalence}, the proof of \ref{THMA} reduces to the observation that for these types of spaces, there is a simple procedure for finding an edge-wise Eulerian surjection.

\begin{proof}[Proof of Theorem~\ref{thm_crossingfinitearcs}\ref{THMA} from Theorem~\ref{thm_MainEquivalence}] 
Let $X$ be a Peano continuum such that for its ground space we have $\ground{X}=Z_1 \oplus Z_2 \oplus \cdots \oplus Z_\ell$ where each $Z_i$ is a Peano continuum. Assume further that $X$ has the even-cut property. By $\ref{romani} \Leftrightarrow \ref{romanii}$ of Theorem~\ref{thm_MainEquivalence}, to complete the proof it suffices to show the existence of an edge-wise Eulerian surjection onto $X$. 

Partition the edge set $E(X) = E' \sqcup E''$ where $E' = \bigcup_{ i \in [\ell]} E(Z_i,Z \setminus Z_i)$ consists of the finitely many cross edges between the components of $\ground{X}$, and $E'' = E \setminus E'$ consists of all the edges that have both endpoints attached to the same component of $\ground{X}$. 

Since $X$ satisfies the even-cut condition, $X_\sim[E']$ is a finite Eulerian multi-graph. Take any Eulerian walk $W$ on $X_\sim[E']$ and extend to an edge-wise Eulerian surjection onto $Y=\ground{X} \cup \bigcup E'$ by inserting, between any two successive edges $e Z_i e'$ on $W$ in $\p{X[E']}_\sim$ a surjection onto $Z_i$ from the end vertex of $e$ to the end vertex of $e'$ in $Z_i$. 

Now by Lemma~\ref{lem_removing edges}, the set $E'' = \set{e_n=x_ny_n}:{n \in K}$ for $K \subset \N$ is either finite, or a zero-sequence of edges. Since Peano continua are uniformly locally arc-connected, \cite[Ch.~VI, \S 50,II Theorem~4]{kuratowski}, 
for each $n \in K$ there is an $x_n-y_n$ arc $\alpha_n$ in $\ground{X}$ such that $\diam{\alpha_n} \to 0$. 
Then $Y_n = e_n \cup \alpha_n$ forms a zero-sequence of simple closed curves. Since $Y$ and each $Y_n$ are pairwise edge-disjoint standard subspaces which  are all edge-wise Eulerian, it follows from Lemma~\ref{lem_pastingedgewiseEulermaps} that $X = Y \cup \bigcup_{n \in K} Y_n$ is edge-wise Eulerian, too.
\end{proof}

\chapter{Eulerian Maps and Peano Graphs}
\label{chapter_Eulerianmaps}

\section{Overview}

Recall from the introduction that we had two, seemingly competing notions for generalised Euler tours in a Peano continuum $X$. First, the notion of an \emph{Eulerian map}, a continuous surjection $g$ from the circle that is strongly irreducible: no proper closed subset $A$ of the circle satisfies $g(A) = g(S^1)$. And second the notion of an \emph{edge-wise Eulerian map}, a continuous surjection from the circle that sweeps through every edge of $X$ exactly once. In this chapter we show that both notions for an Eulerian space are in fact equivalent, and thus establish $\ref{romani} \Leftrightarrow \ref{romanii}$ of Theorem~\ref{thm_MainEquivalence}: a Peano continuum is Eulerian if and only if it is edge-wise Eulerian. One implication, namely $\ref{romani} \Rightarrow \ref{romanii}$, is straightforward.

\begin{lemma}\label{sweep}
Every Eulerian map is edge-wise Eulerian.
\end{lemma}

\begin{proof}
Let us first note that by the intermediate value theorem, every strongly irreducible map $g \colon I \to I$ is injective. Otherwise, there are $a<b$ such that $g(a) = x = g(b)$. 
Since $g$ being constant on $[a,b]$ results in an immediate contradiction, there exists $a < c < b$ such that say $g(c) > x$. By the intermediate value theorem, the interval $[x,g(c)]$ is covered by both $g \restriction [a,c]$ and $g \restriction [c,b]$. But then it is clear that for some non-trivial open interval $U \subset [a,c]$ with $g(U) \subset [x,g(c)]$ we have that $g(I \setminus U) = g(I)$, a contradiction.

To prove the lemma, suppose then there is a strongly irreducible map $g \colon S^1 \to X$ onto some Peano continuum $X$, an edge $e \in E(X)$ and an interior point $x \in e$ such that $g^{-1}(x)$ contains at least two distinct points $a$ and $b$. By continuity, there are disjoint closed subintervals $A$ and $B \subset S^1$ containing respectively $a$ and $b$ in their interior such that $g(A)$ and $g(B) \subset e$. By the first part, both $g \restriction A$ and $g \restriction B$ are injective embeddings, and so $g(A)$ and $g(B)$ are subintervals of $e$ containing $x$ in their interior. Thus, there is an open interval $V \subset e$ with $x \in V \subset g(A) \cap g(B)$. But then for some non-trivial open interval $U \subset A$ with $g(U) \subset V$ we have that $g(S^1 \setminus U) = X$, a contradiction.
\end{proof}

The converse of Lemma~\ref{sweep}, however, does not hold in general, and so the equivalence of Eulerian and edge-wise Eulerian spaces cannot hold function-wise: we already observed that edge-wise Eulerian maps are allowed to pause at points in the ground space. Much more significantly, however, consider for example the hyperbolic 4-regular tree $Y$ from the introduction, where an edge-wise Eulerian map is allowed to trace out non-trivial paths on the boundary circle of $Y$, whereas an Eulerian map is not, as in the following Figure~\ref{fig:Euleriantraces}. Indeed, if say $g \restriction [a,b]$ stays on the boundary for a non-trivial time interval $[a,b] \subset S^1$, then $g  \p{S^1 \setminus (a,b)}$, being closed and covering (the closure of) all edges of $Y$, must be the whole space (as $E(Y)$ is dense in $Y$), contradicting the defining property of an Eulerian map. 
\begin{figure}

\includegraphics[scale=.6]{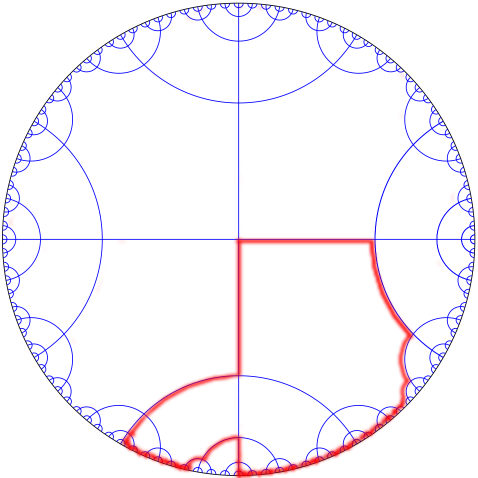} \qquad  \includegraphics[scale=.6]{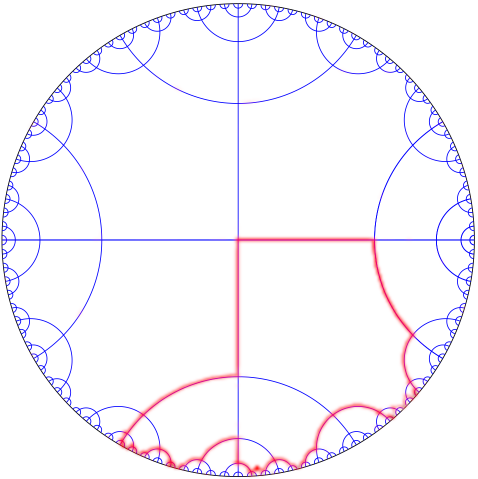} 
%
%
%
%
%
%
%
%
%
%
\caption{Admissible trace of an edge-wise Eulerian map on the left, and an Eulerian map on the right.}
\label{fig:Euleriantraces}
\end{figure}
%
Instead, to establish $\ref{romanii} \Rightarrow \ref{romani}$ in Theorem~\ref{thm_MainEquivalence}, we prove that if there exists an edge-wise Eulerian map $g$ for $X$, then there also exists an Eulerian map $h$ for $X$. 
First, in Section~\ref{eqvt_eul} we establish a number of equivalent definitions for `strongly irreducible'. Most importantly, in the context of Peano graphs (Peano continua whose edges are dense) we can add to the equivalent descriptions that a map $g$ from $S^1$ onto a Peano graph $X$ is  Eulerian if and only if it is edge-wise Eulerian and never spends a positive time interval in the ground space of $X$ (meaning that $g^{-1} (\ground{X})$ does not contain a non-empty open interval), Theorem~\ref{thm_equivalence}. In other words, this behaviour of Eulerian maps that we have seen above is not only necessary, but also sufficient. This natural geometric formulation of `Eulerian map' will be the key to our proof of $\ref{romanii} \Rightarrow \ref{romani}$. 

In  order to harness this geometric intuition, our next step in Section~\ref{sec_reduction} is to establish our reduction result mentioned in the introduction 
so that we may restrict ourselves to Peano graphs. More explicitly, given a Peano continuum $X$ define a Peano graph $X'$ by attaching to $X$ a zero-sequence of loops to a countable dense subset of the interior of the ground space of $X$. It is immediate that $X$ satisfies the even-cut condition if and only if $X'$ does. Crucially we show that $X$ has an Eulerian map if and only if $X'$ has one. Going forward we may always restrict ourselves to Peano graphs, and thus rely on the geometric intuition of an Eulerian map as described above.

Now the strategy is clear: given an edge-wise Eulerian map $g$, we need to modify it so that it remains edge-wise Eulerian, but no longer spends non-trivial time intervals in the ground space. For the problem that edge-wise Eulerian maps may pause at points of the ground space, there is an easy remedy: given any surjection $g \colon S^1 \to X$ onto a non-trivial Peano continuum, by contracting all non-trivial intervals in $g^{-1}(x)$ for each $x \in X$, one obtains an induced edge-wise Eulerian map $\hat{g} \colon S^1 \to X$ which is, by construction, nowhere constant, see Lemma~\ref{lem_waitingtimes}\ref{waitingA}. This observation already establishes $\ref{romanii} \Rightarrow \ref{romani}$ for the class of all graph-like continua, and hence in particular for Freudenthal compactifications of locally finite connected graphs, simply because of the fact that their ground spaces, being totally disconnected, do not contain non-trivial arcs. In fact, this argument shows that for every Peano continuum $X$ whose ground space $\ground{X}$ contains no non-trivial arcs -- if $\ground{X}$ is totally disconnected, but also if it is for example a pseudoarc or any other hereditarily indecomposable continuum \cite[1.23]{Nadler} -- every nowhere constant edge-wise Eulerian map for $X$ is Eulerian. Finally, the harder case, where the ground space does contain non-trivial arcs, will be dealt with in Section~\ref{weak_eqv}. 


\section{Equivalent Definitions for Eulerian Maps}\label{eqvt_eul}

We begin by recalling the following well-studied classes of continuous functions. Let $g\colon X \to Y$ be a continuous map between continua $X$ and $Y$. Then:
\begin{itemize}
\item $g$ is \emph{almost injective}\index{almost injective map|textbf} if the set $\{ x : g^{-1} (g(x)) =\{x\}\}$ is dense in $X$;\footnote{The set of points of injectivity for an almost injective function between compact spaces is not just dense but a dense $G_\delta$, and so large (co-meager) in the sense of Baire category, \cite[Theorem VIII.10.1]{Whyburn}.}

\item $g$ is \emph{irreducible}\index{irreducible map|textbf} if for all proper subcontinua $K \subsetneq X$, we have $g(K) \subsetneq g(X)$;

\item $g$ is \emph{hereditarily irreducible}\index{hereditarily irreducible map|textbf} if for every subcontinuum $K$ of $X$ we have that $g\restriction K$ is irreducible (equivalently, for every pair of subcontinua $A \subsetneq B$ in $X$, we have $g(A) \subsetneq g(B)$);

\item $g$ is \emph{strongly irreducible}\index{strongly irreducible map} if for all closed subsets $A \subsetneq X$, we have $g(A) \subsetneq g(X)$;

\item  $g$ is \emph{arcwise increasing}\index{arcwise increasing map|textbf} if for every pair of arcs $A \subsetneq B$ in $X$ we have $g(A) \subsetneq g(B)$. 
\end{itemize}

In this section we relate these different types of maps, particularly when $X$ is $I$ or $S^1$. The arguments are elementary, and in most cases known or at least folklore. As the results are important for us, and for completeness, we provide brief proofs. For discussions on  hereditarily irreducible and arc-wise increasing images of finite graphs see \cite{Charat,EM}.



\begin{lemma}
\label{l_equivalence}
Let $g\colon I \to Y$ be a continuous surjection. Then the following are equivalent: (a) $g$ is arcwise increasing;
(b) $g$ is hereditarily irreducible; 
(c) $g$ is strongly irreducible; and 
(d) $g$ is almost injective.
\end{lemma}
\begin{proof} 
Clearly, 
$(b) \Leftrightarrow (a)$. For $(a) \Rightarrow (c)$, show the contrapositive. 
So suppose there is a proper closed subset $A$ of $I$ whose image is $g(A)=Y$. Without loss of generality, $A = I \setminus (s,t)$ where $0<s < t<1$. If $g([0,s])=g([0,t])$ then certainly $g$ is not arcwise increasing. 
Otherwise there is an $r$ in $(s,t)$ such that $g(r) \in U:=Y \setminus g([0,s])$. 
By continuity of $g$ at $r$ there is a closed neighbourhood $[a,b]$ of $r$ such that $g([a,b]) \subseteq U$. Since $Y=g(I)=g(A) = g([0,s]) \cup g([t,1])$, we see that $g$ maps $[a,b]$ into $g([t,1])$. Now $g([b,1]) =g([a,1])$ and $g$ is not arcwise increasing. 

For $(c) \Rightarrow (d)$ show that if $g$ is not almost injective then it is not strongly irreducible.\footnote{See  \cite[Theorem VIII.10.2]{Whyburn} for a generalisation of this implication.}
So assume that $\{x : g^{-1} (g(x)) =\{x\}\}$ misses an open interval $(s,t) \subset I$.
This means for all $x \in (s,t)$ there exists $y_x \neq x$ such that $g(x) = g(y_x)$. By the Baire Category Theorem, there is $n \in \N$ and $(s',t') \subsetneq (s,t)$ such that 
$X:=\{x \in (s,t): |x-y_x| \geq 1/n\}$ 
is dense in $(s',t')$. Without loss of generality, $|t' - s'| < 1/n$. But now $g(I \setminus (s',t')) = Y$, since $g(I \setminus (s',t'))$ is closed in $Y$ and contains the set $g(X)$, which was dense in $g(s',t')$. 

For $(d) \Rightarrow (a)$ suppose $f$ is almost injective, and pick subarcs $A \subsetneq B$ in $I$.
Then $B \setminus A$ contains a non-empty open interval which must meet the dense set $\{ x : g^{-1} (g(x)) =\{x\}\}$ say in $x'$.
But then $g(x') \in g(B)  \setminus g(A)$, 
as required for arcwise increasing. 
\end{proof}




Turning to the case of maps from the circle, we deduce that an Eulerian map satisfies all of the following equivalent conditions.
\begin{theorem}
\label{thm_equivalence}
For a continuous surjection $g \colon S^1 \to X$ onto a Peano continuum $X$, the following are equivalent: 
(a) $g$ is arcwise increasing;
(b) $g$ is hereditarily irreducible; 
(c) $g$ is strongly irreducible;  
(d) $g$ is almost injective; and
(e) $g$ is irreducible.

If, additionally,  $X$ is a Peano graph, then the preceding are also equivalent to: 
(f)  $g$ is edge-wise Eulerian and $g^{-1}(\ground{X})$ is zero-dimensional in $S^1$.
\end{theorem}
\begin{proof}
The equivalence of $(a)$ through $(e)$ follows from Lemma~\ref{l_equivalence} and the fact that for $S^1$, every proper closed subset is contained in a proper subcontinuum, giving $(c) \Leftrightarrow (e)$.
Now additionally assume $X$ is a Peano graph.

$(c) \Rightarrow (f)$. Suppose $g$ is  strongly irreducible. By Lemma~\ref{sweep}, $g$ is edge-wise Eulerian. Suppose for a contradiction that $g^{-1}(\ground{X})$ is not zero-dimensional. 
Then there is a non-trivial interval $[a,b] \subset S^1$ such that $g([a,b]) \subset \ground{X}$. However, then $g  \p{S^1 \setminus (a,b)} \supseteq \closure{\bigcup E(X)} = X$, contradicting that $g$ is strongly irreducible.

$(f) \Rightarrow (d)$. For any non-trivial open interval $J \subset S^1$, we have $J \setminus g^{-1}(\ground{X})$ is non-empty, so contains a point $x$ which is mapped under $g$ onto an interior point of some edge of $X$. Since $g$ is edge-wise Eulerian, $x$ is a point of injectivity of $g$. Since $J$ was arbitrary, $g$ is almost injective.
\end{proof}

As mentioned above, the converse to Lemma~\ref{sweep} is false,
 and we may \emph{not} add `$g$ is edge-wise Eulerian' to our list of equivalences, even when restricting to Peano graphs.
Since edge-wise Eulerian maps have, by definition, the geometrically natural property of an `Eulerian path' of sweeping through every edge exactly once, why do we take strongly irreducible as the primary definition of Eulerian? 

The answer is twofold. First, consider, for example, the Gromov compactification  of a locally finite hyperbolic graph $G$ with Gromov boundary $\partial G$.  By property $(f)$, an Eulerian map on $G$ is not allowed to spend any non-trivial time in the boundary $\partial G$. 
Hence, Eulerian maps therefore satisfy the natural property that if  a subpath of the Eulerian map in $G$ `disappears' in some direction $x \in \partial G$ towards infinity along some ray, then it must also return from that very direction $x$ into the graph $G$. 

Our second, equally important reason is that for Peano graphs,  Eulerian maps -- unlike edge-wise Eulerian maps -- can essentially be characterised \emph{purely combinatorially} in terms of a cyclic order and orientation of the edge set, as follows. 

First, fix a Peano graph $X$ and an Eulerian map $g \colon S^1 \to X$. Note that the edges, $E$, of $X$ inherit from $g$ a natural cyclic order. 
Of course the circle, $S^1=\{(\cos (2\pi t), \sin (2\pi t)) \colon t \in [0,1)\}$, has a natural cyclic order and 
(anticlockwise) orientation.  Then any family of open intervals in the circle have an induced cyclic order (pick one point in each interval and use the sub-order).
We have just seen that $g$ is edge-wise Eulerian and $g^{-1}(\ground{X})$ is closed, nowhere dense. 
But this means that the edges, $E$, are in bijective correspondence with the family $\script{U}=\{ g^{-1} (e) : e \in E\}$ of open intervals in $S^1$, which, we note, has dense union. Then $E$ inherits a cyclic order from $\script{U}$.

Second, it is also intuitively clear that, through the natural orientation on $S^1$, any (edge-wise) Eulerian map on a Peano graph   crosses each edge  once in a certain direction, and so induces an orientation of every edge. We make this precise as follows. 
For any spaces $A$ and $B$ let $\script{H}(A,B)$ be the (possibly empty) set of all homeomorphisms from $A$ to $B$, and define $\script{H}(A)=\script{H}(A,A)$ to be the set of all autohomeomorphisms of $A$.  
Every autohomeomorphism of $(0,1)$ (respectively $S^1$) either preserves or reverses the (cyclic) order. For $e \in E(X)$ define an equivalence relation, $\sim_o$, on $\script{H}((0,1),e)$ by $h_1 \sim_o h_2$ if and only if there is an order-preserving $\sigma$ in $\script{H}((0,1))$ such that $h_2 = h_1 \circ \sigma$. 
Then $\script{H}((0,1),e)$ has two equivalence classes under $\sim_o$, corresponding to the two different directions for crossing $e$.
Fix a bijection, $o_e$, between $\script{H}((0,1),e)/{\sim_o}$ and $\{\pm 1\}$. (So $o_e$ randomly assigns a `positive' ($+1$) and `negative' ($-1$) direction to the edge $e$.)
Now suppose we also have an Eulerian map, $g \colon S^1 \to X$. Fix an edge $e$. Fix an order-preserving bijection, $\tau$, between $(0,1)$ and $g^{-1}(e)$, and define $o^\ast_g(e)$  to be $[g \restriction g^{-1}(e) \circ \tau]_{\sim_o}$. (Note that $o_g(e)$ is independent of the choice of $\tau$.) This gives a function $o_g : E \to \{\pm 1\}$ via $o_g(e) = o_e(o^\ast_g(e))$, the orientation of $e$ induced by $g$.

In summary: for a fixed Peano graph $X$ with edge set $E=E(X)$ choose (randomly) a direction $+1$ or $-1$ for each edge, then for any edge-wise Eulerian map $g$ derive combinatorial data of a cyclic order $\le_g$ on $E$ and a function $o_g\colon E \to \{\pm 1\}$ so that $g$ crosses the edges in the order given by $\le_g$ and in the direction given by $o_g$. 

Let us say that another map $g'\colon S^1 \to X$ is \emph{cyclically equivalent} to $g$ if and only if there is an order-preserving  autohomeomorphism, $\varrho$ say, of $S^1$ such that $g' = g \circ \varrho$. 
Then it can be shown that $g$ and $g'$  give the same combinatorial data -- $\le_g$ isomorphic to $\le_{g'}$, and $o_g = o_{g'}$ -- if and only if they are cyclically equivalent.

\smallskip

Now we see how to get from combinatorial data to a function. Fix a Peano graph $X$ with fixed direction for each edge. Let $\le$ be a cyclic order on the edges, $E=E(X)$, and $o$ any function from $E$ into $\{\pm 1\}$. 
Define $g_{\le, o}$ a function from $S^1$ to $X$ as follows.

First select $\script{U}=\script{U}_{\le,o}$, a dense family of {pairwise disjoint} open intervals in $S^1$, which -- in the induced cyclic order -- is isomorphic to $(E,\le)$ (it is well-known that every countable cyclic order can be realised in this fashion), say via $\phi : \script{U} \to E$.
For each $U$ in $\script{U}$, from the randomly assigned direction, $\pm 1$, to the edge $\phi(U)$ compared to the value of $o(\phi(U))$ we get a $\sim_o$ equivalence class in $H((0,1),\phi(U))$ -- let $g^\ast_U$ be any element of this class.
Now select an order preserving bijection, $\tau$ between $U$ and $(0,1)$, and define $g_U = g^\ast_U \circ \tau$.
Define $g_{\le,o}$ to be $g_U$ on each $U$ in $\script{U}$, and  extend, if possible, to a (unique, if it exists)  continuous map from $S^1$ to $X$ (and otherwise extend randomly).

\begin{theorem}
\label{thm_equivalence3}
If $X$ is a Peano graph, with edges $E=E(X)$ and fixed direction for each edge, then the following condition on a continuous surjection $g\colon S^1\to X$ is also equivalent to it being an Eulerian map:
\begin{enumerate}
\item[(g)] there is a cyclic order $\leq$ on $E$ and a function $o \colon E \to \{\pm 1\}$ such that $g$ is cyclically equivalent to $g_{\le, o}$. 
\end{enumerate}
\end{theorem}

\begin{proof}

For $(f) \Rightarrow (g)$, let $g$ be as in $(f)$. Let $\le = \le_g$ and $o=o_g$. Let $\script{U}_g = \{ g^{-1} (e) \colon e \in E\}$ be as above, with the induced cyclic order.
Let $\script{U}=\script{U}_{\le,o}$ be the dense family of open intervals used in the definition of $g_{\le,o}$.  
It is well-known that since $\script{U}$ and $\script{U}_g$ are  dense collections of open intervals which are order-isomorphic, there is an order-preserving autohomeomorphism $\varrho^\ast \in \script{H}( S^1)$ inducing that order-isomorphism. 

Now chasing the definitions, we see that the difference between $g$ and $g_{\le,o} \circ \varrho^\ast$ is caused by choosing the `wrong' class representative on some (possibly, many) intervals $U$ in $\script{U}$. But we can modify $\varrho^\ast$ to get $\varrho$ which is still an order-preserving autohomeomorphism and which `corrects' the mistakes, so $g = g_{\le,o} \circ \varrho$, as required. 

For $(g) \Rightarrow (f)$ note that a function cyclically equivalent to an Eulerian map is Eulerian. So suppose $g=g_{\le,o}$, and $\script{U}=\script{U}_g = \script{U}_{\le,o}$. By construction, $g$ is edge-wise Eulerian, and $g^{-1}(\ground{X} = S^1 \setminus \bigcup \script{U}$ is zero-dimensional, since $\script{U}$ is dense in $S^1$. 
\end{proof}


Finally, we note that Theorem~\ref{thm_equivalence}$(f)$ has the following interesting consequence: it says that if a Peano graph $X$ is Eulerian via an Eulerian map $g$, then $X \cong S^1/{\approx}$ is a quotient of the circle where $\approx$ is the decomposition of $S^1$ into fibres $\set{g^{-1}(x)}:{x \in \ground{X}}$ and points, \cite[3.2.11]{Engelking}. Turning this procedure around, we can engineer (open) Eulerian Peano graphs with prescribed ground spaces as follows:

\begin{theorem}
\label{thm_nadler}
For any compact metrizable space $Z$ there is a Peano graph $X$ with $\ground{X}=Z$. Moreover, for all $x,y \in Z$, the space $X$ can be chosen so that
\begin{enumerate}
	\item $X$ is Eulerian, or
	\item $X$ is open Eulerian from $x$ to $y$.
\end{enumerate}
\end{theorem}


\begin{proof}
Such a construction can be quickly achieved using the \emph{adjunction space} construction, see \cite[A.11.4]{mill} or \cite[2.4.12f]{Engelking}. Let $Z$ be arbitrary. For $(2)$, consider the Cantor middle third set $C \subset I$, and fix a surjection $h \colon C \to Z$ onto $Z$ with $h(0) = x$ and $h(1) = y$  \cite[7.7]{Nadler}. Set $X = I \cup_h Z$, where $I \cup_h Z$ is the quotient of $I$ given by the decomposition into fibres of $h$ and points of $I \setminus C$. By \cite[A.11.4]{mill}, if $g \colon S^1 \to X$ denotes the quotient map, then $g \restriction I \setminus C$ is a homeomorphism (onto the edge set of $X$) and $g(C)$ is homeomorphic to $Z$. Thus, $\ground{X} = Z$ and by Theorem~\ref{thm_equivalence}(f), $g$ is an open Eulerian map from $x$ to $y$.\footnote{For a more explicit construction, we refer the reader to the technique in \cite[Lemma~2.2]{Nadler87}.}

For $(1)$, add one further free arc $e=xy$ to the space $X$ constructed so far.
\end{proof}


%



\section{Reduction to Peano Graphs}
\label{sec_reduction}
The main purpose of this section is to show that in order to prove the Eulerianity conjecture, it suffices to always restrict our attention to the case of Peano graphs, in other words, to Peano continua where the free arcs are dense. This will be done in Section~\ref{subsec_reduction}. In preparation we introduce some background material on Peano continua, Bing's partition theory, and a technical result on almost injective maps from the circle in  Section~\ref{sec_genBandE}. 

In Section~\ref{weak_eqv} the reduction result is used to show the equivalence of Eulerianity and edge-wise Eulerianity, first in Peano graphs, and then in general Peano continua. 

\subsection{Tools for Peano continua}

In the following we shall need Bing's notion of a \emph{partition} of a Peano continuum -- originally from \cite{Brickpartitions,partitioningold}, but we use it in the form of \cite{Mengercurve}. 

\begin{defn}[$\varepsilon$-Peano covers and partitions]
\label{def_Bingpartition}
Let $X$ be a Peano continuum. A \emph{Peano cover}\index{Peano cover|textbf} of $X$ is a finite collection $\script{U}$ of Peano subcontinua of $X$ such that $\script{U}$ covers $X$. 
A Peano cover consisting of regular closed Peano subcontinua additionally satisfying that $\interior{U}$ is connected and $\interior{U} \cap \interior{V}= \emptyset$ for all $U \neq V \in \script{U}$ is called a \emph{Peano partition}\index{Peano partition|textbf}. If $\varepsilon > 0$, then a Peano cover (partition) $\script{U}$ is called an $\varepsilon$ cover (partition) if $\mesh{\script{U}}\leq \varepsilon$.
\end{defn}

\begin{theorem}[{Bing's Partitioning Theorem, \cite{Brickpartitions}}]
\label{thm_BingBrickPartition}
Every Peano continuum admits a decreasing sequence, $\script{U}_n$, of $1/n$-Peano partitions.
\end{theorem}

\subsection{Controlling almost injective maps from the circle}
\label{sec_genBandE}\index{almost injective map}

Harrold, in \cite{harrold}, showed that every Peano continuum without free arcs is the strongly irreducible (equivalently, almost injective) image of the circle, and so is Eulerian. We extend this result -- and also one of Espinoza \& Matsuhashi, see \cite{EM} -- so as to give more control of the map.

For this, we introduce the following notation. Let $A$ and $B$ be spaces. Denote by $\script{C}(A,B)$ the set of all continuous maps from $A$ to $B$. Let $K$ and $L$ be subsets of $A$ and $B$, respectively. Write $\script{S}(A,B; K,L)$ for all elements of $\script{C}(A,B)$ taking $K$ onto $L$, and abbreviate $\script{S}(A,B; A,B)$ by $\script{S}(A,B)$. 
If $X$ is a Peano continuum, then both $\script{C}(I, X)$ and $\script{S}(I, X)$ endowed with the supremum metric $d_\infty$ are (non-empty) complete metric spaces. If in addition $K$ is closed, then $\script{S}(I,X;K,L)$ is a 
closed subspace of $\script{C}(I,X)$ and hence also a complete metric space under the sup-metric.
For sets $T \subset I$ and $g \in \script{S}(I, X)$, we put 
$\script{S}(I,X,g, T) = \set{h \in \script{S}(I, X)}:{h \restriction T=g \restriction T}$.
Note that $\script{S}(I,X,g, T)$ is a non-empty closed subspace of  $\script{S}(I, X)$, so it is itself a complete metric space under the sup-metric.
Lastly, for $F \subset I$ and $\delta > 0$ we put 
\[\script{A}_{F,\delta}(I,X) = \set{h \in \script{S}(I,X)}:{h^{-1}(h(x)) \subset B_\delta(x) \text{ for each } x \in F}\]
and 
\[\script{A}_F(I,X)= \bigcap_{n \in \N} \script{A}_{F,1/n}(I,X) =  \set{h \in \script{S}(I,X)}:{h^{-1}(h(x))=\Set{x} \text{ for each } x \in F}.\] 

\begin{lemma}
\label{lem_basicopen}
Let $X$ be a non-trivial Peano continuum. For each $a \in I$ and $\delta > 0$, the set
$\script{A}_{\singleton{a},\delta}(I,X)$ is open in $\script{S}(I, X)$.
\end{lemma}

\begin{proof}
This result is well-known, and was stated for example (though without proof) in \cite[Lemma 2.3]{sternfeld} and in \cite{harrold}. We briefly sketch the  argument. 

We show that the complement of $\script{A}_{\singleton{a},\delta}(I,X)$ is closed. Suppose that $\set{g_n}:{n \in \N}$ is a sequence of functions in the complement, so for each $n$ there are $x_n,y_n \in I$ with $|x_n - y_n| \geq \delta$ and $g_n(x_n) = a = g_n(y_n)$, such that $g_n \to g$ uniformly. By moving to subsequences and relabeling, we may assume that $x_n \to x$ and $y_n \to y$. But then $|x - y| \geq \delta$ and $g(x) = a = g(y)$. Hence, $g \notin \script{A}_{\singleton{a},\delta}(I,X)$, i.e.\ the complement is closed.
\end{proof}
 
 \begin{theorem}
\label{BandEresult_extended}
Let $X$ be a non-trivial Peano continuum. Let $T,T' \subset I$ and $g \in \script{S}(I,X)$ such that 
\begin{enumerate}
\item $I = T \cup T'$, 
\item $T'$ is closed in  $I$,
\item\label{it3} $Q:=g(T') \subset X$ is a Peano subcontinuum of $X$ without free arcs, and
\item\label{it4} $Q \cap \interior{g(T)} = \emptyset$.
\end{enumerate}
Then for each countable subset $F \subset I$ with 
\begin{enumerate}[resume]
\item\label{it5} $F \cap \closure{T}= \emptyset$,
\end{enumerate}
the set $\script{S}(I,X,g, T) \cap \script{S}(I,X;T',Q) \cap \script{A}_F(I,X)$ is a dense $G_\delta$-subset of $\script{S}(I,X,g,T) \cap \script{S}(I,X;T',Q) = \set{h \in \script{S}(I,X, g, T)}:{h(T')=g(T')}$, and hence non-empty.
\end{theorem}


\begin{proof} As $\script{S}(I,X,g,T) \cap \script{S}(I,X;T',Q)$ is a closed, non-empty subspace of $\script{S}(I,X)$ it is complete under the supremum metric. So the claim that $\script{S}(I,X,g,T) \cap \script{S}(I,X;T',Q) \cap \script{A}_F(I,X)$ is non-empty follows by the Baire Category Theorem once we show that it is a dense $G_\delta$-subset of $\script{S}(I,X,g,T) \cap \script{S}(I,X;T',Q)$.  

Since $\script{A}_F(I,X) = \bigcap_{a \in F}  \bigcap_{m \in \N} \script{A}_{\singleton{a},1/m}(I,X)$, is a countable intersection of open (see Lemma~\ref{lem_basicopen}) sets, it suffices to prove that for each $a \in F$ and each $m  \in \N$, the set $\script{A}_{\singleton{a},1/m}(I,X) \cap \script{S}(I,X,g,T) \cap \script{S}(I,X;T',Q)$
is a dense subset of $\script{S}(I,X,g,T) \cap \script{S}(I,X;T',Q)$. 

So fix some $a \in F$ and $m \in \N$ and consider any map $k \in \script{S}(I,X)$ such that $k$ coincides with $g$ on $T$, and $k(T')=Q$. Take any $\varepsilon > 0 $. We have to find a map $h$ in $\script{A}_{\singleton{a},1/m}(I,X) \cap \script{S}(I,X,g,T) \cap \script{S}(I,X;T',Q)$ with $d_\infty(h,k) < \varepsilon$.

From $k(T)=g(T)$, $k(T') = g(T')$, and (\ref{it3}), (\ref{it4}) and (\ref{it5}), it is straightforward to find a $k' \in  \script{S}(I,X,g,T)\cap \script{S}(I,X;T',Q)$ with $d_\infty(k',k) < \varepsilon/3$ and $k'(a) \notin k(T)$. Next, find a small Peano subcontinuum $P \subset X$ with $k'(a) \in \interior{P} \subset P \subset Q$  and $\diam{P} < \varepsilon / 3$ such that $k'^{-1}(P) \cap T = \emptyset$. After suitably reparameterising $k'$ on $k'^{-1}(P)$ (so that it will be nowhere constant with value $k'(a)$) we obtain a $k'' \in  \script{S}(I,X,g,T)\cap \script{S}(I,X;T',Q)$ such that: 
 $d_\infty(k'',k') < \varepsilon/3$, 
 $k''(a) = k'(a) \notin g(T)=k(T)=k'(T)=k''(T)$, and 
 $k''^{-1}(k''(a))$ is nowhere dense in $I$.

Since $X$ is Peano, there is a basis at $k''(a)$ consisting of Peano subcontinua, in other words, there is a nested sequence of connected, open subsets $U_n$, for $n \in \N$, such that:
 $P_n = \closure{U_n}$ is a Peano subcontinuum of $X$,
 $P_{n+1} \subset U_n$ for all $n \in \N$, 
 $\bigcap_{n \in \N} U_n = \bigcap_{n \in \N} P_n = \singleton{k''(a)}$, 
 $P_0 \subset P$, and  $k''^{-1}(U_0) \cap T = \emptyset$.

We now claim that for some $n$, the compact set $k''^{-1}(P_{n+1})$ is covered by finitely many connected components $(a^n_1,b^n_1), \ldots, (a^n_{N(n)},b^n_{N(n)})$ of the open set $k''^{-1}(U_n)$ such that $|b^n_i - a^n_i| < 1/m$ for all $1 \leq i \leq N(n)$.
Indeed, if not, then by K\"onig's Infinity Lemma \cite[Lemma~8.1.2]{Diestel}, there is a choice of intervals $(a^n_{j(n)},b^n_{j(n)})$ such that:  $|b^n_{j(n)}- a^n_{j(n)}| \geq 1/m$, and 
 $(a^{n+1}_{j(n+1)},b^{n+1}_{j(n+1)}) \subseteq (a^n_{j(n)},b^n_{j(n)})$ for all $n \in \N$.
But then $(a,b) = \bigcap_{n \in \N} (a^n_{j(n)},b^n_{j(n)})$ is an interval of length at least $1/m$ with
$(a,b) = \bigcap_{n \in \N} (a^n_{j(n)},b^n_{j(n)})  \subset \bigcap_{n \in \N} k''^{-1}(U_n) = k''^{-1}(k''(a))$ 
contradicting the fact that $k''^{-1}(k''(a))$ is nowhere dense in $I$.

So let us fix an $n \in \N$ as in the claim and consider $P_{n+1} \subset U_n \subset P_n$. Without loss of generality, assume $a \in (a^n_{N(n)}, b^n_{N(n)})$. Pick arcs $\alpha_i \colon [a^n_i, b^n_i] \to P_{n}$ for $1 \leq i < N(n)$ from $k''(a^n_i)$ to $k''(b^n_i)$ inside $P_{n}$, and note that since $U_{n+1}$ contains no free arcs by (\ref{it3}), the space $\bigcup \alpha_i$ is nowhere dense in $U_{n+1}$. In particular, there is a point $x \in U_{n+1}$ which is not yet covered by any of the $\alpha_i$. Using the Hahn-Mazurkiewicz Theorem, pick a space filling curve $\alpha_{N(n)} \colon [a^n_{N(n)}, b^n_{N(n)}] \to P_n$ from $k''(a^n_{N(n)})$ to $k''(b^n_{N(n)})$, which we may parameterise such that $\alpha_{N(n)} (a) = x$.

Finally, the map $h$ obtained from $k''$ by replacing each $k'' \restriction [a^n_i,b^n_i]$ with $\alpha_i$ for $ i \in [N(n)]$ is as desired. Clearly, $h$ is onto by construction, and $h^{-1}(h(a)) = h^{-1}(x) \subset [a^n_{N(n)}, b^n_{N(n)}]$, so has diameter $< 1/m$ has desired. Further, $k''$ and $h$ differ only within $P_n$, and so $d_\infty(h,k'') \leq \diam{P_n} \leq  \diam{P_0} < \varepsilon/3$. Next, since $k''^{-1}(U_0) \cap T = \emptyset$, we have $h \restriction T= k'' \restriction T =  k \restriction T $ and $h(T') = k''(T') = k(T')$. Finally, we have
\[ d_\infty(h,k) \leq d_\infty(h,k'') + d_\infty(k'',k') + d_\infty(k',k) < \varepsilon /3 + \varepsilon /3 + \varepsilon /3 = \varepsilon\]
and so we have found our surjection $h \in \script{A}_{\singleton{a},1/m}(I,X) \cap \script{S}(I,X,g,T) \cap \script{S}(I,X;T',Q)$ with $d_\infty(h,k) < \varepsilon$, completing the proof.
\end{proof}

\begin{cor}
Let $X$ be a non-trivial Peano continuum without free arcs. Let $T \subset I$ be nowhere dense, and let $g \in \script{S}(I,X)$ such that $g(T)$ is nowhere dense in $X$. Then there is an almost injective map $h \colon I \to X$ with $h \restriction T = g \restriction T$.
\end{cor}

\begin{proof}
As $T$ is nowhere dense, we can find a dense countable subset $F \subset I$ with $F \cap \closure{T} = \emptyset$. Since $g(T)$ is nowhere dense by hypothesis, applying Theorem~\ref{BandEresult_extended} with $T' = S^1$, we obtain an almost injective map $h$ with $h \restriction T = g \restriction T$.
\end{proof}


\begin{remark}
\label{remark_1}
All the results above on almost-injective maps from the closed unit interval, $I$, extend naturally (with the obvious notational changes) to maps from the circle, $S^1$. To see this, note that maps $\hat{g} \colon S^1 \to X$ naturally correspond to maps $g\colon I \to X$ such that $g(0)=g(1)$ and in applying the results, always add $0$ and $1$ to $T$.
\end{remark}

\subsection{The reduction result}
\label{subsec_reduction}

We now show we can reduce the general case of the Eulerianity conjecture (for Peano continua, possibly with some free arcs) to the special case where the free arcs are dense, in other words, to the case of Peano graphs.

Indeed, let $X$ be a Peano continuum with free arcs indexed by $E$. 
Define $X' = X \cup L$ to be the space obtained by attaching a zero-sequence of loops, $L$, to points in a countable dense subset of the part $X \setminus \closure{E}$ of the ground space where the free arcs are not dense.
Then $X'$ is  a Peano graph by Lemma~\ref{lem_addingzerosequences}. It is immediate that $X'$ satisfies the even-cut condition if and only if $X$ does. 
And the next theorem says that $X'$ is Eulerian if and only if $X$ is Eulerian, and so, if the Eulerianity Conjecture holds for $X'$, then it holds for $X$.  

\begin{theorem}[Reduction Result]
\label{thm_reduction}
Let $X$ be a Peano continuum, and $D$ a countable dense subset of $X \setminus \closure{E}$. Define a Peano graph $X'$ by attaching a zero-sequence of loops $L=\set{\ell_d}:{d \in D}$ to points in $D$. 

Then $X'$ is Eulerian if and only if $X$ is Eulerian.
\end{theorem}

\begin{proof}
Enumerate $D = \set{d_n}:{n \in \N}$. First, if $X$ is a Peano continuum, then so is $X' = X \cup \bigcup_{n \in \N} \closure{\ell_{d_n}}$ by Lemma~\ref{lem_addingzerosequences}. Moreover, if $X$ is Eulerian, then so is $X'$, as any almost injective map $g \colon S^1 \to X$ lifts to an almost injective map $g' \colon S^1 \to X'$ by incorporating the loops $\ell_{d_n}$ into $g$ using the results from Section~\ref{sec_waitingtimes}. 

Conversely, assuming that $X'$ is Eulerian, we show $X$ is also Eulerian. To this end, fix an almost injective map $g \colon S^1 \to X'$. Pick a sequence of decreasing $1/n$-Peano partitions $\script{P}_n$ for $X$ (see Definition~\ref{def_Bingpartition} and Theorem~\ref{thm_BingBrickPartition}). Let $\script{P}'_{n+1} \subseteq \script{P}_{n+1}$ be the collection of all $P \in \script{P}_{n+1}$ such that $P$ is 
disjoint from $\closure{E}$, but the unique $Q$ in 
$\script{P}_n$ containing $P$ meets  $\closure{E}$. 
Let $\set{P_j}:{j \in \N}$ be an 
enumeration of $\bigcup_{n \in \N} \script{P}'_{n}$ 
such that $\varepsilon_j = \diam{P_j}$ is monotonically decreasing to $0$ as $j \to \infty$. 
Note that $\interior{P_i} \cap P_j = \emptyset$ whenever $i \neq j$ and that $D \subset \bigcup_{j \in \N} P_j$. Indeed, 
for the last statement note that every $d \in D$ by construction has positive distance from $\closure{E}$, so when the mesh of $\script{P}_n$ is smaller than that distance, there is $P \in \script{P}_n$ such that $d \in P$ and $P \cap \closure{E} = \emptyset$. Finally, observe that each $P_j$ is a Peano subcontinuum of $X$ without free arcs, and so may play the r\^ole of the set $Q$ in item (3) of the previous theorem.

We now define a countable dense set $F \subset S^1$ and a sequence of continuous surjections $g_i \colon S^1 \to X_i$ where $X_i = X' \setminus \set{\ell_d}:{d \in \bigcup_{j < i}P_j }$ such that for all $i \in \N$
\begin{itemize}
\item the set $F$ witnesses that $g_{i}$ is almost injective,
\item $g_i(F) \cap \partial P_j = \emptyset$ for all $j \in \N$,
\item $g_{i+1}$ agrees with $g_i$ on $S^1 \setminus \interior{g_i^{-1}(P_i[X_i])}$, and
\item $g_{i+1} (g_i^{-1}(P_i[X_i]) = P_i$. 
\end{itemize}
[Where for a subcontinuum $P \subset X$ we denote by $P[X_i] = P \cup \set{\ell_d \in E(X_i)}:{d \in P}$, in other words, $P$ with all loops from $L$ that are still present in the space $X_i$.]

Once the construction is complete, we claim that $h = \lim g_i$ is the desired, almost injective surjection from $S^1$ onto $X = \bigcap_{i \in \N} X_i$. Indeed, 
as we change our function value for each point of $S^1$ at most once, and do so inside the target sets $P_i[X_i]$ which are decreasing in size, the sequence is Cauchy and converges to a surjection onto $X$. Moreover, since the sequence $(g_i)_{i \in \N}$ is pointwise eventually constant, it is immediate from the first bullet point that $F$ witnesses that also $h$ is almost injective. 

It remains to complete the construction. Define $g_1 = g$ and let $F \subset g_1^{-1}(E(X'))$ be a countable dense subset of $S^1$ witnessing that $g$ is almost injective (possible by Theorem~\ref{thm_equivalence}(f)). Next, suppose recursively that $g_i$ has already been defined. Consider $T'_i:=g_i^{-1}(P_i[X_i]) \subset S^1$, a closed, compact subspace with non-empty interior (as a positive amount of time is needed to cover the loops $\ell_d$ with $d \in \interior{P_i}$). Let $\set{[a_m,b_m]}:{m \in \N}$ be an enumeration of the maximal non-trivial intervals contained in $g_i^{-1}(P_i[X_i]) $. Then clearly, $g_i(a_m),g_i(b_m) \in \partial P_i = \partial{P_i[X_i]}$.
Consider the natural quotient map $q_i \colon X_i \to X_{i+1}$ which collapses every loop $\ell_d$ in $P_i[X_i]$ onto its base point $d$. Let $g'_i = q_i \circ g_i \colon S^1 \to X_{i+1}$. We then may apply Theorem~\ref{BandEresult_extended} for maps on $S^1$ (see Remark~\ref{remark_1}) to the map $g'_i \in \script{S}(S^1,X_{i+1})$ in order to find a surjection $g_{i+1} \in \script{S}(S^1,X_{i+1},g'_i,T_i) \cap \script{S}(S^1,X_{i+1}; T_i',Q_i) \cap \script{A}_{F_i}(S^1,X_{i+1})$ 
where $T_i = S^1 \setminus \bigcup_{m \in \N} (a_m,b_m)$,   $T_i'= g_i^{-1}(P_i[X_i])$, $Q_i=g'_i(T_i')=P_i$ 
and $F_{i}= \bigcup_{m \in \N} (a_m,b_m) \cap F$.

We claim that $g_{i+1}$ is as desired. That it satisfies the properties of the third and forth bullet points follows from the fact that it is an element of $\script{S}(S^1,X_{i+1},g'_i,T_i)$ and of $ \script{S}(S^1,X_{i+1}; T_i',Q_i)$ respectively.
For the first bullet point, we verify that all points of $F$ are points of injectivity of $g_{i+1}$. Since $g_{i+1} \in \script{A}_{F_i}(S^1,X_{i+1})$, this is clear for points of $F_i \subset F$. Suppose for a contradiction that some $x \in F \setminus F_i$ is no longer a point of injectivity for $g_{i+1}$. Since $g_{i+1} \restriction T_i = g'_i \restriction T_i = g_i \restriction T_i$ and $x$ was a point of injectivity for $g_i$, it must be the case that there is $x' \in (a_m,b_m)$ for some $m \in \N$ such that $g_{i+1}(x) = g_{i+1}(x')$. This, however, implies that $g_{i+1}(x) \in \partial P_i$, but since $g_{i+1}(x) = g_i(x)$, this contradicts the property of the second bullet for $g_i$.
Lastly, it remains to verify that $g_{i+1}(F) \cap \partial P_j = \emptyset$ for all $j \in \N$. This is clear for points in $F \setminus F_i$ as their values are unchanged, and follows for points in $F_i$ from the fact that $g_{i+1} \in \script{A}_{F_i}(S^1,X_{i+1}) \cap \script{S}(S^1,X_{i+1},g_{i+1},T_i)$ readily implies that $g_{i+1}(F_i) \subset \interior{P_i}$.
\end{proof}

\section{Equivalence of Eulerianity and Edge-Wise Eulerianity}\label{weak_eqv}

Recall we have defined a Peano continuum $X$ to be \emph{edge-wise Eulerian} if there is a surjection $g \colon S^1 \to X$ such that $g$ sweeps through every free arc of $X$ precisely once, and we have seen that every Eulerian continuum is edge-wise Eulerian. We now establish the converse, the proof of which establishes the assertion for Peano graphs first, and then, utilizing the reduction result, for general Peano continua.

\begin{theorem}
\label{thm_weaklyEulerianmaps}
\label{thm_weaklyEulerianequivalent}
A space is Eulerian if and only if it is edge-wise Eulerian.
\end{theorem}

\begin{proof}
By Lemma~\ref{sweep}, only the backwards implication requires proof. We first prove this implication for Peano graphs, in other words, when the edges are dense.

The circle has a natural cyclic order where $x \le y \le z$ if we visit $y$ as we travel anticlockwise around the circle starting at $x$ and ending at $z$. Then we say 
a surjection $g \colon S^1 \to X$ is \emph{edge-wise monotone} 
if for every edge $e$ of $X$ its inverse image, $g^{-1} (e)$ is a single open interval in $S^1$ (so $g$ crosses $e$ exactly once) and, after orienting $e$ appropriately, $g$ is monotone (if $x\le y\le z$ in $g^{-1} (e)$ then $g(x) \le g(y) \le g(z)$ in $e$)  from $g^{-1} (e)$ and $e$ (so $g$ may pause when crossing $e$, but does not backtrack). 
Clearly edge-wise Eulerian maps are edge-wise monotone, but observe, also, that if $g$ is edge-wise monotone then, as explained in Lemma~\ref{lem_waitingtimes}(a), we can eliminate the waiting times to get an edge-wise Eulerian map with nowhere dense fibres. 
In any case, it suffices to show that if $X$ has an edge-wise Eulerian  map with nowhere dense fibres then it has an Eulerian map. We do this in two steps.

First of all, let us write $\script{M}(S^1, X) \subseteq \script{S}(S^1,X)$ for the space of edge-wise monotone maps with the sup-metric. We will show that this is a closed subspace, and hence a $G_\delta$ set.
Let us write $\script{W}(S^1, X) \subseteq \script{S}(S^1,X)$ for the space of edge-wise Eulerian maps which have all fibres nowhere dense, with the sup-metric. 
Fix a countable subset $D$ of $S^1$. Noting that a map $g$ from $S^1$ onto $X$ has nowhere dense fibres if and only if for every distinct $d$ and $d'$ from $D$ and every $x$ strictly between them ($d<x<d'$) either $g(x) \ne g(d)$ or $g(x) \ne g(d')$, 
we see that $\script{W}(S^1, X) = \script{M}(S^1, X) \cap \bigcap_{d \ne d' \in D} U_{d,d'}$ where $U_{d,d'} = \bigcup_{d < x < d'} \{ g \in \script{S}(S^1,X) :  g(d) \ne g(x)$ or $g(d') \ne g(x)\}$ is an open set. 
Thus $\script{W}(S^1, X)$ is a non-empty $G_\delta$ subset of $\script{S}(S^1,X)$, which is complete, and
so itself is complete, \cite[4.3.23]{Engelking}. Hence -- by the Baire Category Theorem -- dense $G_\delta$ subsets of $\script{W}(S^1, X)$ are non-empty.

Now to show that $\script{M}(S^1, X) $ is indeed closed, suppose we have a sequence $\Sequence{g_n}:{n \in \N}$ in $\script{M}(S^1,X)$ and $g \in \script{S}(S^1,X)$ with $d_\infty(g_n,g) \to 0$. We need to show that $g \in \script{M}(S^1, X)$, which in turn means we need to show that for every edge $e \in E(X)$, we have $g$ is monotone  on the interval $g^{-1}(e)$.
Fix an edge $e$. It can be oriented in one of two ways. Since the $g_n$'s converge uniformly to $g$, and every $g_n$ is monotone on the interval $g_n^{-1} (e)$ for some orientation of $e$, eventually the orientations must all be the same. So without loss of generality, let us assume $e$ is oriented the same way for all $n$ in $\N$. 
Take any $x,z$ in $g^{-1} (e)$ and any $y$ between them, $x\le y \le z$. Then again by uniform convergence of the $g_n$'s to $g$ and the intermediate value theorem, if $g$ does not respect the order, so we do not have $g(x) \le g(y) \le g(z)$, then for some large enough $n$, $g_n$ will also not respect the order - contradicting $g_n$ being edge-wise monotone.
Now it follows both that $y$ is in $g^{-1} (e)$, which is therefore an interval, and that $g$ is monotone on that interval. Hence, $g \in \script{M}(S^1, X)$ and we have established that $\script{M}(S^1, X) $ is closed.



The second step (for $X$ a Peano graph) is to show that 
for every $a$ in $S^1$ and $\delta > 0$, the set
$\script{A}_{\Set{a},\delta}(S^{1},X) \cap \script{W}(S^1,X) = \set{g \in \script{W}(S^1,X)}:{g^{-1}(g(a)) \subset B_\delta(a)}$
(where $\script{A}_{\Set{a},\delta}(S^{1},X)$ is as defined in Section~\ref{sec_genBandE}) is dense in $\script{W}(S^1,X)$.  Since it is open, see Lemma~\ref{lem_basicopen}, taking any countable dense subset $F \subset S^1$, by Baire Category, there is a function in  $\bigcap_{n \in \N} \bigcap_{a \in F} \script{A}_{\Set{a},1/n}(S^{1},X) \cap \script{W}(S^1,X)$.
This function is then almost injective, so Eulerian by Theorem~\ref{thm_equivalence}, as desired. 

So it remains to check for density. For this, let $g \in \script{W}(S^1, X)$, $a$ in $S^1$ and $\varepsilon > 0$ be arbitrary. Our task is to find $h \in \script{A}_{\Set{a},\delta}(S^{1},X) \cap \script{W}(S^1,X)$ with $d_\infty(g,h) < \varepsilon$.
Since $X$ is Peano, there is a basis at $g(a)$ consisting of Peano subcontinua, so in particular there are connected, open subsets $U_0$ and $U_1$  such that: $\diam{U_0} < \varepsilon/2$, 
$P_1 = \closure{U_1}$ is a Peano subcontinuum of $X$,
and $a \in U_1 \subset P_1 \subset U_0$. 
Clearly, the compact set $g^{-1}(P_{1})$ is covered by finitely many connected components $(a_1,b_1),$ $\ldots,$ $(a_{k},b_{k})$ of the open set $g^{-1}(U_0)$. Relabelling if necessary, assume $a \in (a_{1}, b_{1})$. Let us write $g_i$ for $g \restriction [a_i,b_i]$ where $1\le i \le k$.
We deal with two cases depending on whether or not $g_1$ crosses an edge of $X$.

\noindent \textit{Case 1.} Suppose $g_1$ crosses an edge {$e$} of $X$.
Then we can reparameterise $g_1$ to get $g_1'$ so that $g_1'(a)$ is in $e$. Now define the map $h$ on the circle to be $g_1'$ on $[a_1,b_1]$ and $g$ elsewhere. 
Then $h$ is as desired, indeed $d_\infty(g,h)<\epsilon/2$, $h^{-1}(h(a)) = \Set{a}$ and as $g$ is never constant on a non-trivial interval, by construction of $h$, it too has nowhere dense fibres.

\noindent \textit{Case 2.} Otherwise, by the boundary bumping lemma we know that the image, $\operatorname{ran} g_1$,  
of $g_1 $ is a non-trivial 
subcontinuum of $\ground{X} \cap U_0$. 
In particular, let us fix distinct points $x_1, \ldots, x_{2k-1} \in \operatorname{ran} g_1 $, and -- this is where we assume $X$ is a Peano graph, and the edges are dense --
for each of them a sequence of edges $e^i_n \in U_1$ such that $e^i_n \to x_i$ as $n\to \infty$. 
Now, as $g$ is edge-wise Eulerian, each edge $e^i_n$ must be crossed by precisely one function $g_j$ for $2 \leq j \leq k$.
By the pigeon hole principle we see that for each $i$, 
at least one function $g_{j(i)}$ crosses infinitely many of $\set{e^i_n}:{n \in \N}$. Moreover, since we have $2k-1 = 2(k-1)+1$ many points $x_i$, but only $k-1$ functions, by the  pigeon hole principle again, 
there is one function, say (relabelling if necessary) $g_2$, that is used at least three times, say (after relabelling) for $x_1,x_2,x_3$. 

Now by construction, there are points $y_1,y_2,y_3 \in (a_2,b_2)$ and $\Sequence{z^i_m}:{m \in \N}$ for $i \in [3]$ such that  such: $g_2(y_i) = x_i$, $g_2(z^i_m) \in e^i_{n_m}$
and $z^i_m \to y_i$ as $m \to \infty$.

Relabelling if necessary, let us assume that $y_1<y_2<y_3$, and further, for all $m \in \N$ we have $y_1< z^2_m < y_2$. 
This means, in particular, that $g_2 \restriction [y_1,y_2]$ starts and ends in $\operatorname{ran} (g_1)$ and crosses an edge. Pick $x \leq y \in [a_1,b_1]$ such that $g_1(x) = x_1$ and $g_1(y) = x_2$. 
Then define $g'$ on $S^1$ to be $g$ except  swap $g_1 \restriction [x,y]$ with $g_2 \restriction [y_1,y_2]$. 
Clearly $g'$ is edge-wise Eulerian, has nowhere dense fibres (by construction, given that $g$ has the same property) and has distance $<\epsilon/2$ from $g$. 
Now apply the argument of \textit{Case 1} to $g'$ to get the map $h$. This $h$ is as required: $d_\infty(g,h) \le d_\infty(g,g') + d_\infty(g',h)< \epsilon/2 + \epsilon/2=\epsilon$, and $h$ is in 
$\script{A}_{\Set{a},\delta}(S^{1},X) \cap \script{W}(S^1,X)$.

\medskip

To complete the proof, consider now an arbitrary Peano continuum $X$ which is edge-wise Eulerian. 
Let $g \colon S^{1} \to X$ be a surjection that sweeps through every free arc of $X$ precisely once. 
Let $X'$ be the Peano continuum where we attached a dense zero-sequence of loops of the ground space of $X$, as in Theorem~\ref{thm_reduction}. Then $X'$ is a Peano graph, and $g$ clearly lifts to a surjection $g' \colon S^{1} \to X'$ that sweeps through every free arc of $X'$ precisely once by Lemma~\ref{lem_pastingedgewiseEulermaps}. 
Hence $X'$ is edge-wise Eulerian, and so Eulerian by the first part of this proof. By Theorem~\ref{thm_reduction}, it follows that $X$ is Eulerian, as well. 
\end{proof}

Finally, we conclude this chapter with a further reduction result reducing to the case where we do not have loops.

\begin{theorem}[Loopless reduction result]
\label{thmReduction2}
It suffices to prove the Eulerianity conjecture for Peano graphs without loops. More precisely, Conjecture~\ref{conj_eulerian} holds for a Peano continuum $X$ provided it holds for all loopless Peano graphs $Z$ with $\ground{Z}=\ground{X}$. 
\end{theorem}

\begin{proof}
By the first reduction result {Theorem~\ref{thm_reduction}}, is suffices to consider Peano graphs $X$ only. Since the Eulerianity conjecture holds for spaces $X$ where $\ground{X}$ is a singleton (in which case $X$ is either a circle, a wedge of finitely many circles, or a Hawaiian earring), we may assume that $|\ground{X}| > 1$. So consider such a Peano graph $X$ with $|\ground{X}| > 1$ satisfying the even-cut condition, and let $L = \set{e \in E(X)}:{e(0) = e(1)} \subset E(X)$ be the collection of loops in $X$. 
Then $Y=X-L$ is a Peano continuum, but may no longer be a Peano graph. Let $U = \interior{\closure{\bigcup L}} \cap \ground{X}$. If $U = \emptyset$, set $F :=\emptyset$. Otherwise, let $D = \Set{d_1,d_2,\ldots}$ be a countable dense subset of $U$. Since $X \neq S^1$, no $d_n$ is isolated in $\ground{X}$. For each $d_n$ consider a small Peano continuum neighbourhood $P_n \subset X$ with $d_n \in \interior{P_n} \subset P_n \subset\interior{\closure{\bigcup L}}$. Then $P_n - L \subset \ground{X}$ is a non-trivial Peano continuum. Hence, there exists a small non-trivial arc $\alpha_n \subset \ground{X}$ from $d_n$ to say $x_n$ of diameter $\leq 2^{-n}$. Add a new edge / free arc $f_n$ from $d_n$ to $x_n$ of length $\operatorname{dist}(d_n,x_n)\leq 2^{-n}$, and set $F = \set{f_n}:{n \in \N}$. Then $Z=Y+F$ is a Peano graph with $\ground{Z} = \ground{X}$. Moreover, $Z$ inherits the  even-cut condition from $X$, since loops in $L$ and edges in $F$ each have both their endpoints in the same component of $\ground{X} = \ground{Z}$, and hence to not appear in any finite edge cut. By assumption, there exists an edge-wise Eulerian map $g_Z$ for $Z$. This turns naturally into an edge-wise Eulerian map $g_Y$ for $Y$, by replacing every newly added edge $f_n$ by $\alpha_n$. But using Lemma~\ref{lem_pastingedgewiseEulermaps}, we may incorporate the zero-sequence of loops in $L$ into $g_Y$ in order to obtain an edge-wise Eulerian map $g_X$ for $X$. By Theorem~\ref{thm_weaklyEulerianmaps}, it follows that $X$ is Eulerian.
\end{proof}


\chapter{Approximating by Eulerian Decompositions}
\label{chap_Eulerdecomp}

\index{Eulerian decomposition|(}
\index{decomposition (of Peano continuum)|(}

From the introduction we know that the key task facing us is the construction of Eulerian maps for Peano continua with the even-cut condition. From the last chapter, we know that we may restrict our attention to constructing edge-wise Eulerian maps. The goal for this chapter is then to provide one such construction. In order to do so, we introduce a versatile framework which we call `approximating sequences of Eulerian decompositions', and then show that these can indeed be used to give an edge-wise Eulerian map, thus completing the proof $\ref{romanii} \Leftrightarrow \ref{romaniii}$ announced in Theorem~\ref{thm_MainEquivalence}. The implication $\ref{romanii} \Rightarrow \ref{romaniii}$ is proved in Theorem~\ref{thm_Euleriangivesstrongdecomposition} and $\ref{romaniii} \Rightarrow \ref{romanii}$ is proved in Theorem~\ref{thm_WeaklyEulerianMappingThm2}.

The idea behind this framework of Eulerian decompositions lies in the observation that any edge-wise Eulerian map induces a countable cyclic order on the edge set $E(X)$ of our Peano continuum $X$. As in the case of graph-like spaces \cite{euleriangraphlike}, we want to approximate such a cyclic order on a finitary version of $X$, and then choose a sequence of compatible approximations that `converge' to the desired cyclic order on $X$. In this chapter, we formalise this idea. We describe what we understand about finite approximations and lay down a set of rules that these have to satisfy in order to make the ideas of `compatible' and `converging' mathematically sound, and then state and prove our main mapping result, Theorem~\ref{thm_WeaklyEulerianMappingThm2}, for constructing edge-wise Eulerian maps.


\section{Eulerian Decompositions}

 An important tool in structural graph theory is the notion of a \emph{tree-decomposition}, due to Halin \cite{halin}, and rediscovered and made widely known by Robertson and Seymour in their graph-minors project \cite{graphminors2}. Roughly, a tree decomposition $(T,\tau)$ of a graph $G$ consists of a tree $T$ and a map $\tau$ such that $\tau(t)$ is a subgraph of $G$ for every $t \in V(T)$, such that the various subgraphs (`parts') $\set{\tau(t)}:{t \in V(T)}$ form a cover of the graph $G$ whose elements are roughly  arranged like $T$, see also \cite[\S 12.3]{Diestel}. 

In analogy, we will now consider Eulerian decompositions: covers of a Peano continuum $X$ by finitely many parts which are arranged roughly like an Eulerian graph.

\subsection{Setup and definitions}

\begin{defn}
\label{def_standardsubspace}
Let $X$ be a Peano continuum. A subspace $Y \subset X$ is called \emph{standard}\index{standard subspace|textbf} if $Y$ contains all edges of $X$ it intersects.
\end{defn}


Recall that for an edge $e$ of a finite multi-graph or a Peano continuum, we write $e(0)$ and $e(1)$ for the two end vertices of $e$ (if $e$ is a loop, then $e(0)=e(1)$), see Lemma~\ref{lem_removing edges}.

\begin{defn}[Eulerian decomposition]
\label{def:Eulerdecomp}\index{Eulerian decomposition|textbf}
Let $X$ be a Peano continuum, $G$ be a finite multi-graph with bipartitioned edge set $E(G) = F \sqcup D$, and $\eta$ be a map with domain $V(G) \cup E(G)$ such that
\begin{enumerate}[label=(E\arabic*)]
\item\label{Eta1} $\eta (v)$ is a non-empty standard Peano subcontinuum of $X$ for all $v \in V(G)$,
\item\label{Eta2} $\eta(f) \in E(X)$ for all $f \in F$, and
\item\label{Eta3} $\eta(d) \subset \ground{X}$ is a (possibly trivial) arc for all $d \in D$.
\end{enumerate}
The pair $(G, \eta)$ is called a \emph{decomposition}\index{decomposition (of Peano continuum)|textbf}\footnote{Note that due to \ref{Eta2} and \ref{Eta3}, the information $E(G) = F \sqcup D$ is encoded in $\eta$.} of $X$ if it satisfies the following four conditions:
\begin{enumerate}[resume, label=(E\arabic*)]
\item\label{E1a} the family $\set{\eta(x)}:{x \in V \cup F}$ forms a cover of $X$,
\item\label{E1b} the elements of $\set{\eta(x)}:{x \in V \cup F}$ are pairwise $E(X)$-edge-disjoint,\footnote{This implies that $\eta \restriction F$ is injective; however, for distinct vertices $v$ and $w$ of $G$, $\eta(v) = \eta(w)$ could be the same tile, which must then be contained in the ground space. Note also that $\eta (v)$ could contain free arcs which are not free in $X$. These don't play a role for the requirement of edge-disjoint.}
\item\label{E2} $(\eta(f))(j) \in \eta(f(j))$ for all $f \in F$ and $j \in \Set{0,1}$, and
\item\label{E3} $(\eta(d))(j) \in \eta(d(j))$ for all $d \in D$ and $j \in \Set{0,1}$. 
\end{enumerate}
The \emph{width}\index{width (of decomposition)|textbf} of a decomposition is $w(G,\eta):=\max \set{\diam{\eta(v)}}:{v \in V}$\index{w(G,e)@$w(G,\eta)$|see {width (of decomposition)} \textbf}. 
The edges in $F$ are also called \emph{real}\index{real edge|textbf} or \emph{displayed}\index{displayed edge|textbf} edges, and the edges in $D$ are the \emph{dummy}\index{dummy edge|textbf} edges of $G$. The elements $\set{\eta(v)}:{v \in V}$ are called \emph{tiles}\index{tile|textbf} of the decomposition. A decomposition $(G, \eta)$ where $G$ is Eulerian, is called an \emph{Eulerian decomposition} of $X$.
\end{defn}

Dummy edges $d$ between vertices $v,w$ of $ G$ represent the possibility of moving from tile $\eta(v)$ to $\eta(w)$ through a common point in their overlap (if $\eta(d)$ is a singleton) or through an arc  contained in the ground space of $X$ (if $\eta(d)$ is a non-trivial arc). 
As an illustration, consider two Eulerian decompositions of the hyperbolic 4-regular tree $X$.

\begin{figure}[h!]
\begin{tikzpicture}[scale=2.2]



\node[draw, red, circle,scale=.4, fill=red, label={[label distance=-.1]135:$\textcolor{red}{\delta_1}$}](D1) at (135:1) {};
\node[draw, red, circle,scale=.4, fill=red, label={[label distance=-.1]-45:$\textcolor{red}{\delta_2}$}](D2) at (-45:1) {};

\clip (0,0) circle (1);
\draw[fill=red!20!white,draw=red!20!white] (1,0) ellipse (0.6 and .8); 
\draw[fill=red!20!white,draw=red!20!white] (-1,0) ellipse (0.6 and .8);
\draw[fill=red!20!white,draw=red!20!white] (0,1) ellipse (0.8 and .6); 
\draw[fill=red!20!white,draw=red!20!white] (0,-1) ellipse (0.8 and .6);


\draw (0,0) circle (1);
\clip (0,0) circle (1);

\hgline{90}{270}
\hgline{0}{180}

\foreach \x in {0,90,...,360}{
\hgline{\x-30}{\x+30}
}

\foreach \x in {0,30,...,360}{
\hgline{\x-10}{\x+10}
}

\foreach \x in {0,10,...,360}{
\hgline{\x-3.33}{\x+3.33}
}

\foreach \x in {0,10,...,1080}{
\pgfmathsetmacro\y{\x/3}
\hgline{\y-1.111}{\y+1.111}
}

\node[draw, circle,scale=.2, fill=black, label={[label distance=-.1]45:$x$}](C) at (0,0) {};

\node[draw, red, circle,scale=.4, fill=red](D11) at (135:1) {};
\node[draw, red, circle,scale=.4, fill=red](D22) at (-45:1) {};

\end{tikzpicture}%
\quad\quad\quad
\begin{tikzpicture}[scale=1.2]

\node[draw, circle,scale=.2, fill=black, label={[label distance=-.1]45:$v$}](C) at (0,0) {};

\node[draw, rectangle,scale=.6, fill=red!20!white]  (N1) at (0,1.5) {};
 \node[draw, rectangle,scale=.6, fill=red!20!white] (N2) at (-1.5,.0) {};
 \node[draw, rectangle,scale=.6, fill=red!20!white] (N3) at (0,-1.5) {};
 \node[draw, rectangle,scale=.6, fill=red!20!white] (N4) at (1.5,0) {};
 
 \node[rectangle,scale=.6, fill=white]  (N666) at (0,-2) {};
 
 \node[rectangle,scale=.6, fill=white]  (N666) at (-2,0) {};

 \draw[blue] (C) -- (N1);
 \draw[blue] (C) -- (N2);
 \draw[blue] (C) -- (N3);
 \draw[blue] (C) -- (N4);
 
 \node () at (-1,-1) {$G$};

  \draw[red]    (N1) to[out=180,in=90] (N2);
    \node[label={[label distance=-.1]135:$\textcolor{red}{d_1}$}](T1) at (135:1.3) {};
   \draw[red]    (N3) to[out=0,in=270] (N4);
   \node[label={[label distance=-.1]-45:$\textcolor{red}{d_2}$}](T2) at (-45:1.3) {};

\end{tikzpicture}%
\quad
\begin{tikzpicture}[scale=1.2]

\node[draw, circle,scale=.2, fill=black, label={[label distance=-.1]225:$v_2$}](C) at (-.1,-.1) {};
\node[draw, circle,scale=.2, fill=black,, label={[label distance=-.1]45:$v_1$}](C1) at (.1,.1) {};

\node[draw, rectangle,scale=.6, fill=red!20!white]  (N1) at (0,1.5) {};
 \node[draw, rectangle,scale=.6, fill=red!20!white] (N2) at (-1.5,.0) {};
 \node[draw, rectangle,scale=.6, fill=red!20!white] (N3) at (0,-1.5) {};
 \node[draw, rectangle,scale=.6, fill=red!20!white] (N4) at (1.5,0) {};
 
 \node[rectangle,scale=.6, fill=white]  (N666) at (0,-2) {};
 
 \node[rectangle,scale=.6, fill=white]  (N666) at (-2,0) {};

 \draw[blue] (C1) -- (N1);
 \draw[blue] (C) -- (N2);
 \draw[blue] (C) -- (N3);
 \draw[blue] (C1) -- (N4);
 
 \node () at (-1,-1) {$G'$};

   \draw[red]    (N1) to[out=180,in=90] (N2);
    \node[label={[label distance=-.1]135:$\textcolor{red}{d_1}$}](T1) at (135:1.3) {};
   \draw[red]    (N3) to[out=0,in=270] (N4);
   \node[label={[label distance=-.1]-45:$\textcolor{red}{d_2}$}](T2) at (-45:1.3) {};

\end{tikzpicture}
\caption{Two Eulerian decompositions $\p{G,\eta}$ and $\p{G',\eta'}$ for $X$ with tiles in pink and black (single vertices), displayed edges in blue, dummy edges $\eta(d_i) = \Set{\delta_i} = \eta'(d_i)$ in red, and $\eta(v) = \Set{x} = \eta'(v_i)$.}
\end{figure}
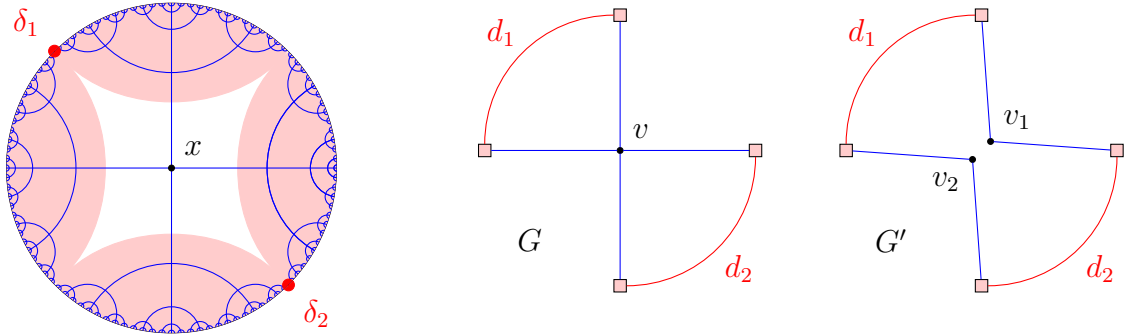

Recall that an \emph{edge-contraction}\index{edge-contraction|textbf} is the combinatorial analogue of collapsing the closure of an edge in a topological graph to a single point. Formally, given an edge $e=xy$ in a multi-graph $G=(V,E)$ (with parallel edges and loops allowed), the contraction $G / e$ is the graph with vertex set $V \setminus \Set{x,y} \sqcup \Set{v_e}$ and edge set $E \setminus \Set{e}$, and every edge formally incident with $x$ or $y$ of $G$ is now incident with $v_e$. Note that all edges parallel to $e$ are now loops in $G / e$. If $e$ was a loop in $G$, then $G/e = G -e$. The contraction of more than one edge is denoted by $G / \sequence{e_1,\ldots,e_k}$\index{G/@$G / e$ \& $G / \sequence{e_1,\ldots,e_k}$|see {edge-contraction} \textbf}. The order in which we contract edges does not matter. Any such graph $G'$ which can be obtained by a sequence of contractions from $G$ is called a \emph{contraction minor}\index{contraction minor|textbf} of $G$, denoted by $G'  \preccurlyeq G$.

\begin{lemma}[Contractions on Eulerian decompositions.]
\label{lem:EulerdecompContraction}
Suppose $\script{D} = (G,\eta)$ is an \textnormal{[}Eulerian\textnormal{]} decomposition of $X$ with edge partition $E= E(G) = F \sqcup D$. Then for an arbitrary edge $e=xy \in E$, there is an \textnormal{[}Eulerian\textnormal{]} decomposition $\script{D}/e := (G', \eta')$ where $G'=G / e$, $E' = E - e$ with induced partition $F' \sqcup D'$, and the function $\eta'$ given by
\begin{enumerate}[label=\textnormal{(C\arabic*)}]
\item\label{C1} $\eta' (v_e) = \eta(x) \cup \eta(e) \cup \eta(y)$,
\item\label{C2} $\eta' (v) = \eta(v)$ for all $v \neq v_e$, and
\item\label{C3} $\eta'(e') = \eta(e')$ for all $e' \in E'$.
\end{enumerate}
\end{lemma}

\begin{proof}
By property \ref{E2} and \ref{E3} for $\script{D}$ (depending on whether $e \in F$ or $e \in D$ respectively), we have that $\eta' (v_e)$ is a standard subcontinuum of $X$. The remaining properties are easily verified.

Finally, it is clear that if $G$ is Eulerian, then so is $G'$.
\end{proof}

\begin{defn}
\label{def_extendingEulerDecomp}
For two decompositions $\script{D}_1= (G_1,\eta_1)$ and $\script{D}_2= (G_2,\eta_2)$ of $X$, we say that $\script{D}_2$ \emph{extends}\index{extension (of decomposition)|textbf} $\script{D}_1$, in symbols $\script{D}_1 \preccurlyeq \script{D}_2$\index{$\preccurlyeq$|see {extension (of decomposition)} \textbf}, if there is a sequence of edges $e_1,\ldots,e_k \in E(G_2)$ such that $\script{D}_1 = \script{D}_2 / \sequence{e_1,\ldots,e_k}$.
\end{defn}

In particular, $\script{D}_1 \preccurlyeq \script{D}_2$ implies that $G_1 \preccurlyeq  G_2$, and conversely, every contraction minor $G_2 / \sequence{e_1,\ldots,e_k}$ gives rise to a corresponding Eulerian decomposition which is extended by $G_2$. For illustration, consider the following decompositions of the hyperbolic tree $X$.
\begin{figure}[h!]
\begin{tikzpicture}[scale=2]



\node[draw, red, circle,scale=.4, fill=red, label={[label distance=-.1]135:$\textcolor{red}{\delta_1}$}](D1) at (135:1) {};
\node[draw, red, circle,scale=.4, fill=red, label={[label distance=-.1]-45:$\textcolor{red}{\delta_2}$}](D2) at (-45:1) {};

\node[draw, red, circle,scale=.4, fill=red, label={[label distance=-.1]0:$\textcolor{red}{\delta_3}$}](D3) at (15:1) {};
\node[draw, red, circle,scale=.4, fill=red, label={[label distance=-.1]80:$\textcolor{red}{\delta_4}$}](D4) at (75:1) {};
\node[draw, red, circle,scale=.4, fill=red, label={[label distance=-.1]180:$\textcolor{red}{\delta_5}$}](D5) at (195:1) {};
\node[draw, red, circle,scale=.4, fill=red, label={[label distance=-.1]265:$\textcolor{red}{\delta_6}$}](D6) at (255:1) {};

\clip (0,0) circle (1);
\draw[fill=red!20!white,draw=red!20!white] (1,0) ellipse (0.6 and .8); 
\draw[fill=red!20!white,draw=red!20!white] (-1,0) ellipse (0.6 and .8);
\draw[fill=red!20!white,draw=red!20!white] (0,1) ellipse (0.8 and .6); 
\draw[fill=red!20!white,draw=red!20!white] (0,-1) ellipse (0.8 and .6);

\foreach \x in {0,30,...,360}{
\draw[fill=red!40!white,draw=red!40!white, rotate around={\x:(\x:1)}] (\x:1) ellipse (0.2 and .26); 
}

\draw (0,0) circle (1);
\clip (0,0) circle (1);

\hgline{90}{270}
\hgline{0}{180}

\foreach \x in {0,90,...,360}{
\hgline{\x-30}{\x+30}
}

\foreach \x in {0,30,...,360}{
\hgline{\x-10}{\x+10}
}

\foreach \x in {0,10,...,360}{
\hgline{\x-3.33}{\x+3.33}
}

\foreach \x in {0,10,...,1080}{
\pgfmathsetmacro\y{\x/3}
\hgline{\y-1.111}{\y+1.111}
}

\node[draw, circle,scale=.2, fill=black](C) at (0,0) {};

\node[draw, red, circle,scale=.4, fill=red](D11) at (135:1) {};
\node[draw, red, circle,scale=.4, fill=red](D22) at (-45:1) {};

\node[draw, red, circle,scale=.4, fill=red](D1) at (15:1) {};
\node[draw, red, circle,scale=.4, fill=red](D2) at (75:1) {};
\node[draw, red, circle,scale=.4, fill=red](D1) at (195:1) {};
\node[draw, red, circle,scale=.4, fill=red](D2) at (255:1) {};

\end{tikzpicture}%
\quad
\begin{tikzpicture}[scale=1.2]

\node[draw, circle,scale=.2, fill=black](C) at (0,0) {};

\node[draw, rectangle,scale=.6, fill=red!20!white]  (N1) at (0,1.5) {};
 \node[draw, rectangle,scale=.6, fill=red!20!white] (N2) at (-1.5,.0) {};
 \node[draw, rectangle,scale=.6, fill=red!20!white] (N3) at (0,-1.5) {};
 \node[draw, rectangle,scale=.6, fill=red!20!white] (N4) at (1.5,0) {};
 
 \node[rectangle,scale=.6, fill=white]  (N666) at (0,-2) {};
 
 \node[rectangle,scale=.6, fill=white]  (N666) at (-2,0) {};

 \draw[blue] (C) -- (N1);
 \draw[blue] (C) -- (N2);
 \draw[blue] (C) -- (N3);
 \draw[blue] (C) -- (N4);
 
 \node () at (-1,-1.68) {$G_1$};

  \draw[red]    (N1) to[out=180,in=90] (N2);
    \node[label={[label distance=-.1]135:$\textcolor{red}{d_1}$}](T1) at (135:1.3) {};
   \draw[red]    (N3) to[out=0,in=270] (N4);
   \node[label={[label distance=-.1]-45:$\textcolor{red}{d_2}$}](T2) at (-45:1.3) {};

   \draw[draw=white] (0,-1.3) circle (.85); 
   
\end{tikzpicture}%
\quad
\begin{tikzpicture}[scale=1]

\draw[dotted, draw=black] (1.75,0) circle (.85); 
\draw[dotted, draw=black] (-1.75,0) circle (.85); 
\draw[dotted, draw=black] (0,1.75) circle (.85); 
\draw[dotted, draw=black] (0,-1.75) circle (.85); 

\node[draw, circle,scale=.2, fill=black](C) at (0,0) {};

\node[draw, circle,scale=.2, fill=black]  (N1) at (0,1) {};
 \node[draw, circle,scale=.2, fill=black] (N2) at (-1,.0) {};
 \node[draw, circle,scale=.2, fill=black] (N3) at (0,-1) {};
 \node[draw, circle,scale=.2, fill=black](N4) at (1,0) {};
 
 \foreach \x in {1,2,3,4}{
 \node[draw, rectangle,scale=.6, fill=red!50!white]  (N\x1) at (90*\x-15:2) {};
 \node[draw, rectangle,scale=.6, fill=red!50!white] (N\x2) at (90*\x:2) {};
 \node[draw, rectangle,scale=.6, fill=red!50!white] (N\x3) at (90*\x+15:2) {};
 
  \draw[blue] (N\x) -- (N\x1);
  \draw[blue] (N\x) -- (N\x2);
  \draw[blue] (N\x) -- (N\x3);
 }

 \draw[blue] (C) -- (N1);
 \draw[blue] (C) -- (N2);
 \draw[blue] (C) -- (N3);
 \draw[blue] (C) -- (N4);

  \draw[red]    (N13) to[out=180,in=90] (N21);
   \draw[red]    (N33) to[out=0,in=270] (N41);
   
 \node () at (-1.5,-2) {$G_2$};

  \draw[red, thick]    (N11) -- (N12);
    \draw[red, thick]    (N22) -- (N23);
      \draw[red, thick]    (N31) -- (N32);
        \draw[red, thick]    (N42) -- (N43);

    \node[label={[label distance=-.1]135:$\textcolor{red}{d_1}$}](T1) at (135:1.8) {};
   \node[label={[label distance=-.1]-45:$\textcolor{red}{d_2}$}](T2) at (-45:1.8) {};
   
   \node[label={[label distance=-.1]5:$\textcolor{red}{d_3}$}](D3) at (0:1.8) {};
\node[label={[label distance=-.1]85:$\textcolor{red}{d_4}$}](D4) at (90:1.8) {};
\node[label={[label distance=-.1]185:$\textcolor{red}{d_5}$}](D5) at (180:1.8) {};
\node[label={[label distance=-.1]265:$\textcolor{red}{d_6}$}](D6) at (270:1.8) {};

\end{tikzpicture}
\caption{Eulerian decompositions $\p{G_1,\eta_1} \preccurlyeq \p{G_2,\eta_2}$ with dummy edges satisfying $\eta_1(d_i) = \delta_i$ for $i \in [2]$ and $\eta_2(d_i) = \delta_i$ for $i \in [6]$. Note that $G_1 \preccurlyeq G_2$ by contracting all edges inside the dotted subgraphs of $G_2$.}
\label{figure_extendingEulerdecomp}
\end{figure}
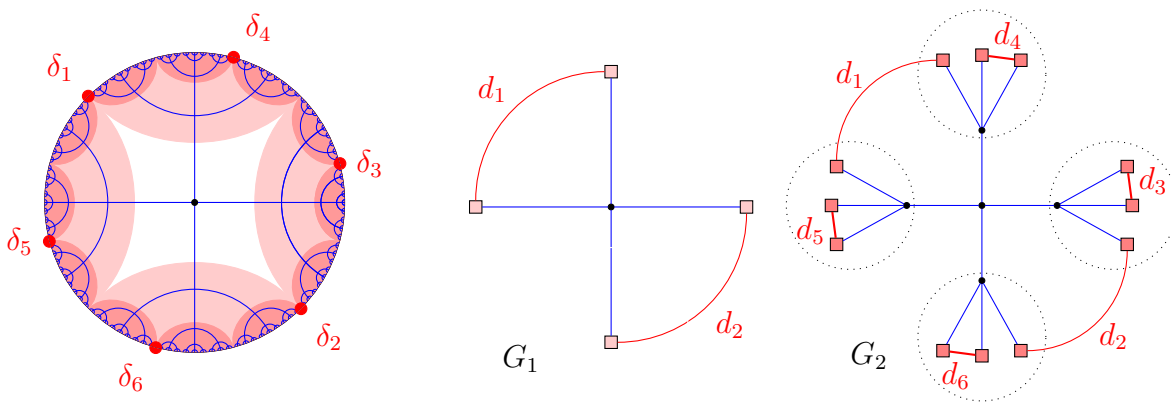

\begin{defn}
\label{def:approximating}
A sequence of \textnormal{[}Eulerian\textnormal{]} decompositions $\Sequence{\script{D}_n}:{n \in \N}$ for a Peano continuum $X$ is called an \emph{approximating sequence of \textnormal{[}Eulerian\textnormal{]} decompositions}\index{approximating sequence (of Eulerian decompositions)|textbf} for $X$, if
\begin{enumerate}[label=(A\arabic*)]
	\item\label{A1} $\script{D}_n \preccurlyeq \script{D}_{n+1}$ for all $n \in \N$, and
    \item\label{A2} $w(\script{D}_n) \to 0$ as $n \to \infty$.
\end{enumerate}
\end{defn}

\subsection{From Eulerian maps to Eulerian decompositions}

One motivation behind the definition of an Eulerian decomposition is they can be generated from every (edge-wise) Eulerian map $g \colon S^1 \to X$. In fact, any such map yields a surprising simple approximating sequence as follows:

\begin{theorem}
\label{thm_Euleriangivesstrongdecomposition}
Every edge-wise Eulerian space admits an approximating sequence \linebreak $\Sequence{\p{G_n,\eta_n}}:{n \in \N}$ of Eulerian decompositions, where each $G_n$ is a cycle of length $n$.
\end{theorem}

\begin{proof}
Suppose that $g \colon S^1 \to X$ is an edge-wise Eulerian map. Then the preimages $I_e := g^{-1}(e) \subset S^1$ for edges $e \in E(X)$ form a collection of disjoint open intervals on $S^1$. Let $E(X)=\set{e_j}:{j \in J}$ for some (possibly finite) $J \subset \N$ be an enumeration of the edge set of $X$, and let $\Delta= \Set{\delta_1,\delta_2,\ldots}$ be a countable dense subset of $S^1 \setminus \closure{\bigcup \set{I_{e}}:{e \in E(X)}}$. Set $E_n = \set{e_i}:{i \in [n]}$ and $\Delta_n = \set{\delta_i}:{i \in [n]}$ (if $\Delta$ is empty, $\Delta_n$ is empty, too). 

For $n \in \N$, let $C_n = \Set{J^n_1,\ldots,J^n_{k_n}}$ denote the set of connected components of $S^1 \setminus \p{ \Delta_n \cup \bigcup \set{I_e}:{e \in E_n}}$. Let $V_n = \set{v_J}:{J \in C_n }$, $F_n=\set{f_e}:{e \in E_n}$ and $D_n= \set{d_\delta}:{\delta \in \Delta_n}$ be duplicate sets of $C_n$, $E_n$ and $\Delta_n$ respectively. In our Eulerian decomposition $\p{G_n,\eta_n}$, the graph $G_n$ will be a cycle with vertex set $V_n$ and edge set $E(G_n)=F_n \sqcup D_n$. 

 Define $\eta_n(v_J) := g(\closure{J})$ for each $v \in V_n$. By construction, $\eta_n(v)$ is a standard Peano subcontinuum of $X$, giving \ref{Eta1}. Set $\eta_n(f_e) := e$ and $\eta_n(d_\delta) := \delta$ for \ref{Eta2}--\ref{E1b}. Since every interval in $\set{I_e}:{e \in E_n}$ and every point in $\Delta_n$ is incident with the closure of precisely two components of $C_n$, transferring this assignment to $G_n$ satisfies \ref{E2} and \ref{E3} (formally, if $\closure{I_e} \cap J \neq\emptyset$ we put $f_e \sim v_J$, and similarly, if $\delta \in \closure{J}$, put $d_\delta \sim v_J$). Hence, all properties of Definition~\ref{def:Eulerdecomp} are satisfied, and so $(G_n,\eta_n)$ is an Eulerian decomposition of $X$.

To see that $\Sequence{\p{G_n,\eta_n}}:{n \in \N}$ is an approximating sequence, note that for \ref{A1}, it is easily verified that $\p{G_{n+1},\eta_{n+1}} / \sequence{e_{n+1},d_{n+1}} = \p{G_n,\eta_n}$. For \ref{A2}, note that by our density assumption on $\Delta$, it follows that $\mesh{V_n} \to 0$. By elementary topological arguments, this implies that also $\mesh{\set{\eta_n(v)}:{v \in V_n}} \to 0$, i.e.\ $w(G_n,\eta_n) \to 0$.
\end{proof}

\subsection{A link between even-cut property and Eulerian decompositions}

Our second motivation for Eulerian decompositions is that by permitting the model graph $G$ to be Eulerian, and not necessarily only a cycle, such decompositions can be built assuming just the even-cut condition, as demonstrated by the following observation which forms the blueprint for the more intricate constructions in the later chapters. 


For the construction, we recall the following notion:

\begin{defn}[Intersection graph]
\label{def_intersectiongraph}
For $\script{U}$ a family of subsets of $X$, the associated \emph{intersection graph}\index{intersection graph|textbf} $G_\script{U}$ is  the graph with vertex set $\script{U}$, and an edge $UV$ for $U \neq V \in \script{U}$ whenever $U \cap V \neq \emptyset$.
\end{defn}

If $\script{U}$ is a finite cover of a Peano continuum $X$, it follows from the connectedness of $X$ that $G_\script{U}$ is a finite connected graph.\footnote{For a cover $\script{U}$, the intersection graph $G_\script{U}$ is sometimes also called the \emph{nerve} of the cover.} 

\begin{blueprint}
\label{obs_blueprint}
Suppose $X$ satisfies the even-cut condition. Then any Peano partition of $X$ into standard subspaces gives rise to an Eulerian decomposition for some suitable choice of dummy edges.
\end{blueprint}

\begin{proof}
Let $\script{U}$ be a (finite) Peano partition of $X$ into standard subspaces. Let $F \subset \script{U}$ denote the collection of standard subspaces consisting of a single edge, and put $V = \p{\script{U} \setminus F} \cup S$ where $S$ is the finite collection of isolated points of $X - F$.

Now let $G'$ be any graph with vertex set $V$ and edge set $F$ satisfying \ref{E1a}, \ref{E1b} and \ref{E2} of Definition~\ref{def:Eulerdecomp}. Our task is to add some new dummy edges $D$ to $G'$ to form a supergraph $G$ that will be the desired Eulerian decomposition satisfying \ref{E3}.

Towards this, consider the auxiliary graph $H=(V,E_H)$ given by the intersection graph $G_V$ on $V$ associated with the cover $V$ of $X - F$. 
We shall prove that we can find a multi-subset $D \subset E_H$ as desired.

As a first step, we claim that for each component $C$ of $H$, the number of odd-degree vertices of $G'$ in $C$ is even. To see the claim, note first that $X - F$ has finitely many connected components, Lemma~\ref{lem_removing edges}, and for every component $C$ of $H$, the underlying subset $\bigcup C$ is a connected component of $X - F$ by \ref{Eta1}. Thus, the bipartition $(C,D)$ with $D= V - C$ of $V=V(H)=V(G')$ induces a bipartition of $\ground{X}$, and hence an edge cut $B:=E(\bigcup C, \bigcup D) \subset F$ of $X$, which must be even by assumption. However, property~\ref{E2} of $G'$ implies that $E(C,D) = B$ is also an edge cut of $G'$ containing the same edges. In particular, the quotient graph $G'_C$ of $G'$ where we collapse $D$ to a single vertex $v_D$ has the property that $v_D$ has even degree, as $v_D$ is adjacent precisely to the evenly many edges in $B$, plus possibly some loops (which do not affect the parity of the vertex degree). By the \emph{Handshaking Lemma}, the number of odd-degree vertices in $G'_C$ is even. Since $v_D$ has even degree, it follows that the number of odd-degree vertices of $G'_C$ in $C$ (and hence also of $G'$ in $C$) is even, and thus the claim follows.  

Hence, we may pair up the odd-degree vertices of $G'$ such that pairs lie in the same component of $H$. For each such pair $\Set{u,v}$, consider a $u-v$ path in $H$. By taking the mod-2 sum over the edge sets of all these paths, we obtain an edge set $D_1 \subset E_H$ such that by adding $D_1$ to $G'$, one obtains an even graph $G''$. \index{even graph}

Since the intersection graph $H$ is connected, we may find an edge set $D_2 \subset E_H$ such that adding $D_2$ to $G''$ results in a connected graph. Then define $G := G '' \cup 2 \cdot D_2$, i.e.\ for every edge in $D_2$ we add two parallel dummy edges to $G$, in order to ensure connectedness without affecting the degree parity conditions.

 Finally, to make sure that property~\ref{E3} of Definition~\ref{def:Eulerdecomp} is satisfied, note that by definition of the intersection graph $H$, for every $d =xy \in E_H$, the sets $x,y \in V$ intersect, and hence we may choose a point (i.e.\ a trivial arc) $\eta(d)$ contained in $x \cap y \subset X$, satisfying property~\ref{E3} as required.
\end{proof}

%
%
%
%
%
%
%
%
%

\section{Obtaining an Edge-Wise Eulerian Map}

\subsection{Translating combinatorial information to topolopy}


\index{edge-wise Eulerian map|(} For the benefit of clarity, and because we will need to jump between combinatorial and topological graphs, we denote for a combinatorial multi-graph $G$ by $|G|$ the underlying topological space. Recall that for an edge $e$ of a finite multi-graph or a Peano continuum, we write $e(0)$ and $e(1)$ for the two end vertices of $e$, and $e(x)$ for $x \in (0,1)$ for the corresponding interior point on $e$.

\begin{defn}[Usc function, covering function]
For a topological space $X$ let $2^X = \set{A \subset X}:{A \text{ nonempty, closed}}$. A function $g \colon Y \to 2^X$ is \emph{upper semi-continuous}\index{upper semi-continuous function|textbf} (usc\index{usc|see {upper semi-continuous function} \textbf}) if for all $y \in Y$ and all open sets $U \supset f(y)$ there is an open neighbourhood $V$ of $y$ such that $\bigcup_{y' \in V} g(y') \subset U$. The function $g$ is said to \emph{cover}\index{covering function|textbf} $X$ if $X  = \bigcup \set{g(y)}:{y \in Y}$.
\end{defn}

\begin{lemma}
\label{lem_uscfunction}
Suppose $(G,\eta)$ is an Eulerian decomposition of some Peano continuum $X$. Then the map $\hat{\eta} \colon |G| \to 2^{X}$ given by 
\begin{itemize}
\item $\hat{\eta}(v) := \eta(v)$ for all $v \in V$, and
\item $\hat{\eta}(e(y)) : = \Set{\p{\eta(e)}(y)}$ for all $e \in E(G)$ and $y \in (0,1)$
\end{itemize}
defined on the 1-complex $|G|$ of $G$ is upper semi-continuous, covers $X$, and is injective and acts as identity for points on real edges.\footnote{Interior points of a dummy edge $d$ for which $\eta(d)$ is trivial are mapped constantly to that singleton.} Moreover, $\diam{\hat{\eta}(y)} \leq w(G,\eta)$ for all $y \in |G|$.
\end{lemma}

\begin{proof}
First, it is immediate from property \ref{E1a} that $\hat{\eta}$ covers $X$. Next, the usc-condition for $\hat{\eta}$ is evidently satisfied for interior points on edges of $G$. So consider a vertex $v \in G$ and an open set $U \subset X$ with $P=\eta(v) \subset U$. To simplify notation, let us write $f_X := \eta(f)$ for every edge $f \in F$, and similarly $d_X := \eta(d)$ for every edge $d \in D$. 

By \ref{E2}, every edge $f \in F$ incident with $v$ in $G$, say $f(j) = v$, satisfies that $f_X(j) \in \eta(v)$, and hence $\closure{f_X} \cap U$ is an open neighbourhood of $f_X(j) \in \closure{f_X} \subset X$. Since $\hat{\eta}$ acts as the identity between $f$ and $f_X$, there is an open neighbourhood $V_f$ of $v$ in $\closure{f}$ such that $\bigcup_{y' \in V_f} \hat{\eta}(y') = \closure{f_X} \cap U$. By \ref{E3}, we similarly obtain an open set $V_d$ for every $d \in D$. Together, this yields that
$$V = \singleton{v} \cup \bigcup \set{V_f}:{f \in F, \, f \sim v} \cup \bigcup \set{V_d}:{d \in D, \,d \sim v}$$
is an open neighbourhood in $|G|$ of the vertex $v$ satisfying that $\bigcup_{x' \in V} \hat{\eta}(x') \subset U$, which establishes that $\hat{\eta}$ is upper semi-continuous.

That $\hat{\eta}$ is injective and acts as identity for points on real edges follows from \ref{E1b}. Finally, that $\diam{\hat{\eta}(y)} \leq w(G,\eta)$ for all $y \in |G|$ is clear from construction.
\end{proof}

Lastly, we record how the usc-maps corresponding to two comparable Eulerian decompositions relate to each other:

\begin{lemma}
\label{lem_uscfunction2}
Let $X$ be a Peano continuum. For two Eulerian decompositions $\script{D}_1= (G_1,\eta_1)$ and $\script{D}_2= (G_2,\eta_2)$ of $X$ with $\script{D}_1 \preccurlyeq \script{D}_2$, let $\varrho \colon |G_2| \to |G_1|$ denote the edge-contraction map corresponding to $G_1 \preccurlyeq G_2$. Then the associated usc-maps $\hat{\eta}_1$ and $\hat{\eta}_2$ satisfy
$ \hat{\eta}_2(y) \subset \hat{\eta}_1(\varrho(y))$ for all $y \in |G_2|.$
\end{lemma}

\begin{proof}
It suffices to prove the lemma in the case where we contract a single edge, say $\script{D}_1 = \script{D}_2 / e$ with $e=ab$. In this case, 
$$\varrho \colon |G_2| \to |G_1|, \; z \mapsto \begin{cases}
z & \text{for all } z \in |G_2| \setminus \closure{e}, \; \text{ and} \\
v_e & \text{for all } z \in \closure{e} = \Set{a} \cup e \cup \Set{b}.
\end{cases}
$$
Also, according to Lemma~\ref{lem:EulerdecompContraction}, we have $G_1 = G_2 / e$ and $\eta_1$ is given by
\begin{itemize}
\item $\eta_1 (v_e) = \eta_2(a) \cup \eta_2(e) \cup \eta_2(b)$,
\item $\eta_1 (v) = \eta_2(v)$ for all $v \neq v_e$, and
\item $\eta_1 (f) = \eta_2(f)$ for all $f \in E(G_2) \setminus  \Set{e}$.
\end{itemize}
To verify the assertion of the lemma, consider some $z \in |G_2|$. If $z$ is an interior point of some edge $f \neq e$, then it follows from the statement in the third bullet point that $\hat{\eta}_1(\varrho(z)) = \hat{\eta}_1(z) = \hat{\eta}_2(z)$. Similarly, if $z$ is a vertex other than $a$ or $b$, then it follows from the second bullet point that $\hat{\eta}_1(\varrho(z)) = \hat{\eta}_1(z) = \hat{\eta}_2(z)$. Finally, if $z$ is an end vertex or interior point of $e$, then it follows from the first bullet point that $\hat{\eta}_1(\varrho(z)) = \hat{\eta}_1(v_e) = \eta_2(a) \cup \eta_2(e) \cup \eta_2(b) \supseteq \hat{\eta}_2(z)$. 
\end{proof}

\subsection{Construction of edge-wise Eulerian maps} We now prove our main theorem of this chapter that every approximating sequence of Eulerian decompositions gives rise to an edge-wise Eulerian map,  completing the proof of $\ref{romaniii} \Rightarrow \ref{romanii}$. 
\begin{theorem}[Mapping Theorem]
\label{thm_WeaklyEulerianMappingThm2}
Any Peano continuum $X$ admitting an approximating sequence of Eulerian decompositions is edge-wise Eulerian.
\end{theorem}

\begin{proof}
Let $\Sequence{\script{D}_n}:{n \in \N}$ with $\script{D}_n=(G_n,\eta_n)$ be an approximating sequence of Eulerian decompositions for $X$, each $G_n$ with edge bipartition $E_n = F_n \sqcup D_n$ into real and dummy edges. Note that by property~\ref{A1} and Definition~\ref{def_extendingEulerDecomp}, we have $G_n$ is a contraction minor of $G_{n+1}$ for all $n \in \N$, and hence the sequence $\Sequence{G_n}:{n \in \N}$ forms an inverse system of finite Eulerian multi-graphs under contraction bonding maps. Hence, the inverse limit  $\Gamma = \varprojlim G_n$ 
 is an Eulerian graph-like continuum, see \cite[Thm.~13, Prop.~17]{euleriangraphlike}. Write $F = \bigcup F_n$ and $D = \bigcup D_n$. Then $E(\Gamma) = F \sqcup D$. Note that there is a natural bijection between $F$ and $E(X)$ via $\eta(f) := \eta_n(f)$ if $f \in F_n$, which is well defined by property \ref{C3}. Further, it is readily checked that \ref{A2} and \ref{E1a} imply that $\eta$ is onto, while \ref{E1b} implies that $\eta$ is injective.

We now construct a continuous surjection $\hat{\eta} \colon |\Gamma| \to X$ such that $\hat{\eta}$ is injective for interior points on $f \in F$ and $\hat{\eta} \restriction f \colon f \to \eta(f)$ is a homeomorphism for interior points on $f \in F \subset E(\Gamma)$ to its associated edge $\eta(f) \in E(X)$ for all $f \in F$. 
For the construction of $\hat{\eta}$, consider first for each $n \in \N$ the function
$$q_n \colon |\Gamma| \to 2^{X}, \;  z = \Sequence{z_i}:{i \in \N} \mapsto \hat{\eta}_n(z_n),$$
which, by Lemma~\ref{lem_uscfunction}, is upper semi-continuous, covering, and is injective and acts as identity for points on edges $f \in F$.
Moreover, Lemma~\ref{lem_uscfunction2} shows that 
\begin{equation}
\tag{$\ddagger$}
\label{eqdagger}
q_{n+1}(z) \subseteq q_{n}(z)
\end{equation}
for all $n \in \N$ and $x \in |\Gamma|$. Thus, $\bigcap_{n \in \N} q_n(z) \subset X$ is a nested intersection of non-empty closed subsets of $X$, and so it follows from compactness of $X$ that this intersection is non-empty. At the same time, however, we have $\diam{q_{n}(z)} \leq w(G_n,\eta_n) \to 0$ by Lemma~\ref{lem_uscfunction} and \ref{A2}, and so this intersection must be a singleton for each $z \in |\Gamma|$. Hence, there is a function 
$$\hat{\eta} \colon |\Gamma| \to X \; \text{ defined by } \;  \Set{\hat{\eta}(z)} = \bigcap_{n \in \N} q_n(z) \; \text{ for all } \, z \in |\Gamma|.$$
As the image of each $q_n$ is an upper semi-continuous function that covers $X$ and satisfies (\ref{eqdagger}), it follows from \cite[General Mapping Theorem 7.4]{Nadler} that the map 
$\hat{\eta} \colon |\Gamma| \to X$ is a continuous surjection as desired. Further, it is clear by the definition of $\hat{\eta}$ that for every real edge $f \in F$ we have $\hat{\eta}^{-1}(\eta(f)) = f$ and $\hat{\eta} \restriction f$ acts a identity from $f \in F$ onto $\eta(f) \in E(X)$.

In order to complete the proof, note that since $\Gamma$ is an Eulerian graph-like continuum, there is an Eulerian map $h \colon S^1 \to |\Gamma|$. In particular, $h$ is a continuous surjection with the property that for every open edge $f \in E(\Gamma)$ (dummy and real edges alike) we have $I_f:=h^{-1}(f)$ is an interval on $S^1$ and $h \restriction I_f$ is a homeomorphism from $I_f$ onto $f$. 

We now claim that $g = \hat{\eta} \circ h \colon S^{1} \to X$ is the desired edge-wise Eulerian map. Clearly, as the composition of surjective functions, $g$ is itself a surjection from the circle onto $X$. To see that $g$ is edge-wise Eulerian, we need to check that $g$ sweeps through each edge of $X$ precisely once. So let $e \in E(X)$ be arbitrary. By our considerations above, there is a unique $f \in F$ with $\eta(f) =e$. But $g^{-1}(e) = h^{-1} \circ \hat{\eta}^{-1}(e) = I_f$. Since $h_f=h \restriction I_f$ is a homeomorphism from $I_f$ onto $f$, and  $\hat{\eta}_f = \hat{\eta} \restriction f$ acts as identity between interior points of $f$ and $e$, it follows that $g \restriction I_f$ is as the composition of the homeomorphisms $\hat{\eta}_f \circ h_f$ itself a homeomorphism from $I_f$ onto $\eta(f)=e$. Thus, we have verified that $g$ is an edge-wise Eulerian map, and hence that $X$ is edge-wise Eulerian. \index{edge-wise Eulerian map|)}
\end{proof}

\section{Simplicial Maps}

In this last section on Eulerian decompositions, we describe an equivalent condition to Definition~\ref{def_extendingEulerDecomp} about compatible Eulerian decompositions, which lends itself better to the constructions in the next two chapters. 
    
    \begin{defn}[Contraction map, edge-contraction map]
    \label{defn_edgecontraction} We call a surjective map $ \varrho \colon G_{2} \to G_1$ between two graphs $G_{i} = (V_i,E_i)$ a \emph{contraction map}\index{contraction map|textbf} if
\begin{enumerate}[label=(Q\arabic*)]
	\item\label{Q1} $\varrho(V_2) = V_1$,
	
	\item\label{Q2} $\varrho$ restricts to a bijection between $E_2 \setminus \varrho^{-1}(V_1)$ and $E_1$,
    
    \item\label{Q3} $\varrho(e(j)) = (\varrho(e))(j)$ for all $e \in E_2 \setminus \varrho^{-1}(V_1)$ and $j \in \Set{0,1}$, and
    
    \item\label{Q4} $\varrho(e(j)) = \varrho(e)$ for all $e \in E_2 \cap \varrho^{-1}(V_1)$ and $j \in \Set{0,1}$.
\end{enumerate}
If additionally,
\begin{enumerate}[label=(Q\arabic*), resume]
\item\label{Q5} $\varrho^{-1}(v)$ is a connected subgraph of $G_2$ for all $v \in V(G_1)$,
\end{enumerate}
then the map $\varrho$ is called an \emph{edge-contraction map}.\index{edge-contraction map|textbf}
\end{defn}

Thus, an edge-contraction map $ \varrho \colon G_{2} \to G_1$ is precisely a map witnessing that $G_1 \preccurlyeq G_2$, whereas a contraction map may identify vertices that are not necessarily connected by an edge.
 
\begin{defn}
\label{defn_etacompatible}
Let $\script{D}_1 = \p{G_1,\eta_1}$ and $\script{D}_2 = \p{G_2,\eta_2}$ be decompositions of a Peano continuum $X$. A contraction map $\varrho \colon G_2 \to G_1$ is called \emph{$\eta$-compatible}\index{eta-compatible@$\eta$-compatible (contraction)|textbf} if 
$$\eta_1(x) = \bigcup \set{\eta_2(y)}:{y \in \varrho^{-1}(x)}$$
for all $x \in V(G_1) \cup E(G_1)$.
\end{defn}

\begin{lemma}
\label{lem_equivalenceExtend}
Suppose $\script{D}_1 = \p{G_1,\eta_1}$ and $\script{D}_2 = \p{G_2,\eta_2}$ are both decompositions of a Peano continuum $X$. Then $\script{D}_1 \preccurlyeq \script{D}_2$ if and only if there is an $\eta$-compatible edge-contraction map $\varrho \colon G_2 \to G_1$.
\end{lemma}

\begin{proof}
This follows from the observation that $G_1 \cong G_2 / \sequence{e_1,\ldots,e_k}$ if and only if there is an edge contraction map $\varrho \colon G_2 \to G_1$ such that $\varrho^{-1}(V_1) = \Set{e_1,\ldots,e_k}$.
%
\end{proof}

\index{decomposition (of Peano continuum)|)}
\index{Eulerian decomposition|)}



\chapter{Product-Structured Ground Spaces}
\label{chapter_ProductRemainders}

\section{Introduction}

In this chapter we establish that the Eulerianity conjecture holds for Peano continua with {\emph{product-structured ground space}\index{product-structured ground space|textbf}\index{ground space!product-structured|textbf}, i.e.\ Peano continua} $X$ whose ground space 
$\ground{X} = V \times P$ is the product of a (compact) zero-dimensional space $V$ with a Peano continuum $P$, thereby proving the second case \ref{THMB} of our main result Theorem~\ref{thm_crossingfinitearcs} stated in the introduction.

\begin{theorem}
\label{thm_ProductSpaceRemainders}
{A Peano continuum with product-structured ground space} is Eulerian if and only if it satisfies the even-cut condition. 
\end{theorem}

Bula, Nikiel and Tymchatyn have asked whether the Eulerianity Conjecture holds for spaces with ground set $C \times K$, where $C$ is the Cantor set and $K$ is any continuum (not necessarily Peano), \cite[Problem~3]{Koenigsberg}. For this question, our Theorem~\ref{thm_ProductSpaceRemainders} gives a strong answer in the case where $P = K$ is a Peano continuum. For our result, the assumption that $P$ is Peano is crucial. To demonstrate this, recall that Bula, Nikiel and Tymchatyn have also asked whether a Peano continuum $X$ with ground space a continuum (not necessarily Peano) satisfies the Eulerian conjecture \cite[Problem~2]{Koenigsberg}. We believe that this question is, maybe unexpectedly so, at least as hard as the situation discussed in Theorem~\ref{thm_ProductSpaceRemainders}: indeed, with the techniques from this chapter one can establish the Eulerianity conjecture for spaces $X$ with ground space a Cantor fan, or even a generalised fan of the form $\ground{X} =\p{ V \times P} / \set{(v,p)}:{v \in V}$ for some $p \in P$.

\subsection{Proof strategy} {Consider a Peano continuum $X$ with product-structured ground space $\ground{X} = V \times P$. By the reduction results, it suffices to prove Theorem~\ref{thm_ProductSpaceRemainders} for loopless Peano graphs. By Theorem~\ref{thm_MainEquivalence} $\ref{romaniii} \Rightarrow \ref{romani}$} we need to construct an approximating sequence of Eulerian decompositions for $X$. The first ingredient to construct this approximation is the observation that every Peano graph $X$ with ground space $\ground{X} = V \times P$ exhibits a fractal-like behaviour as follows: for every point $(v,p) \in V \times P$ and every $\varepsilon>0$ there exists $ V' \times P' \subset V \times P$ {of diameter at most $\varepsilon$} such that $v \in V' \subset V$ is clopen, $p \in \interior{P'} \subset P' \subset P$ and $P'$ is a regular subcontinuum of $P$, and $X':=X[V' \times P']$ is again a Peano graph of the same form as in the theorem, see Lemma~\ref{lem_makeW-midconnected}. Let us call such a space $X'$ a \emph{tile} of $X$. 
Utilising this fractal-like behaviour, our main technical result in this chapter is the so-called \emph{decomposition theorem,} Theorem~\ref{thm_decompositionforProductRemainders}, which says roughly that any Peano continuum with product-structured ground space can be decomposed into edge-disjoint tiles all of arbitrarily small diameter plus some finitely many cross edges that go between tiles, such that most of the tiles now satisfy the even-cut condition. 

Crucially, to control all edge cuts simultaneously, we borrow and extend in Section~\ref{s:fundamentalcycles} the techniques of topological spanning trees, fundamental circuits and infinite thin sums from the recently developed infinite graph and infinite matroid theory, see \cite[\S 8.7]{Diestel} and \cite{infinitematroids,Matroids}.

In the final section of this chapter, Section~\ref{sec_approxforprodcutground}, we then demonstrate how this decomposition theorem can be used, now using the assumption that the original space $X$ satisfied the even-cut condition for the first time, to construct an approximating sequence of Eulerian decompositions for $X$.

\section{Spanning Trees and the Even-Cut Condition}
\label{s:fundamentalcycles}

{From here on, a \emph{graph-like continuum} is an object $X=(V,E)$ where $X$ is a (metrizable) continuum and $V \subseteq X$ is a closed zero-dimensional subset such that $X \setminus V = \bigoplus_{e \in E} (0,1)$ is a topological sum of intervals. See \cite{euleriangraphlike} for additional background information. This describes the same class of spaces as our earlier definition of graph-like which required  $\ground{X}$ to be zero-dimensional -- however, this rather more combinatorial definition allows vertices of $V$ to subdivide free arcs of $X$, i.e.\ the inclusion in $\ground{X} \subseteq V$ might be proper. For example, both $V=\Set{0,1}$ and $V$ the middle third Cantor set can function as vertex set of a graph-like continuum homeomorphic to the unit interval $I$. We also refer the reader to a more in-depth discussion of graph-like spaces in Section~\ref{sec_52}. }

Before we embark on our proof, we need some preliminary results about \emph{spanning trees} in graph-like continua.\index{graph-like space|(} These notions are by now standard in the theory of infinite graphs (see e.g.\ \cite[\S 8]{Diestel} and \cite{DSurv}) and they do generalise nicely to graph-like continua. Indeed, this is not by accident and could be seen as a corollary to the general theory of infinite matroids and matroids induced by graph-like spaces, see \cite{infinitematroids,Matroids}. However, as there are direct proofs for the results we need, and so as to make it easier for the reader, we simply state and prove what we need. 


\begin{lemma}
\label{lem_spanningtrees}
The following are equivalent for a standard subspace $T$ of a graph-like continuum $Z$: 
\begin{enumerate}
\item\label{edge-minimal} $T$ is edge-minimally connected,
\item $T$ is uniquely arc-connected,
\item $T$ is connected and does not contain a non-trivial cycle, and
\item $T$ is a dendrite.
\end{enumerate}
\end{lemma}

\begin{proof}
Recall that a graph-like continuum is hereditarily locally connected, so every subcontinuum of $Z$ is automatically Peano \cite[Corollary~8]{euleriangraphlike}. The equivalence of (3) and (4) holds by the definition of \emph{dendrite} (see \cite[10.1]{Nadler}). The equivalence of (2) and (3) is easy. To see that (1) and (3) are equivalent, note that if $T$ contains a cycle, then deleting an edge on that cycle does not disconnect $T$, and conversely, if deleting an edge $e = xy$ does not disconnect $T$, then for any $x-y$ arc $P$ in $T - e$, we have $P \cup e$ is a cycle.
\end{proof}

\begin{defn}[Spanning tree]
A subspace $Y$ of a graph-like continuum $X=(V,E)$ is called \emph{spanning} if $V \subset Y$. A spanning standard subspace $T$ of a graph-like continuum $Z$ is called a \emph{spanning tree}\index{spanning tree|textbf} of $Z$ provided it satisfies one (and therefore every) condition in Lemma~\ref{lem_spanningtrees}. 
\end{defn}

Spanning trees of graph-like continua are easy to construct, because connectivity is preserved under nested intersections---so in order to obtain a standard subspace with property (\ref{edge-minimal}), one only needs to enumerate all edges from a graph-like continuum, and then delete the next edge in line as long as it is not a bridge at that current stage.

\begin{defn}[Fundamental cuts; fundamental cycles]
\label{def:fundcycles}
Let $T$ be a spanning tree of a graph-like continuum $Z$.
\begin{itemize}
\item If $f \in E(T)$ is an edge of $T$, then by Lemma~\ref{lem_removing edges} and property (1) in Lemma~\ref{lem_spanningtrees}, the space $T - f$ has two connected components with vertex sets say $A$ and $B$ which form a clopen partition of $V(T)=V(Z)$. The corresponding edge cut $E(A,B)$ of $Z$ is also called the \emph{fundamental cut}\index{fundamental cut|textbf} of $f$, denoted by $D_f$.
\item If $e \notin E(T)$, then $T$ contains a unique standard arc $A$ between the endpoints of $e$. The \emph{fundamental cycle}\index{fundamental cycle|textbf} $C_e$ is given by the edge set $E(A) \cup \singleton{e}$. Note that $Z[C_e]$ is indeed homeomorphic to $S^1$.
\end{itemize}
\end{defn}

Observe that for $f \in E(T)$ and $e \notin E(T)$ one has $e \in D_f$ if and only if $f \in C_e$.

\begin{defn}[Thin family]
Let $E$ be a set. A multi-set $(C_j \colon j \in J)$ of subsets of $E$ is called \emph{thin}\index{thin family|textbf} if for all $e \in E$, we have $\cardinality{\set{j \in J}:{e \in C_j}} < \infty$.
\end{defn}

\begin{defn}[Thin sum]
\label{def:thinsum}
For a thin family $(C_j \colon j \in J)$, the sum 
$$C=\sum_{j \in J} C_j := \set{e \in E}:{\cardinality{\set{j \in J}:{e \in C_j}} \text{ is odd}}$$
is well-defined. We say that $C$ is the \emph{thin sum}\index{thin sum|textbf} over the $(C_j \colon j \in J)$.
\end{defn}


The following theorem is in some sense a natural generalisation of the corresponding theorem for finite and infinite graphs \cite[Theorems~1.9.5 and 8.7.1]{Diestel} respectively. 

\begin{theorem}
\label{thm_thinsumfundamentalcircuits}
Let $X=(V,E)$ be a graph-like continuum, and $D \subset E$. Then all topological cuts of $X[D]$ are even if and only if $D$ is a thin sum of fundamental cycles of any spanning tree of $X$.
\end{theorem}

\begin{proof}
Compare to \cite[8.7.1]{Diestel}, where this statement is proved for Freudenthal compactifications of locally finite graphs (which form a proper subclass of the class of graph-like continua). For additional background, see \cite{infinitecycles}.

To see that a thin sum of cycles satisfies the even-cut condition, recall that by \cite[Lemma~6]{euleriangraphlike}, any single cycle $C$ intersects any topological cut of $X$ in an even number of edges. This extends immediately to finite symmetric differences, as is easily verified. But then this also extends to thin sums of cycles: since cuts are finite, only finitely many cycles in our thin sum can meet the cut, and so the result follows.

For the converse implication, suppose $X[D]$ satisfies the even-cut condition and fix any spanning tree $T$ of $X$. We show that
 $D = \sum_{e \in D \setminus E(T)} C_e$. 
To see that this sum is well-defined, observe that $f \in C_e$ if and only if $e \in D_f$. Since fundamental cuts are finite, the above is the sum over a thin family. To prove the equality, we claim that {the symmetric difference}
 $D' := D \; {\triangle} \; \sum_{e \in D \setminus E(T)} C_e = \emptyset$. 
First, it is clear that $D' \subset E(T)$, since every edge $e \in D \setminus E(T)$ has been eliminated by the corresponding $C_e$ (and all other edges in $C_e$ lie in $E(T)$ by construction). 

Second, the existence of an edge $f \in D'$ leads to a contradiction as follows: since $f \in D' \subseteq E(T)$, it follows that 
$ f  \in D_f \cap D' \subseteq D_f \cap E(T) = \Set{f}$.

Thus, $D_f$ is a topological cut meeting $D'$ in an odd number of edges. This contradicts the fact that both $D$ (by assumption) and the thin sum $\sum_{e \in D \setminus E(T)} C_e$ (by virtue of the first proven implication) meet every cut in an even number of edges. 
\end{proof}

\section{Sparse Edge Sets}
\label{s:concentrated}

\subsection{Properties of sparse edge sets}

\index{edge set!sparse|(} Given a Peano graph $X$ with ground set $\ground{X} = V \times P$, we will now investigate under which conditions certain (infinite) edge sets can be removed without harming local connectedness or density. 
%
Recall from Section~\ref{edge_cuts_degree} that a subset $F \subset E(X)$ of edges is called \emph{sparse (in $X$)} if $X[F]$ is a graph-like compactum (i.e.\ if $\closure{\bigcup F} \setminus \bigcup F$ is zero-dimensional). 
Note that the property of an edge set $F$ being sparse is inherited by subsets of $F$. 

\begin{lemma}
\label{lem_sparseproperties}
Let $X$ be a Peano continuum \textnormal{[}Peano graph\textnormal{]} $X$ and $F \subset E(X)$ a sparse edge set. Then the following assertions hold.
\begin{enumerate}[label=(\roman*)]
\item\label{lem_removingzerosequences} The non-trivial components of $X-F$ form a zero-sequence of standard Peano continua \textnormal{[}Peano graphs\textnormal{]}.
\item\label{groundspace} If $\ground{X}$ contains no $1$-point components, then $\ground{X-F} = \ground{X}$.
\item\label{lem_uniformlylargegroundspace} \label{lem_removingspanningtrees2} If for some $\delta >0$ all components of $\ground{X}$ have diameter at least $\delta$, then $X-F$ consists of finitely many Peano continua \textnormal{[}Peano graphs\textnormal{]}, so is locally connected.
\end{enumerate}
\end{lemma}

%

\begin{proof}
Let $\script{D}$ denote the collection of components of $X-F$. It is clear that each element of $D$ is a standard subcontinuum. We first show that $\script{D}$ forms a null-family. Otherwise, for some $\varepsilon > 0$ there are infinitely $D_n \in \script{D}$ with $\diam{D_n} \geq \varepsilon$ for all $n \in \N$. By sequential compactness of the hyperspace \cite[4.18]{Nadler}, we may assume that $D_n \to D$, i.e.\ $D_n$ converges to a continuum $D$ in the Hausdorff metric \cite[4.2]{Nadler}. And since $\diam{D_n} \geq \varepsilon$ for all $n \in \N$, we have -- by the properties of the Hausdorff metric -- that $\diam{D} \geq \varepsilon$, too. Moreover, since edges are open, we necessarily have $D \subset \ground{X}$. But now, since $D$ is a non-trivial continuum and $\closure{\bigcup F} \setminus \bigcup F$ is zero-dimensional, there is $x \in D$ and a connected neighbourhood $U$ of $x$ in $X$ with $U \cap X[F] = \emptyset$. However, since $D_n \to D $ there exists $N \in \N$ such that $D_n \cap U \neq \emptyset$ for all $n \geq N$. Therefore, $D \cup U \cup D_N$ is a connected subset of $X - F$, contradicting that $D_N$ was a component. This contradiction establishes that $\script{D}$ forms a null-family, and hence that the subfamily $\script{D}' \subset \script{D}$  of non-trivial elements of $\script{D}$ forms a zero-sequence.

To see that each $D \in \script{D}'$ is a Peano continuum, note that by construction, $D \setminus \closure{F}$ is open, so hence locally connected, and moreover dense in $D$. It follows that the interior of $D$ is locally connected with zero-dimensional boundary (as the boundary is a subset of the zero-dimensional $X[F] \cap \ground{X}$, and so  $D$ must be a Peano continuum, since if a continuum fails to be locally connected at some point, then it fails to be locally connected at all points of a non-trivial subcontinuum, \cite[5.13]{Nadler}.

Finally, if $X$ is a Peano graph, then each $D \in \script{D}'$ is a Peano graph too, i.e.\ has dense edge set. Suppose to the contrary that for some non-trivial component $D$, its edge set $E(D) = \set{e \in E(X)}:{e \subset D}$ is not dense in $D$. Since $\closure{F} \setminus F$ is zero-dimensional,
there is $x \in D $ and a connected open neighbourhood $U$ of $x$ in $X$ with $U \cap \closure{\bigcup \p{E(D) \cup F}} = \emptyset$. Since by assumption $E(X)$ is dense in $X$ and forms a zero-sequence by Lemma~\ref{lem_removing edges}, there is an edge $e \in E(X)$ completely contained in $U$. But since $U \subset D$, this implies $e \in E(D)$, a contradiction. 

For \ref{groundspace}, note that the inclusion $\ground{X-F} \subset \ground{X}$ holds for all edge sets $F \subset E(X)$ and all $X$, as free edges in $E(X) \setminus F$ remain free in $X-F$. For the converse inclusion to hold, however, the additional assumptions of the statement are necessary. So suppose there was $x \in \ground{X} \setminus \ground{X-F}$. Then there is a free arc $\alpha$ in $X - F$ with $x \in \alpha$. But then $\closure{\alpha} \cup X[F]$ is a compact graph-like space in $X$ forming a neighbourhood of $x$ in $X$, from which it follows that $x$ forms a singleton component in $X$.

For \ref{lem_removingspanningtrees2}, it now follows from the previous step that every component $X-F$ has diameter at least $\delta$, and so by \ref{lem_removingzerosequences}, $X-F$ must consist of finitely many Peano continua.
\end{proof}

\subsection{Sparse spanning trees} The purpose of this section is to give a fairly general procedure how to find non-trivial sparse edge sets. 


\begin{lemma}
\label{lem_findinggraphlikeswithtargets}
Let $X$ be a Peano continuum. For every zero-dimensional compact set $Y \subset \ground{X}$, there exists a standard\index{standard subspace} graph-like continuum $Z \subset X$ with $Y \subset Z$.
\end{lemma}

\begin{proof}
The proof modifies an idea by Ward of \emph{approximating a Peano continuum by finite trees}, see \cite{ward2} and \cite{ward}.

Let $\Sequence{\script{U}_n}:{n \in \N}$ be a refining sequence of finite $2^{-n}$ Peano covers of $X$ where $U_0 = \Set{X}$ is the trivial cover, {cf.~Definition~\ref{def_Bingpartition}}. For a subset $A \subset X$, define $\script{U}_n\restriction A := \set{U \in \script{U}_n}:{U \cap A \neq \emptyset}$. Recursively, we will define finite, i.e.\ compact trees $T_n \subset X$ and finite vertex sets $V_n \subset T_n$ such that for all $n\in\N$,
\begin{enumerate}[label=(\arabic*)]
\item\label{it(1)} $T_n \subset T_{n+1}$ as topological subspaces,
\item\label{it(2)} $V_n \subset V_{n+1}$,
\item\label{it(3)} $V_n$ is the set of branch- and end-vertices of $T_n$,
\item\label{it(4)} $\script{U}_n\restriction Y \subset \script{U}_n\restriction T_n$, and
\item\label{it(5)} $\script{U}_{n} \restriction Y$ covers $T_{n+1} \setminus T_n$, and
\item\label{it(6)} $\script{U}_{n} \restriction Y$ covers $V_{n+1} \setminus V_n$.
\end{enumerate}

Let $T_0= V(T_0) = \Set{t_0}$ be an arbitrary singleton tree. Since $U_0 = \Set{X}$, this satisfies \ref{it(4)}. All other conditions are trivial or vacuous at this point. This completes the base case. For the recursion step, suppose that $T_0,\ldots,T_n$ are already defined according to $(1)-(6)$, and pick finitely many points points $A=\Set{a_1,\ldots,a_k}$ such that $\script{U}_{n+1} \restriction Y = \script{U}_{n+1}\restriction A$. Let $S_0 :=T_n$, $V(S_0) := V_n$ and suppose we already have constructed a sequence of finite tree $S_0 \subset S_1 \subset \cdots \subset S_i$ for $i <k$ such that $S_i$ contains $\Set{a_1,\ldots,a_i}$ and such that $S_i \setminus T_n$ is covered by $\script{U}_{n} \restriction Y$. Consider $a_{i+1}$. Again, if $a_{i+1} \in S_i$, set $S_{i+1}:= S_i$. Otherwise, pick $U \in \script{U}_n$ such that $a_{i+1} \in U$, and also pick $t \in T_n \cap U$ (possible by \ref{it(4)}). Pick an arc $\alpha \colon I \to U$ from $t$ to $a_{i+1}$. Since $S_i$ is compact, there is a maximal $x_{i+1} < 1$ such that $\alpha (x_{i+1}) \in S_i$. Define $S_{i+1} = S_i \cup \alpha([x_{i+1},1])$, and $V(S_{i+1}) = V(S_i) \cup \Set{\alpha(x_{i+1}),a_{i+1}}$. Since $\alpha$ was an arc completely contained in $U$, we have $S_{i+1} \setminus T_n$ is covered by $\script{U}_{n} \restriction Y$. In the end, put $T_{n+1} := S_k$ and $V_{n+1} = V(S_k)$. Clearly, $T_{n+1}$ is a finite tree with vertex set $V_{n+1}$. Moreover, by choice of $A$, it satisfies \ref{it(4)}. Finally, \ref{it(5)} and \ref{it(6)} follow since all $S_i$ satisfied that $S_i \setminus T_n$ is covered by $\script{U}_{n} \restriction Y$, and so then does $S_k = T_{n+1}$. This completes the recursive construction.

Define $T = \bigcup_{n \in \N} T_n$, and $V = \bigcup V_n$. Our aim is to show that $Z=\closure{T}$ is a graph-like continuum containing $Y$. Clearly, $T$ is connected, and hence $Z$ is compact connected. To see that $Z$ covers $Y$, note that for any $y \in Y$, since $W_n:=\bigcup \p{ \script{U}_n \restriction \singleton{y}}$ has vanishing diameter for $n \to \infty$, the family $\set{W_n}:{n \in \N}$ forms a neighbourhood base of $y$ in $X$. By property \ref{it(4)}, every $W_n$ intersects $T$, and so $y \in \closure{T}$. Since $y \in Y$ was arbitrary, this shows $Y \subset \closure{T}= Z$. Finally, the proof that $Z$ is graph-like essentially relies on the following observation: 

\medskip
{\bf Claim:} \emph{For every $p \notin Y$ there is a open set $U \subset X$ with $p \in U$ such that for some $n \in \N$ we have $U \cap \closure{T} \subset T_n$ and $U \cap \closure{V} \subset V_n$.}

\medskip

To see the claim, note that if $p \notin Y$, then $\varepsilon = \operatorname{dist}\p{p,Y} > 0$, and so there is $n$ large enough such that $2^{-n} < \varepsilon$. Let $ W:= \bigcup \p{\script{U}_n \restriction Y}$ and $U = X \setminus W$. Then $U$ is open and $p \in U$. Moreover, $\closure{T} \cap U =  \closure{T} \setminus W = \p{ \closure{T_n} \cup \closure{T \setminus T_n}} \setminus W \subset \closure{T_n} = T_n$ by property \ref{it(5)}, and the fact that $T_n$ is compact. Similarly, $\closure{V} \cap U =  \closure{V} \setminus W \subset \closure{V_n} = V_n$ by property \ref{it(6)}, and the fact that $V_n$ is finite. This establishes the claim.

\medskip

Next, we argue that the set $V(Z):=Y \cup  V$ is a vertex set for $Z$ witnessing that $Z$ is graph-like. First, by the claim, $V(Z)$ is closed in $X$ and hence compact. Moreover, since each $V_n$ is finite and $Y$ is zero-dimensional, also $V(Z)$ is zero-dimensional by the countable sum theorem for dimension, \cite[Thm.\ 1.5.2]{engelkingdimension}. 

Further, we need to show that each $p \in Z \setminus V(Z)$ has a neighbourhood homeomorphic to an open interval. So let $p \in Z \setminus V(Z)$. Let $U$ be as in the claim, i.e.\ $U$ is a neighbourhood of $p$ such that  $U \cap Z = U \cap \closure{T} \subset T_n$. Then $U \setminus V_n$ is open, and $(U \setminus V_n) \cap Z \subset T_n \setminus V_n$ consists of finitely many connected components, each homeomorphic to an open interval.

Finally, to make $Z$ standard, define $Z' = Z \setminus \bigcup \set{e}:{e \cap Z \neq \emptyset \neq Z \setminus e}$. Since $Y \subset \ground{X}$, we still have $Y \subset Z'$, and further, $Z'$ is still connected, as no half edge is needed for connectivity in $Z$.
\end{proof}

\begin{defn}[Sparse spanning tree]
Let $X$ be a Peano continuum. A spanning tree $T$ of $X_\sim$ is \emph{sparse} if its edge set $E(T)$ is sparse in $X$.\index{spanning tree!sparse|textbf}
\end{defn}

\begin{lemma}[Existence of sparse spanning trees]
\label{lem_concentrated2}
Every Peano continuum $X$ with {product-structured ground space} $\ground{X} = V \times P$ admits a sparse spanning tree.
\end{lemma}

\begin{proof}
Pick $p \in P$, and put $Y := V \times \Set{p}$, a compact zero-dimensional subset of $\ground{X}$. By Lemma~\ref{lem_findinggraphlikeswithtargets}, there exists a standard graph-like continuum $Z \subset X$ with $Y \subset Z$. Let $\pi \colon X \to X_{\sim}$ be the quotient map. Since $Y$ intersects every component of $\ground{X}$, it follows that $\pi(Z)$ is a spanning graph-like subcontinuum of $X_{\sim}$. Let $T \subset \pi(Z)$ be a spanning tree of $X_{\sim}$. Then $E(T) \subset E(X_{\sim}) = E(X)$, and since $Z$ was graph-like, it is evident that $\closure{E(T)} \subset Z$ is a graph-like compactum, i.e.\ $E(T)$ is sparse in $X$.\index{edge set!sparse|)}
\end{proof}

\index{graph-like space|)}

\section{Tiles in Peano Graphs with Product-Structured Ground Spaces}
\label{s:dividing}

We discuss fractal properties of Peano continua $X$ with ground space $\ground{X} = V\times P$. 

\subsection{Tiles via horizontal restriction} First, we discuss tiles that result by restricting to well-behaved subsets of $V$. {In the following, we use the symbol $\bigoplus$ to denote a union of disjoint clopen sets.}

\begin{lemma}
\label{lem_LocConnectedGivesMulticut}
Every locally connected compactum $X$ with {product-structured} ground space  $\ground{X} = V \times P$ \textnormal{[}and dense edge set\textnormal{]} is of the form $X = \bigoplus_{A \in \script{A} } X_A$, where $\script{A}$ is a (finite) clopen partition of $V$ and $X_A \subset X$ is a standard Peano continuum  \textnormal{[}Peano graph\textnormal{]} with {product-structured} ground space $\ground{X_A} = A\times P$. 
\end{lemma}

\begin{proof}

As a locally connected compactum, $X$ has finitely many components, \cite[VI \S49, II Theorem~7]{kuratowski}. Moreover, since $P$ is connected, each component $C$ is of the form $C=X[A_C \times P]$ with $A \subset V$. Since $C$ is closed, if follows from compactness and the continuity of projection maps that $A_C \subset V$ is closed. Moreover, for distinct components $C \neq C'$ we clearly have $A_C \cap A_{C'} = \emptyset$. Therefore, every $A_C$ is a clopen subset of $V$. Hence, the collection $\script{A}$ of such clopen $A_C \subset V$ is the desired (finite) clopen partition of $V$.
\end{proof}

\begin{cor}
\label{cor_removingconsets}
If $X$ is a Peano graph with {product-structured ground space} $\ground{X} = V \times P$, and $F \subset E$ is sparse, then there is a (finite) clopen partition $\script{A}$ of $V$ such that $X - F= \bigoplus_{A \in \script{A} } X_A$ where each $X_A \subset X$ is a standard Peano graph with {product-structured} ground space $\ground{X_A} = A\times P$.
\end{cor}

\begin{proof}
By Lemma~\ref{lem_sparseproperties}\ref{lem_uniformlylargegroundspace}, the space $X - F$ is locally connected with ground space $\ground{X} = V \times P$, so the assertion follows from Lemma~\ref{lem_LocConnectedGivesMulticut}.
\end{proof}

\begin{cor}
\label{cor_restrictingV}
If $X$ is a Peano graph with {product-structured ground space} $\ground{X} = V \times P$ and $B\subset V$ is clopen, then there is a (finite) clopen partition $\script{B}$ of $B$ such that $X[B \times P]= \bigoplus_{{A} \in \script{B} } X_{{A}}$ where each $X_{{A}} \subset X$ is a standard Peano graph with {product-structured} ground space $\ground{X_{{A}}} = {A} \times P$.
\end{cor}

\begin{proof}
Since $F=E(B \times P, \p{V \setminus B} \times P)$ is a finite edge cut of $X$, the edge set $F$ is sparse, and so the result follows from the previous Corollary~\ref{cor_removingconsets}, by taking $\script{B}$ to be the subcollection of $\script{A}$ of elements that intersect $B$.
\end{proof}

\subsection{Tiles via vertical restriction} Next, we discuss tiles that result by restricting to well-behaved subsets of $P$.

\begin{lemma}
\label{lem:blowingUpConntdSets}
Let $X$ be a Peano graph, $x \in \ground{X}$, and $U \subset X$ a connected set such that $U \cap \ground{X}$ is a neighbourhood of $x$ in $\ground{X}$. Then for every $\varepsilon >0$ there is a connected neighbourhood $V$ of $x$ in $X$ such that $V \subset B_\varepsilon(U)$. 
\end{lemma}

\begin{proof}

If $y$ is an endpoint of some edge $e$, write $B_{\delta}^e(y)$ (where $0 < \delta \leq 1$) for the half-open interval with endpoint $y$ of diameter $\delta$ on $e$.
Then put
$$V := U \cup \set{B_{\varepsilon}^e(y)}:{e \in E \text{ and } y \in \closure{e} \cap U} \subset X.$$
Then $V$ is connected, and it is a neighbourhood of $x$ in $X$ (as almost all edges in $E$ have diameter $< \epsilon$), and by construction, we have $V \subset B_\varepsilon(U)$.
\end{proof}

\begin{lemma}
\label{lem_makeW-midconnected}
For every Peano graph $X$ with {product-structured} ground space  $\ground{X} = V \times P$, every $W \subset P$ a regular closed Peano subcontinuum and for every $\varepsilon > 0$, there is a (finite) clopen partition $\script{A}$ of $V$ with $\operatorname{mesh}(\script{A}) \leq \varepsilon$ such that $X[A \times W]$ is a Peano graph for all $A \in \script{A}$.
\end{lemma}

\begin{proof}
By Lemma~\ref{lem_LocConnectedGivesMulticut} it suffices to show that the induced subspace $X_W=X[V \times W]$ inherits local connectedness from $X$. This is trivial for points in the interior of $X_W$, i.e.\ interior points of edges, and points in $V \times \interior{W}$. So consider an arbitrary point $x=(v,w)$ for $v \in V$ and $w \in \partial W$, and fix $\delta>0$. Our task is to find a connected open neighbourhood $V$ of $x$ in $X_W$ of diameter at most $\delta$. First, pick a connected open neighbourhood $U$ of $w$ in $W$ with $\diam{U} < \delta/3$. Then $V \times  \p{U \cap \interior{W}}$ is a non-empty open subset of $X$, and so it follows from local connectedness of $X$ that there are $A \subset V$ clopen with $v \in A$, $B \subset U \cap \interior{W}$ open, and a connected open set $Y \subset X$ with $\diam{Y} < \delta/3$, $Y \subset U$ and $X[A\times B] \subset Y$. 

But then $Y'=Y \cup X[A \times U]$ is connected, and restricts to a neighbourhood of $(v,w)$ in $\ground{X_W}$ of diameter $\diam{Y'} \leq \delta/3$. So applying Lemma~\ref{lem:blowingUpConntdSets} to $Y'$ with $\epsilon = \delta/3$ provides a connected neighbourhood as desired.
\end{proof}

\subsection{Ground-space covering tiles}

\begin{lemma}
\label{lem_supersetedges}
Suppose for a Peano continuum $P$ with edges $E=E(P)$ and ground space $Z=Z(P)$,  we have a set of edges $F$ such that $Z \cup \bigcup F$ is locally connected. Then $Z \cup \bigcup F'$ is locally connected for all $F \subseteq F' \subseteq E$.
\end{lemma}

\begin{proof}
Let $Y = Z \cup \bigcup F$. By local connectedness, all components of $Y$ are open, and so it follows from compactness that $Y$ has finitely many components. Moreover, since the edges in $F' \setminus F$ form a zero-sequence of Peano subcontinua, the result now follows from (a natural adaption of) Lemma~\ref{lem_addingzerosequences}.
\end{proof}

Relying on the results established above about sparse spanning trees, our aim for this short section is to prove the following theorem.
 
\begin{theorem}
\label{thm_dividingthm}
The edge set $E(X)$ of every Peano graph $X$ with {product-structured} ground space $\ground{X}=V \times P$ with $P$ non-degenerate admits a bipartition $E(X) = E_1 \sqcup E_2$ into two edge sets both dense for $\ground{X}$ such that both $X_i = X[E_i]$ are locally connected.
\end{theorem}

\begin{proof}
Let $\Sequence{\script{U}_n}:{n \in \N}$ be a decreasing sequence of $2^{-n}$-Peano partitions for $P$ with $\script{U}_0 = \Set{P}$, {cf.~Definition~\ref{def_Bingpartition}}. Let $\script{R}=(R,\leq)$ be the corresponding \emph{refinement tree}, that is $\script{R}(n)$, the $n$th level of $\script{R}$, indexes the elements of $\script{U}_n$, so $\script{U}_n = \set{U_r}:{r \in \script{R}(n)}$, and $r \leq r'$ if and only if $U_r \supseteq U_{r'}$. Recall that each $\script{U}_n$ is finite, and so $\script{R}$ is a locally finite tree. Write $\script{R}({\leq} n):= \bigcup_{i \leq n} \script{R}(i)$ and similarly $\script{R}({<} n):= \bigcup_{i < n} \script{R}(i)$.

We now recursively construct
\begin{itemize}
\item a family of finite multicuts $\set{\script{A}_r}:{r \in \script{R}}$ of $V$, and
\item subtrees $T_{r,A} \subset X_\sim$ for $r \in \script{R}$ and $A \in \script{A}_r$
\end{itemize}
such that 
\begin{enumerate}
\item\label{divide1} $r \leq r' \in \script{R}$ implies $\script{A}_r  \succcurlyeq \script{A}_{r'}$,
\item\label{divide2} $\operatorname{mesh}(\script{A}_r) \leq 2^{-n}$ for $r \in \script{R}(n)$,
\item\label{divide3} for each $r \in \script{R}(n)$ and $A \in \script{A}_r$, the space 
$$X_{r,A}=X[A \times U_r] \setminus \bigcup \set{E(T_{A',s})}:{s \in \script{R}({<} n), \; A' \in \script{A}_s}$$
is a Peano graph, 
\item\label{divide4} $T_{r,A}$ is a sparse spanning tree for $X_{r,A}$ for all $r \in \script{R}$ and $A \in \script{A}_r$ (unless $(X_{r,A})_\sim$ has a single vertex, in which case $T_{r,A}$ consists of an arbitrary edge from $X_{r,A}$).
\end{enumerate}

For $n = 0$, and $r \in \script{R}(0)$ the unique root of $\script{R}$, the trivial (finite) clopen partition $\script{A}_r=\Set{V}$ is clearly sufficient. Now let $n \in \N$ and suppose we have already defined finite multicuts $\set{\script{A}_r}:{r \in \script{R}({\leq} n)}$ of $V$, and subtrees $T_{r,A} \subset X_\sim$ for $r \in \script{R}({\leq} n)$ and $A \in \script{A}_r$ according to (1)--(4). Consider $r \in \script{R}(n)$. Since $X_{r,A}$ is a Peano graph by (\ref{divide3}), we may use Lemma~\ref{lem_concentrated2} to find sparse spanning trees $T_{r,A}$ for $X_{r,A}$ for each $A \in \script{A}_r$, unless $A$ is a singleton, in which case we let $T_{r,A}$ consist of an arbitrary edge from $X_{r,A}$. Then property (\ref{divide4}) is satisfied. By Corollary~\ref{cor_removingconsets}, each 
$$X'_{r,A}:=X_{r,A} \setminus \bigcup \set{E(T_{A',s})}:{s \in \script{R}(n), \; A' \in \script{A}_s}$$
remains locally connected. Consider an arbitrary successor $s$ of $r$, i.e.\ some $s \in \script{R}(n+1)$ with $r< s$. By Corollary~\ref{cor_restrictingV} and Lemma~\ref{lem_makeW-midconnected}, there is a (finite) clopen partition $\script{B}_{A,s}$ of $A$ with $\mesh{\script{B}_{s,A}} \leq 2^{-\p{n+1}}$ such that $X'_{r,A}[B \times U_s]$ is a Peano continuum for each $B \in \script{B}_{s,A}$. Then $\script{A}_s := \bigcup \set{\script{B}_{s,A}}:{A \in \script{A}_r}$ satisfies (\ref{divide1}), (\ref{divide2}) and (\ref{divide3}).

Once the recursion is complete, let us write
$L_n := \bigcup \set{E(T_{r,A})}:{r \in \script{R}(n), \; A \in \script{A}_r}$ for the edge set of all trees on level $n \in \N$, and note that it follows from properties~(\ref{divide3}) and (\ref{divide4}) that $L_n \cap L_m = \emptyset$ for all $n \neq m \in \N$. Thus, by 
defining
$$E'_1 = \bigcup_{n \in \N} L_{2n} \quad \text{and} \quad E'_2 = \bigcup_{n \in \N} L_{2n+1}$$ 
we obtain two disjoint edge sets of $E$. So it remains to check that $E'_1$ and $E'_2$ each are dense in $V \times P$ and induce a locally connected subspace of $X$. This will complete the proof, as then by Lemma~\ref{lem_supersetedges}, any partition $E = E_1 \sqcup E_2$ with $E_1 \supseteq E'_1$ and $E_2 \supseteq E'_2$ satisfies the assertion of the lemma.

Indeed, to see that $X[E'_1]$ is locally connected and dense, pick $(v,p) \in V \times P$ and $\delta > 0$ arbitrarily, and let $k=2n$ large enough so that $\mesh{\script{A}_k} < \delta/2$ and $\mesh{\script{U}_k} < \delta/4$ by (\ref{divide1}). Pick $A \in \script{A}_k$ with $v \in A$ and let $U= \bigcup \set{U' \in \script{U}_k}:{p \in U'}$. Then $\diam{U}<\delta/2$ and $p \in \interior{U}$. By choice of $T_{r,A}$ in (\ref{divide4}) (where $r \in \script{R}(k)$ is the index of an element $U_r \subset U$) we have $(A \times U) \cup T_{r,A} \subset X[E'_1]$ is connected, of diameter at most $\delta$, and contains at least one edge. Using Lemma~\ref{lem:blowingUpConntdSets}, and the fact that $\delta$ was arbitrary, this establishes local connectedness and density for $E'_1$. The case $E'_2$ is similar after choosing $k$ to be odd. 
\end{proof}

\subsection{A decomposition theorem}
\label{subsec_decomp2}
\label{subsec_decomp3}

The following result combines the combinatorial techniques from Section~\ref{s:fundamentalcycles} with the topological techniques from the previous Sections~\ref{s:concentrated} and \ref{s:dividing}. It will be used to prove our main decomposition theorem below. 

Recall that $\partial A$ denotes the boundary operator, {and that a subset $S$ is said to \emph{separate} \index{separator} two (not necessarily disjoint) subsets $A$ from $B$ in $X$ if each connected component of the subspace $X \setminus S$ intersects at most one of $A$ or $B$.}

\begin{lemma}
\label{lem_ipperlower3}
Let $Q_1$ and $Q_2$ be Peano subcontinua of some non-degenerate Peano continuum $P$ such that (a) $Q_1 \cup Q_2 = P$, (b) $Q_1 \setminus {\interior{Q_2}}$ and $Q_2 \setminus {\interior{Q_1}}$ are non-empty regular closed subcontinua with connected interior, and (c) $Q_1 \cap Q_2 = W = W_1 \oplus \cdots \oplus W_k $ is a finite disjoint union of regular closed Peano continua $W_i$ each with connected interior such that $\interior{W}$ separates $Q_1$ from $Q_2$.\footnote{For a typical example let $P = S^1$, and $Q_1$ a clockwise arc on $P$ from 8 to 4 o'clock, and $Q_2$ a clockwise arc on $P$ from 2 to 10 o'clock. } 
Then for any locally connected compactum $X$ with dense edge set and { product-structured ground space} $\ground{X} = V \times P$, there is a partition $E(X) = E_1 \sqcup E_2 \sqcup F$ such that
\begin{enumerate}
\item $X[E_i]$ is locally connected, and $\partial E_i = V \times Q_i$ for $i=1,2$,
\item $\cardinality{F} < \infty$, 
\item $X[E_2]$ satisfies the even-cut condition.
\end{enumerate}
\end{lemma}

\begin{proof}
We may assume that $X[V \times W_1]$ is connected -- as otherwise, by (c) and Lemma~\ref{lem_makeW-midconnected}, there is a clopen partition $\script{B}$ of $V$ such that $X[B \times W_1]$ is a Peano continuum for all $B \in \script{B}$. Assign the finitely many cross-edges of the clopen partition associated with $\script{B}$ to $F$ and apply the following argument to each $X[B \times P]$ individually. Hence we may find, by Lemma~\ref{lem_concentrated2}, a sparse spanning tree $T \subset X_\sim$ such that for any edge $e \in E(T)$, both its endpoints lie in $ V \times W_1$. By Lemma~\ref{lem_removingspanningtrees2}, the remaining space $X':=X[V \times P] -E(T) $ is a locally connected, metrisable compactum with a dense collection of edges. 

Hence, by Lemma~\ref{lem_makeW-midconnected} and Theorem~\ref{thm_dividingthm}, we can partition each edge set of $X'[V \times W_i]$ into $E^i_1$ and $E^i_2$ such that both $(V \times W_i) \cup E^i_j$ are locally connected with $E^i_j$ being a dense collection of edges for all $i \in [k]$ and $j \in [2]$. Let
$$E'_1 = \bigcup \set{E^i_1}:{i \in [k]} \cup \set{e=xy \in E(X)}:{x \in V \times \p{Q_1 \setminus Q_2}, \; y \in V \times Q_1}$$
and
$$E'_2 =  \bigcup \set{E^i_2}:{i \in [k]} \cup \set{e=xy \in E(X)}:{x \in V \times \p{Q_2 \setminus Q_1}, \; y \in V \times Q_2}.$$
We claim that $\partial E'_j = V \times Q_j$ and $(V \times Q_j) \cup E'_j$ is locally connected for $j=1,2$. Consider the case $j=1$ (the other case is similar). By (b) and Lemma~\ref{lem_makeW-midconnected}, it follows that $X[V \times \p{Q_1 \setminus \interior{Q_2}}]$ is locally connected. And by construction, we also have $\p{V \times W} \cup \bigcup \set{E^i_1}:{i \in [k]}$ is locally connected. Hence, it follows that their union is a locally connected space with ground set $V \times Q_1$ whose edge set is a subset of $E'_1$. But then it follows from Lemma~\ref{lem_supersetedges} that we may add all remaining edges from $E'_1$ without harming local connectedness or density. The claim is established.

By this point, we have accounted for all edges in $E(X)$ apart from edges of $T$, and edges of $F:=E(V \times \p{Q_1 \setminus Q_2} , V \times \p{Q_2 \setminus Q_1})$. Note that $F$ is finite: since $\interior{W}$ separates $Q_1$ from $Q_2$, the sets $\p{Q_1 \setminus Q_2}$ and $\p{Q_2 \setminus Q_1}$ have positive distance from another, and so since $E(X)$ forms a zero-sequence, only edges of sufficiently large diameter can be in $F$.
 
Thus, it remains to distribute the edges of $T$ between $E'_1$ and $E'_2$. We will do this as to make sure that $X[E_2]$ satisfies the even-cut condition, and let $E_2 = \sum \set{C_e}:{e \in E'_2}$, i.e.\ consider the thin sum of fundamental cycles of edges in $E'_2$ with respect to $T$, Definitions~\ref{def:fundcycles} and \ref{def:thinsum}. Note that $E'_2 \subset E_2 \subset E'_2 \cup E(T)$, so $\partial E_2 = V\times Q_2$. Moreover, since $E_2$ is the thin sum of {fundamental cycles}, it follows from Theorem~\ref{thm_thinsumfundamentalcircuits} that $X[E_2]$ satisfies the even-cut condition. Finally, let $E_1 := E(X) \setminus \p{E_2 \cup {F}}$. Then also $E'_1 \subset E_1 \subset E'_1 \cup E(T)$, so $\partial E_1 = V\times Q_1$. Moreover, as $E_1$ and $E_2$ are supersets of $E'_1$ and $E'_2$ respectively, both $(V \times Q_i) \cup E_i$ are locally connected by Lemma~\ref{lem_supersetedges}. 
\end{proof}


Recall the definition of a Peano partition from Definition~\ref{def_Bingpartition}. We can visualize the way the different elements of such a partition $\script{U}$ interact by its intersection graph $G_\script{U}$, see Definition~\ref{def_intersectiongraph}. Note that if this $\script{U}$ is a finite cover of a Peano continuum $X$, it follows from the connectedness of $X$ that $G_\script{U}$ is a finite connected graph. 

\begin{lemma}
\label{lem_positivedistance}
Let $\script{U}$ be a finite Peano partition of a connected set $X$, $G_\script{U}$ its associated intersection graph, and $U \in \script{U}$. If we denote by $N(U)$ all neighbours of $U$ in $G_\script{U}$, then $U$ and $\bigcup V(G_\script{U}) \setminus \p{U \cup N(U)}$ are disjoint closed sets in $X$, and therefore have some positive distance.
\end{lemma}

\begin{proof}
They are disjoint by the definition of intersection graph and neighborhood, and they are closed as a finite union of closed sets.
\end{proof}

\begin{theorem}[Decomposition Theorem]
\label{thm_decompositionforProductRemainders}
For every $\varepsilon > 0$ and every Peano continuum $P$, there exists a finite cover
$\script{P}=\Set{P_1,\ldots,P_k}$ of $P$ consisting of Peano subcontinua with $\mesh{\script{P}} < \varepsilon$ such that 
{every locally connected compactum $X$ with dense edge set $E=E(X)$ and product-structured ground space $\ground{X} = V \times P$} admits a finite partition $E = E_1 \sqcup \cdots \sqcup E_{k} \sqcup F$ such that
\begin{enumerate}
\item $\cardinality{F} < \infty$, 
\item $\partial E_i = V \times P_i$,
\item $X_i := X[E_i]$ is locally connected for all $i \in [k]$,
\item $X_i$ satisfies the even-cut condition for all $i \neq 1$.
\end{enumerate}
\end{theorem}

Note that while $\Set{P_1,\ldots,P_k}$ is not a Peano partition of $P$, but only a cover (i.e.\ $P_i \cap P_j$ may have non-empty interior), the resulting tiles $\Set{X_1,\ldots,X_k}$ of the decomposition theorem together with the finitely many edges from $F$ do form a Peano partition of $X$: for all these tiles and edges are edge-disjoint, and as the edges of $X$ are dense, this means they all have pairwise disjoint interiors.

\begin{proof}
Suppose for a contradiction that the statement is false for some $\varepsilon > 0$, and consider the class $\script{C}$ of all Peano continua that witness the failure of $\varepsilon$. For each $P \in \script{C}$ let $k_P \in \N$ denote the minimum {cardinality} over all $\varepsilon/3$-Peano partitions of $P$, and fix $P \in \script{C}$ such that $k = k_P$ is minimal. Let $\script{U}$ be a $\varepsilon/3$-Peano partition of $P$ with $\cardinality{U} = k$, which exists by Theorem~\ref{thm_BingBrickPartition}.
 
Clearly, we must have $k \geq 3$, as otherwise, $\diam{P}< \epsilon$ and there is nothing to do. Now pick a spanning tree $T$ for its associated intersection graph $G=G_\script{U}$ (see Definition~\ref{def_intersectiongraph}), and let $U$ be a leaf of this tree, and denote by $N_G(U)$ the neighbourhood of $U$ in $G_\script{U}$. Set $P' := U \cup \bigcup N(U)$ and $P'' = \bigcup V(T) \setminus \Set{U}$. Since $U$ was a leaf of $T$, the induced subgraph $G_\script{U} - \Set{U}$ is connected, $P'$ and $P''$ are both Peano subcontinua of $P$ together covering $P$ such that $\interior{P' \cap P''} = \interior{\bigcup N({U})}$ consists of finitely many Peano subcontinua of $P$ separating $P'$ from $P''$, see Lemma~\ref{lem_positivedistance}. Also note that $\diam{P'} \leq \varepsilon$. 

Further, note that $\script{U}' := \script{U} \setminus \Set{U}$ is an $\epsilon/3$-Peano partition for the Peano continuum $P''$. By minimality of $k_P$, it follows that $P'' \notin \script{C}$ and so there is a finite cover $\script{Q}=\Set{P_1,\ldots,P_{\ell}}$ of $P''$  satisfying the conclusion of the theorem. To obtain the final contradiction, we show that the finite cover $\script{P}=\Set{P_1,\ldots,P_{\ell},P'}$ of $P$ witnesses that $P$ could not have been a counterexample. Clearly, $\mesh{P} < \varepsilon$.

To see the other assertions, consider an arbitrary locally connected compactum $X = \p{V \times P} \cup E$ with $V$ compact zero-dimensional and the collection of free arcs $E$ being dense. By construction of $P'$ and $P''$ we may apply Lemma~\ref{lem_ipperlower3} to find a partition $E = E' \sqcup E'' \sqcup F'$ of the edge set $E$ of $X$ such that
\begin{itemize}
\item $\partial E'= V \times P'$ and $X' = (V \times P') \cup E'$ is locally connected and satisfies the even-cut condition,
\item $\partial E''= V \times P''$ and $(V \times P'' )\cup E''$ is locally connected, and
\item $\cardinality{F'} < \infty$.
\end{itemize}
Next, by the assumptions on the cover $\script{Q}$ of $P''$, we may find a further partition $E'' = E_1 \sqcup \cdots \sqcup E_{\ell} \sqcup F''$ such that
\begin{itemize}
\item $\cardinality{F''} < \infty$, 
\item $E_i$ is dense in $V \times P_i$,
\item $X_i := V \times P_i \cup E_i$ is locally connected for all $i \in [\ell]$,
\item $X_i$ satisfies the even-cut condition for all $i \neq 1$.
\end{itemize}
But then we see that the edge partition $E = E_1 \sqcup \cdots \sqcup E_{\ell} \sqcup E' \sqcup F$ for $F:=F' \cup F''$ witnesses that $\script{P}$ does satisfy the assertion of the theorem after all. 
\end{proof}

\section{Approximating Sequences of Eulerian  Decompositions}
\label{sec_approxforprodcutground}

\subsection{Blanket assumptions} \label{blanketassumptions} Given our work in Chapter~\ref{chapter_Eulerianmaps}, for our proof of Theorem~\ref{thm_ProductSpaceRemainders}, we may assume for the rest of this chapter, without any loss of generality, that our Peano continuum $X$ with product structured ground space $\ground{X} = V \times P$ and edge set $E=E(X)$ satisfies the following additional assumptions: 
\begin{itemize}
\item $X$ is a Peano graph without loops by the second reduction result, Theorem~\ref{thmReduction2}. 
\item $X$ has diameter bounded by $1$.
\item $P$ is not a singleton (as otherwise, $X$ is a graph-like continuum, a class for which the Eulerianity conjecture is already known to hold \cite{euleriangraphlike}).
\end{itemize}

\subsection{Covering the ground-set by tiles}
The plan is now to apply the decomposition Theorem~\ref{thm_decompositionforProductRemainders} recursively, in order to construct an approximating sequence of Eulerian decompositions for $X$ as in Theorem~\ref{thm_WeaklyEulerianMappingThm2}. 

First, we recursively construct a sequence $(\script{P}_n \colon n \in \N)$ of finite covers of $P$ and a locally finite tree $\script{R}$ with levels $\script{R}(n)$ such that for all $n \in \N$ we have
\begin{enumerate}
\item[(COVER)]
	\begin{enumerate}
\item $\script{P}_0 = \Set{P} = \Set{P_r}$ for $\Set{r}=\script{R}(0)$ the root of $\script{R}$, 
\item $\mesh{\script{P}_n} \leq  2^{-n}$, and
\item $\script{P}_{n+1} \preceq \script{P}_{n}$ witnessed by the \emph{refinement tree} $\script{R}$, i.e.\ for all $r < r'$ with $r \in \script{R}(n)$ and $r' \in \script{R}(n+1)$ we have $P_r \in \script{P}_{n}$, $P_{r'} \in \script{P}_{n+1}$ and ${P_r \supseteq P_{r'}}$, 
\item For $r \in \script{R}(n)$ writing $r^+ := \set{s \in \script{R}(n+1)}:{r < s}$, we have that $\set{P_s}:{s \in r^+}$ is a finite cover of $P_r$ satisfying the assertions of Theorem~\ref{thm_decompositionforProductRemainders} for $P_r$.
\end{enumerate}
\end{enumerate}
The base case is given in $(a)$. Now whenever $\script{P}_n$ is already constructed, pick for each $Q \in \script{P}_n$ a cover $\script{P}_Q$ of $\mesh{\script{P}_Q} \leq 2^{-\p{n+1}}$ according to the Decomposition Theorem~\ref{thm_decompositionforProductRemainders} for $Q$, and let $\script{P}_{n+1} := \bigcup \set{\script{P}_Q}:{Q \in \script{P}_n}$. Moreover let $\script{R}=(R,\leq)$ be the corresponding \emph{refinement tree}, that is $\script{R}(n)$, the $n$th level of $\script{R}$, indexes the elements of $\script{P}_n$, so $\script{P}_n = \set{P_r}:{r \in \script{R}(n)}$, and $r < r'$ for $r \in \script{R}(n)$ and $r' \in \script{R}(n+1)$ if and only if $P_{r'} \in \script{P}_{P_r}$.

To formulate our next properties, we use the following piece of notation: if $r \in \script{R}(n)$, then $r^-$ denotes the unique node in $\script{R}(n{-}1)$ with $r^- < r$. In fact, note that $\script{R}$ embeds into the tree $\N^{<\N}$ of finite natural sequence ordered by extension. Thus, without loss of generality, we assume from now on that $\script{R} \subset \N^{<\N}$. In particular, the root of $\script{R}$ will be denoted by $\emptyset$, each level $\script{R}(n) = \script{R} \cap \N^n$ consists of the $n$-element sequences in $\script{R}$, and for every $r \in \script{R}$ we may assume that $r^+ = \Set{r\cat 0,r\cat 1,\ldots,r\cat k_r }$ for some suitable $k_r \in \N$, with $r \cat i$ denoting the extension of the finite sequence $r$ by a new last element $i$.

We now construct by recursion on $n \in \N$
\begin{itemize}
\item a family $\set{\script{A}_r}:{r \in \script{R}(n)}$ of (finite) clopen partitions of $V$, 
\item a family $ \set{E_{r,A}}:{r \in \script{R}(n), A \in \script{A}_r}$ of pairwise disjoint subsets of $E$, and
\item a family $ \set{F_{r,A}}:{r \in \script{R}(n), A \in \script{A}_r}$ of pairwise disjoint, finite subsets of $E$,
\end{itemize}  
such that for all $r \in \script{R}$ the following holds:
\begin{enumerate}
\item[(CUT)]
	\begin{enumerate}
    	\item $\script{A}_r = \Set{V}$ for the unique node $r \in \script{R}(0)$,
\item $\operatorname{mesh}(\script{A}_r) \leq 2^{-n}$ for all $r \in \script{R}(n)$,
\item $r \leq r' \in \script{R}$ implies $\script{A}_r  \succcurlyeq \script{A}_{r'}$,
	\end{enumerate}
\item[(EDGE)]
	\begin{enumerate}
\item $E_{r,V} = E$ for the unique node $r \in \script{R}(0)$,
\item $E_{r,A} =  F_{r,A}  \sqcup \bigsqcup \set{E_{s,A'}}:{s \in r^+, A' \in \script{A}_s}$ for all $A \in \script{A}_r$, 
\end{enumerate}
\item[(TILE)]
	\begin{enumerate}
\item $X_{r,A} = X[E_{r,A}]$ is a Peano graph with $\ground{X_{r,A}}=A \times P_r$ for all $A \in \script{A}_r$,
\item all tiles $X_{A,s}$ for $s \in r^+ \setminus \Set{r\cat 0}$ and $A \in \script{A}_{s}$ satisfy the even-cut condition.
\end{enumerate}
\end{enumerate}

\begin{proof}[Construction]
By recursion on $n \in \N$. The base case {is given by setting $E_{r,V}:= E$ for} the unique node $r \in \script{R}(0)$, {in accordance with EDGE(a)}: then we have $X=\p{V\times P} \cup E = \p{A\times P_r} \cup E_{r,A} = X_{r,A} \text{ for } A \in \script{A}_r = \Set{V}$.

Now suppose the construction has progressed up to some tile $X_{r,A}$ with $r \in \script{R}(n)$ and $A \in \script{A}_r$, which is a Peano graph with ground space $A \times P_r$ by TILE(a). By Corollary~\ref{cor_restrictingV} there is a (finite) clopen partition $\script{B}_A$ of $A$ with $\operatorname{mesh}(\script{B}_A) \leq 2^{-(n+1)}$ such that $X_{r,B} = X_{r,A}[B \times P_r]$ is a Peano graph with ground space $B \times P_r$ for each $B \in \script{B}_A$. Let $F(\script{B}_A)$ denote the finite set of cross-edges the clopen partition $\script{B}_A$ induces in $X_{r,A}$.

By property COVER(d) for $P_r$, the Decomposition Theorem~\ref{thm_decompositionforProductRemainders} applied to $X_{r,B}$ returns a finite partition
$$E_{r,B} = E_{r \cat 0,B} \sqcup \cdots \sqcup E_{r \cat k_r,B} \sqcup F_{r,B}$$ 
so that the corresponding tiles $Y_{i,B} := \p{B \times P_{r \cat i}} \cup E_{r \cat i,B}$
are locally connected with a dense collection of edges for all $i \leq k_r$, and so that $Y_{i,B}$ satisfies the even-cut condition for all $i \neq 0$. By Lemma~\ref{lem_LocConnectedGivesMulticut}, for each $Y_{i,B}$ there is a (finite) clopen partition $\script{A}_{r \cat i,B}$ of $B$ so that $Y_{i,B} = \bigoplus_{A' \in \script{A}_{r \cat i},B} X_{r \cat i, A'}$ where $X_{r \cat i, A'} \subset Y_{i,B}$ is a standard Peano graph with ground space $\ground{X_{r \cat i, A'}} = A' \times P_{r \cat i}$ and edge set say $E_{r \cat i,A'}$, giving TILE(a), , and $F_{r,A} = F(\script{B}_A) \cup \bigcup_{B \in \script{B}_A} F_{r,B}$ is finite, satisfying EDGE(b). Further, by the moreover-part of Lemma~\ref{lem_LocConnectedGivesMulticut}, each $X_{r \cat i, A'}$ for $A \in \script{A}_{r \cat i,B}$ with $i \neq 0$ satisfies the even-cut condition, giving TILE(b). Now for each $i \leq k_r$ define $\script{A}_{r \cat i} = \bigcup_{A \in \script{A}_r} \bigcup_{B \in \script{B}_A} \script{A}_{r \cat i,B}$, which is a (finite) clopen partition of $V$ satisfying CUT(b) and (c). Then by construction, for all $A' \in \script{A}_{r \cat i}$ we have $X_{r \cat i,A'}  = X[E_{r \cat i,A'}]$ is a Peano graph.  The construction is complete.
\end{proof}

We need the following two elementary results, the proofs of which are evident.

\begin{lemma}
\label{evencuttopsum}
{Let $X$ be a metrizable compactum with a finite clopen partition $X = \bigoplus_{i \in [n]} X_i$.} Then $X$ satisfies the even-cut condition if and only if each $X_i$ satisfies the even-cut condition. \qed
\end{lemma}

\begin{lemma}
\label{evencutcomplement}
Let $Z$ be a compact graph-like space satisfying the even-cut condition. Suppose that {$E(Z)$ has a partition} $E(Z) = E_0 \sqcup \cdots \sqcup E_k$ such that $Z[E_i]$ satisfies the even-cut condition for all $1 \leq i \leq k$. Then also $Z[E_0]$ satisfies the even-cut condition. \qed
\end{lemma}

For $k \in \N$ and $s \in \N^{<\N}$, write $s\cat 0^k := s \cat {\underbrace{0 \cat 0 \cat \cdots \cat 0}_{k \text{ times}}}$. When using this notation, we usually require that $s$ does not end on $0$.

For $r \in \script{R}$, let $E_r : = \bigcup_{A \in \script{A}_r} E_{r,A}$, $F_r : = \bigcup_{A \in \script{A}_r} F_{r,A}$ and $X_r = X [E_r]$. Then $X_r = \bigoplus_{A \in \script{A}_r} X_{r,A}$, and hence it follows by property TILE(b) and Lemma~\ref{evencuttopsum} that whenever $r$ does not end on $0$, then $X_r$ satisfies the even-cut condition. 

The following simple observation is the key for constructing an Eulerian decomposition. {Informally, it says that while tiles indexed by a finite sequence ending on $0$ do not satisfy the even-cut condition, this property can be restored by taking into account (finitely many) exceptional edge sets $F_r$.} 

\begin{lemma}
\label{lem_evencutlowertilesplusF}
For every $t \in \N^{<\N}$, and $s$ not ending on $0$ with $t = s\cat 0^n$, the graph-like space $ Z_t := X_\sim[E_{t} \sqcup \bigsqcup_{k= 0}^{n-1} F_{s\cat 0^k}]$ has the even-cut property.
\end{lemma}

\begin{proof}
First, if $n = 0$, then $Z_t = X_\sim[E_{t}]$ has the even-cut property by assumption if $t = \emptyset$, and otherwise by TILE(b) and Lemma~\ref{evencuttopsum}. Now  consider $t = s\cat 0^{n+1}$, let $r = s \cat 0^n$ and assume inductively that $Z_r$ has the even-cut property. Recall that by EDGE(b), we have $E_{r} =  F_{r}  \sqcup \bigsqcup \set{E_{s}}:{s \in r^+}$. Since each $s \neq r\cat 0$ has the even-cut property, it follows from Lemma~\ref{evencutcomplement} that also the complement of these sets in $Z_r$ has the even-cut property. But clearly, the edge-complement of $\set{E_{s}}:{s \in r^+}$ is precisely $Z_t$.
\end{proof}

\subsection{Three auxiliary graphs}

To build an approximating sequence of Eulerian decompositions, we will now construct suitable Eulerian multi-graphs $(G_n,\eta_n)$ approximating the decomposition constructed above in TILE(a). We will do this in three stages reminiscent of the steps in the blueprint from Observation~\ref{obs_blueprint}.
\begin{itemize}
	\item First, construct a sequence of auxiliary multi-graphs $\Sequence{G'_n}:{n \in \N}$ each living on the tiles at stage $n$ and has as edge set $F_n$ of all remaining edges of $X$ at stage $n$.
	\item Second, we form a sequence of even\footnote{A finite graph is called \emph{even} if all its vertices have even degree.} multi-graphs $\Sequence{G''_n}:{n \in \N}$, where each $G''_n$ is a supergraph of $G'_n$ formed by adding some type-E dummy edges. This step is the critical part of the argument, relying on the even-cut properties in TILE(b).
	\item Finally, form a sequence of even, connected multi-graphs $\Sequence{G_n}:{n \in \N}$, where each $G_n$ is a super-graph of $G''_n$ formed by adding some type-C dummy edges to $G''_n$,\footnote{The purpose of \emph{type-E edges} will be to make all degrees of $G_{n+1}$ \textbf{e}ven, and the purpose of \emph{type-C edges} is to make $G_{n+1}$ \textbf{c}onnected.}
\end{itemize}
making sure in all steps that we always have compatible inverse limits $\varprojlim G'_n \hookrightarrow \varprojlim G''_n \hookrightarrow \varprojlim G_n$, each with contraction maps (Definition~\ref{defn_edgecontraction}) as bonding maps. The reader may picture this process as in the following two figures, Figures~\ref{fig:fig7} and \ref{fig:fig8}.

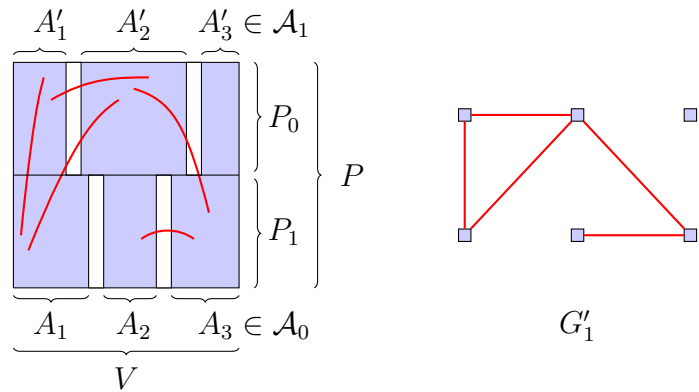
\begin{figure}[ht]
\begin{tikzpicture}[scale=1]
\draw[fill=blue!20!white] (0,0) -- (3,0) -- (3,3) -- (0,3) -- (0,0);
\draw[mybrace=0.5] (3,-.8) -- (0,-.8);
\node at (1.5,-1.2) {$V$};

\draw[mybrace=0.5] (3.2,1.45) -- (3.2,0);
\node at (3.6,2.25) {$P_0$};
\draw[mybrace=0.5] (3.2,3) -- (3.2,1.55);
\node at (3.6,.75) {$P_1$};
\draw[mybrace=0.5] (4,3) -- (4,0);
\node at (4.5,1.5) {$P$};

 \node[draw, white, rectangle,scale=.6, fill=white] (N1) at (6,.7) {};
 \node[draw, white, rectangle,scale=.6, fill=white] (N2) at (6,2.3) {};
 
 \node[draw, white, rectangle,scale=.6, fill=white] (N3) at (7.5,.7) {};
 \node[draw, white, rectangle,scale=.6, fill=white] (N4) at (7.5,2.3) {};

 \node[draw, white, rectangle,scale=.6, fill=white] (N5) at (9,.7) {};
 \node[draw, white, rectangle,scale=.6, fill=white] (N6) at (9,2.3) {};

\draw[fill=white] (1,0) -- (1.2,0) -- (1.2,1.5) -- (1,1.5) -- (1,0);
\draw[fill=white] (1.9,0) -- (2.1,0) -- (2.1,1.5) -- (1.9,1.5) -- (1.9,0);
\draw[fill=white] (.7,1.5) -- (.9,1.5) -- (.9,3) -- (.7,3) -- (.7,1.5);
\draw[fill=white] (2.3,1.5) -- (2.5,1.5) -- (2.5,3) -- (2.3,3) -- (2.3,1.5);

\draw[mybrace=0.5] (1,-.1) -- (0,-.1);
\node at (.5,-.5) {$A_1$};
\draw[mybrace=0.5]  (1.9,-.1) -- (1.2,-.1);
\node at (1.6,-.5) {$A_2$};
\draw[mybrace=0.5] (3,-.1) -- (2.1,-.1);
\node at (3.2,-.5) {$A_3 \in \script{A}_0$};

\draw[mybrace=0.5] (0,3.1) -- (.7,3.1);
\node at (.5,3.5) {$A'_1$};
\draw[mybrace=0.5]  (.9,3.1) -- (2.3,3.1);
\node at (1.6,3.5) {$A'_2$};
\draw[mybrace=0.5] (2.5,3.1) -- (3,3.1);
\node at (3.2,3.5) {$A'_3 \in \script{A}_1$};

\draw[red,thick] (.5,2.5) .. controls (1,2.8) and  (1.4,2.8) .. (1.8,2.8);
\draw[red,thick] (.2,.5) .. controls (.8,2) and  (1.1,2.3) .. (1.4,2.5);
\draw[red,thick] (1.6,2.65) .. controls (2.1,2.5) and  (2.3,2.2) .. (2.6,1);
\draw[red,thick] (1.7,0.65) .. controls (1.9,0.8) and  (2.2,0.8) .. (2.4,0.65);
\draw[red,thick] (.1,.7) .. controls (.3,2.7) and  (.4,2.7) .. (.4,2.8);
\draw (0,1.5) -- (3,1.5);

 \node[draw, rectangle,scale=.6, fill=blue!20!white] (N1) at (6,.7) {};
 \node[draw, rectangle,scale=.6, fill=blue!20!white] (N2) at (6,2.3) {};
 
 \node[draw, rectangle,scale=.6, fill=blue!20!white] (N3) at (7.5,.7) {};
 \node[draw, rectangle,scale=.6, fill=blue!20!white] (N4) at (7.5,2.3) {};

 \node[draw, rectangle,scale=.6, fill=blue!20!white] (N5) at (9,.7) {};
 \node[draw, rectangle,scale=.6, fill=blue!20!white] (N6) at (9,2.3) {};
 
 \draw[red,thick]    (N1) -- (N4);
  \draw[red,thick]    (N1) -- (N2);
   \draw[red,thick]    (N2) -- (N4);
    \draw[red,thick]    (N3) -- (N5);
     \draw[red,thick]    (N4) -- (N5);
     
     \node at (7.5,-.5) {$G'_1$};
\end{tikzpicture}

\caption{
A sketch of $E_\emptyset = E_0 \sqcup F_\emptyset \sqcup E_1$ and the corresponding tiles on the left. On the right, the first auxiliary graph $G'_1$ with edge set $F_\emptyset$.
}
\label{fig:fig7}
\end{figure}

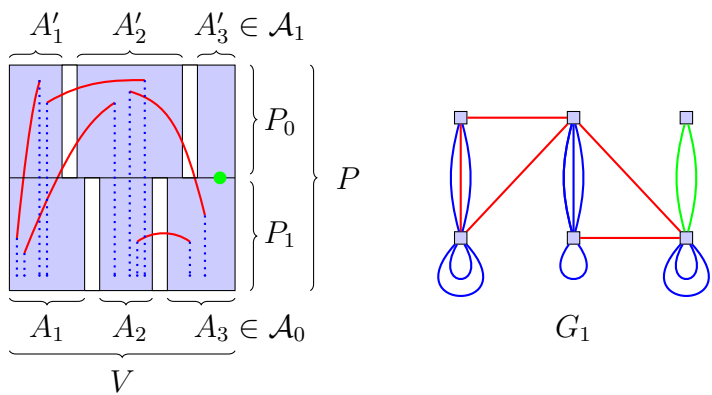
\begin{figure}[ht]
\begin{tikzpicture}[scale=1]
\draw[fill=blue!20!white] (0,0) -- (3,0) -- (3,3) -- (0,3) -- (0,0);
\draw[mybrace=0.5] (3,-.8) -- (0,-.8);
\node at (1.5,-1.2) {$V$};

\draw[mybrace=0.5] (3.2,1.45) -- (3.2,0);
\node at (3.6,2.25) {$P_0$};
\draw[mybrace=0.5] (3.2,3) -- (3.2,1.55);
\node at (3.6,.75) {$P_1$};
\draw[mybrace=0.5] (4,3) -- (4,0);
\node at (4.5,1.5) {$P$};

 \node[draw, white, rectangle,scale=.6, fill=white] (N1) at (6,.7) {};
 \node[draw, white, rectangle,scale=.6, fill=white] (N2) at (6,2.3) {};
 
 \node[draw, white, rectangle,scale=.6, fill=white] (N3) at (7.5,.7) {};
 \node[draw, white, rectangle,scale=.6, fill=white] (N4) at (7.5,2.3) {};

 \node[draw, white, rectangle,scale=.6, fill=white] (N5) at (9,.7) {};
 \node[draw, white, rectangle,scale=.6, fill=white] (N6) at (9,2.3) {};

\draw[fill=white] (1,0) -- (1.2,0) -- (1.2,1.5) -- (1,1.5) -- (1,0);
\draw[fill=white] (1.9,0) -- (2.1,0) -- (2.1,1.5) -- (1.9,1.5) -- (1.9,0);
\draw[fill=white] (.7,1.5) -- (.9,1.5) -- (.9,3) -- (.7,3) -- (.7,1.5);
\draw[fill=white] (2.3,1.5) -- (2.5,1.5) -- (2.5,3) -- (2.3,3) -- (2.3,1.5);

\draw[mybrace=0.5] (1,-.1) -- (0,-.1);
\node at (.5,-.5) {$A_1$};
\draw[mybrace=0.5]  (1.9,-.1) -- (1.2,-.1);
\node at (1.6,-.5) {$A_2$};
\draw[mybrace=0.5] (3,-.1) -- (2.1,-.1);
\node at (3.2,-.5) {$A_3 \in \script{A}_0$};

\draw[mybrace=0.5] (0,3.1) -- (.7,3.1);
\node at (.5,3.5) {$A'_1$};
\draw[mybrace=0.5]  (.9,3.1) -- (2.3,3.1);
\node at (1.6,3.5) {$A'_2$};
\draw[mybrace=0.5] (2.5,3.1) -- (3,3.1);
\node at (3.2,3.5) {$A'_3 \in \script{A}_1$};

\draw[red,thick] (.5,2.5) .. controls (1,2.8) and  (1.4,2.8) .. (1.8,2.8);
\draw[red,thick] (.2,.5) .. controls (.8,2) and  (1.1,2.3) .. (1.4,2.5);
\draw[red,thick] (1.6,2.65) .. controls (2.1,2.5) and  (2.3,2.2) .. (2.6,1);
\draw[red,thick] (1.7,0.65) .. controls (1.9,0.8) and  (2.2,0.8) .. (2.4,0.65);
\draw[red,thick] (.1,.7) .. controls (.3,2.7) and  (.4,2.7) .. (.4,2.8);
\draw (0,1.5) -- (3,1.5);

 \draw[blue,thick,dotted]  (.2,.5) --  (.2,.2);  \node[draw, blue, circle,scale=.05, fill=blue] (B5) at (.2,.2) {};
  \draw[blue,thick,dotted]  (.1,.7) --  (.1,.2);  \node[draw, blue, circle,scale=.05, fill=blue] (B5) at (.1,.2) {};
  
  \draw[blue,thick,dotted]  (2.4,.65) --  (2.4,.2);  \node[draw, blue, circle,scale=.05, fill=blue] (B5) at (2.4,.2) {};
  \draw[blue,thick,dotted]  (1.7,0.65) --  (1.7,.2);  \node[draw, blue, circle,scale=.05, fill=blue] (B5) at (1.7,.2) {};

 \node[draw, rectangle,scale=.6, fill=blue!20!white] (N1) at (6,.7) {};
 \node[draw, rectangle,scale=.6, fill=blue!20!white] (N2) at (6,2.3) {};
 
 \node[draw, rectangle,scale=.6, fill=blue!20!white] (N3) at (7.5,.7) {};
 \node[draw, rectangle,scale=.6, fill=blue!20!white] (N4) at (7.5,2.3) {};

 \node[draw, rectangle,scale=.6, fill=blue!20!white] (N5) at (9,.7) {};
 \node[draw, rectangle,scale=.6, fill=blue!20!white] (N6) at (9,2.3) {};
 
 \draw[red,thick]    (N1) -- (N4);
  \draw[red,thick]    (N1) -- (N2);
   \draw[red,thick]    (N2) -- (N4);
    \draw[red,thick]    (N3) -- (N5);
     \draw[red,thick]    (N4) -- (N5);
     
     \node at (7.5,-.5) {$G_1$};

 \draw[blue,thick]    (N1) to[out=105,in=-105] (N2);
  \draw[blue,thick]    (N1) to[out=75,in=-75] (N2);
   \draw[blue,thick]    (N3) to[out=105,in=-105] (N4);
  \draw[blue,thick]    (N3) to[out=75,in=-75] (N4);
   \draw[blue,thick]    (N3) -- (N4);

   \draw[blue,thick]    (N1) .. controls (5.6,0) and  (6.4,0) .. (N1);
    \draw[blue,thick]    (N1) .. controls (5.1,-.3) and  (6.9,-.3) .. (N1);
     \draw[blue,thick]    (N3) .. controls (7,0) and  (8,0) ..  (N3);
    \draw[blue,thick]    (N5) .. controls (8.6,0) and  (9.4,0) ..  (N5);
    
    \draw[blue,thick]    (N5) .. controls (8.1,-.3) and  (9.9,-.3) ..  (N5);
    
  \draw[blue,thick,dotted] (.4,2.8) -- (.4,.2);    \node[draw, blue, circle,scale=.05, fill=blue] (B1) at (.4,.2) {};
  \draw[blue,thick,dotted] (.5,2.5) -- (.5,.2);   \node[draw, blue, circle,scale=.05, fill=blue] (B2) at (.5,.2) {};
  
    \draw[blue,thick,dotted] (1.8,2.8) -- (1.8,.2);  \node[draw, blue, circle,scale=.05, fill=blue] (B3) at (1.8,.2) {};
  \draw[blue,thick,dotted] (1.4,2.5) -- (1.4,.2);  \node[draw, blue, circle,scale=.05, fill=blue] (B4) at (1.4,.2) {};
   \draw[blue,thick,dotted]  (1.6,2.65) --  (1.6,.2);  \node[draw, blue, circle,scale=.05, fill=blue] (B5) at (1.6,.2) {};
  
  \draw[blue,thick,dotted] (2.6,1) -- (2.6,.2);
  \node[draw, blue, circle,scale=.05, fill=blue] (B6) at (2.6,.2) {};
     


 \draw[green,thick]    (N5) to[out=105,in=-105] (N6);
  \draw[green,thick]    (N5) to[out=75,in=-75] (N6);
   \draw[blue,thick]    (N3) to[out=105,in=-105] (N4);
      \node[draw, green, circle,scale=.4, fill=green](D3) at (2.8,1.5) {};

\end{tikzpicture}

\caption{
Type-E dummy edges in blue turn $G_1'$ into an even graph, with their $\eta_1$ images drawn as dotted arcs. Type-C dummy edges in green make $G_1$ connected, with their common $\eta_1$ image being a trivial arc.\protect\footnotemark}
\label{fig:fig8}
\end{figure}
\footnotetext{We remark that for ease of formalisation, our algorithm will add additional type-C edges not drawn in this picture.}

\medskip
{\bf Building the first auxiliary graph.} For every $n \in \N$ we recursively construct decompositions $\p{G'_n,\eta'_n}$ with $G'_n$ a finite multi-graph encoding the edge patterns between the tiles at step $n$. 
So formally, the graph $G'_n$ has vertex set $V_n$ and edge set $F_n$ where 
\begin{itemize}
\item $ V_n = \set{v_{r,A}}:{r \in \script{R}(n), \, A \in \script{A}_r}$ and
\item $F_n := \bigcup \set{F_{r,A}}:{r \in \script{R}({<}n), \, A \in \script{A}_r}$.\footnote{$F_n$ should not be confused with $F_{(n)}$ where $(n)$ is a one-element sequence on the first level of $\script{R}$.}
\end{itemize}
and $\eta'_n$ is defined by  $\eta'_n \restriction F_n = \operatorname{id}$ and $\eta_n(v_{r,A}) := X_{r,A}$ for all vertices in $V_n$. Note that on our way to build a decomposition, $(G'_n,\eta_n)$ satisfies \ref{Eta1}, \ref{Eta2}, \ref{E1a} and \ref{E1b} of a decomposition according to Definition~\ref{def:Eulerdecomp}. Edge-vertex incidence in $G'_n$ is defined recursively in $n$\footnote{If one such displayed free arc $f \in F_n$ has an endpoint $(x,y) \in V \times P$ in $X$, then all vertices $v_{r,A} \in V_n$ with $y \in P_r$ and $x \in A$ are potential candidates for the corresponding endvertex of $f$ in $G'_n$. This is where we make a recursive choice.} so as to satisfy \ref{E2} and Definition~\ref{def_extendingEulerDecomp} for $F_n$. For this, observe that for every $n \in \N$ there is a natural (surjective) contraction map 
$$\varrho'_n \colon G'_{n+1} \to G'_n, \; v_{r,A} \mapsto v_{r^-,A'} \; \text{ and } \; f \mapsto \begin{cases}
f & \text{if } f \in F_{n}, \\
v_{r,A} & \text{if } f \in F_{n+1} \setminus F_n, \; f \in F_{r,A}.
\end{cases}$$
which clearly corresponds to the relation $X_{r,A} \subset X_{r^-,A'}$ where $A'$ is the unique element of $\script{A}_{r^-}$ satisfying $A' \supseteq A$. Indeed, it is straightforward to check that properties \ref{Q1} -- \ref{Q4} in Definition~\ref{defn_edgecontraction} for a contraction map are satisfied.

Since $G'_0$ is the unique edge-less graph on a single vertex, there is nothing to do. Suppose that $G'_n$ has already been defined so that \ref{E2} and Definition~\ref{def_extendingEulerDecomp} are satisfied for the finite sequence $(G'_i \colon i \leq n)$. Consider $f \in E(G'_{n+1}) = F_{n+1}$. If $f \in F_n$, and say $f_{G'_n}(0) = v_{r,A}$ for some $r \in \script{R}(n)$ and $A \in \script{A}_r$, then by our recursive assumptions we have $f(0) := (x,y) \in  A \times P_r$. Choose any $s \in r^+$ such that $y \in P_s \subset P_r$ and let $A'$ be the unique element of $\script{A}_{s}$ satisfying $A' \subset A$, and define $f_{G'_{n+1}}(0) = v_{s,A'}$. Similarly, if $f \in F_{n+1} \setminus F_n$, i.e.\ $f \in F_{r,A}$ for some $r \in \script{R}(n)$ and $A \in \script{A}_r$, then if say $f(0) := (x,y) \in V \times P$ choose any $s \in r^+$ such that $y \in P_s$ and let $A'$ be the unique element of $\script{A}_{s}$ satisfying $A' \subset A$, and define $f_{G'_{n+1}}(0) = v_{s,A'}$, and similarly for $f(1) := (x',y') \in V \times P$.

\noindent {\bf Summary:} Each $\script{D}'_n = (G'_n,\eta'_n)$ forms a decomposition of $X$ (cf.\ Definition~\ref{def:Eulerdecomp}), and $\varrho'_n \colon G'_{n+1} \to G'_n$ is an $\eta$-compatible contraction map (cf.\ Definition~\ref{defn_edgecontraction} and \ref{defn_etacompatible}). 

\medskip
{\bf Building the second auxiliary graph.}
For our second auxiliary graph $G''_n$, for each edge $e$ of $G'_n$, we will add two corresponding \emph{type-E dummy edges} $d^{e(0)}$ and $d^{e(1)}$ to $G'_n$, making sure that \ref{Eta3} and \ref{E3} are satisfied for each $\p{G''_n,\eta''_n}$. We also make sure that $\varrho'_n$ extends to a contraction map $\varrho''_n \colon G''_{n+1} \to G''_n$. 

\begin{defn}
	For $e \in E(X)$, write $e(i) = (x_{e(i)},y_{e(i)}) \in V \times P$ for its endpoints $e(0)$ and $e(1)$ in $X$. For every $e \in E(X)$, there is a unique index $m = m(e)$ such that $e \in F_{m+1} \setminus F_m$, and so there is a unique $s=s^e \in \script{R}(m)$ such that $e \in E_{s,A}$ for some $A \in \script{A}_s$. For every $k \geq m$, let $s^e(k) = s \cat 0^{k-m} \in \script{R}(k)$. Note that for every edge $e$, the set $\set{P_{s^e(k)}}:{k \geq m(e)}$ is a nested zero-sequence of subcontinua of $P$, and hence there is a unique point contained in the intersection $\bigcap_{k \geq m(e)} P_{s^e(k)}$ which we denote by $\sigma(e)$. Further, for $k \geq m$ and $i \in \Set{0,1}$, let $A^{e(i)}(k) \in \script{A}_{s^e(k)}$ be the unique element with $x_{e(i)} \in A^{e(i)}(k)$. For $e \in E$,  and $k> m(e)$ we write $v^{e(i)}(k) := v_{s^e(k),A^{e(i)}(k)} \in V_k$, and call this vertex the \emph{root vertex associated with the endpoint $e(i)$ at stage $k$.} Finally, fix arcs $\alpha^{e(i)} \subset \Set{x_{e(i)}} \times P_{s^e}$ from $e(i) = (x_{e(i)},y_{e(i)})$ to $(x_{e(i)}, \sigma(e))$ for each $e \in F_n$ and $i \in \Set{0,1}$.
\end{defn}

Define $\p{G''_n,\eta''_n}$ by adding to $G'_n$ a set of dummy edges $D''_n = \set{d^{e(0)}, \, d^{e(1)}}:{e \in F_n}$, and extend $\eta'_n$ to a map $\eta''_n$ by defining $\eta''_n(d^{e(i)}) = \alpha^{e(i)}$ on the newly added dummy edges. By construction of the arcs $\alpha$, this assignment satisfies \ref{E3} for $\eta''_n$. Further, edge-vertex incidence for type-E dummy edges in $G''_n$ is given by $d^{e(i)}_{G_n}(0) := e_{G_n}(i)$ and $d^{e(i)}_{G_n}(1) := v^{e(i)}(n)$, that is to say, the edge $d^{(e(i))}$ connects an endpoint of $e$ in $G_n$ to the root vertex associated with the endpoint at stage $n$.

Moreover, we extend the map $\varrho'_n$ to a contraction map $\varrho''_n \colon G''_{n+1} \to G''_n$ by defining 
$$\varrho''_n(d^{e(i)})=\begin{cases} \varrho'(e) & \text{if } d^{e(i)} \in D_{n+1} \setminus D_n  \\ d^{e(i)} & \text{if } d^{e(i)} \in D_n. \end{cases}$$ 

\begin{theorem}
\label{thm_EvenAuxGraph}
Each $G''_{n+1}$ is an even multi-graph, $\script{D}''_n = (G''_n,\eta''_n)$ forms a decomposition of $X$, and $\varrho''_n \colon G''_{n+1} \to G''_n$ is an $\eta$-compatible contraction map.
\end{theorem}

\begin{proof}
It is routine to check that $\script{D}''_n = (G''_n,\eta''_n)$ forms a decomposition of $X$. Moreover, the map $\varrho''_n \colon G''_{n+1} \to G'_n$ is a contraction map, because we added new dummy edges only between vertices in the same fibre of $\varrho'$. Hence, \ref{Q4} of a contraction map is still satisfied, and the other properties are inherited from $\varrho'_n$. To see that $G''_{n+1}$ is even, we make use of the following observation, which is immediate from the construction.

\smallskip
\textbf{Observation:} \emph{For every $n \in \N$, the edge set of $G''_{n}$ can be partitioned into a family of edge-disjoint trails\footnote{Recall that a \emph{trail} is a walk without repeated edges. } $\set{T_{n}(e)}:{e \in F_n}$ whose vertex-edge sequence is given by  }
$$T_{n}(e) = v^{e(0)}(n),  \, d^{e(0)}, \, e_{G_n}(0),  \, e, \, e_{G_n}(1), \, d^{e(1)}, \, v^{e(1)}(n).$$
%
We are now ready to calculate the parity of vertex degrees in $G''_n$, relying on the elementary fact that every inner vertex of a trail $T$ has even degree in the subgraph induced by $T$, and every end-vertex of an open trail $T$ (i.e.\ a trail with distinct start and end-vertices) has odd degree in the subgraph induced by $T$. So consider some vertex $v=v_{t, A} \in V(G_{n})$. Write $t = s \cat 0^j$ where $s$ does not end on zero and $j \in \N$. By Lemma~\ref{lem_evencutlowertilesplusF}, $A$ induces an even edge cut $C$ in $Z_t := X_\sim[E_{t} \sqcup \bigsqcup_{k= 0}^{j-1} F_{s\cat 0^k}]$. Furthermore, since $X_{t,A}$ with ground set $A \times P_t$ is a connected component of $X[E_{t}]$, it follows that $C \subset \bigsqcup_{k= 0}^{j-1} F_{s\cat 0^k}$.

\medskip
\textbf{Claim:} The vertex \emph{$v$ has odd degree in $T_n(e)$ if and only if $e \in C$.}
\smallskip

The claim implies the theorem, since the number of trails in which $v$ has odd degree is even. To prove the claim, note that $e \in C$ if and only if $x_{e(0)} \in A$ and $x_{e(1)} \notin A$ (or vice versa), which happens -- since $C \subset \bigsqcup_{k= 0}^{j-1} F_{s\cat 0^k}$ -- if and only if $v^{e(0)}(n) = v$ and $v^{e(1)}(n) \neq v$ (or vice versa), i.e.\ if and only if $v$ has odd degree in $T_n(e)$.
\end{proof}

\medskip
{\bf Building the Eulerian decompositions.}
To build Eulerian\index{Eulerian decomposition} (i.e.\ even and connected) graphs $G_n$ from $G''_n$ so that the maps $\varrho_n$ become edge-contractions, it now suffices to recursively add further dummy edges to $G''_{n+1}$ only between vertices of the same fibre $\varrho''^{-1}_n(v)$ such that every such fibre becomes connected. By induction, this will imply that each $G_n$ is connected.

The Eulerian decompositions $\p{G_n,\eta_n}$ are built recursively. Since $2^0 = 1$, both $G_0 = G''_0$ are the unique graph on a single vertex without loops. Now suppose $G_n$ has already been defined. Assume inductively that

\begin{enumerate}[label=($\ddagger$\arabic*)]
\item \label{ddagger1} every dummy edge $d=v_{t,A} v_{t',A'} \in E(G_n) \setminus E(G'_n)$ has an associated point $\eta(d) = \p{q_V(d),q_P(d)} \in V \times P$ which is contained in the intersection of the corresponding tiles $X_{t,A} \cap X_{t',A'}$.
\item \label{ddagger2} Moreover, assume there is an equivalence relation $\sim$ on the dummy edges in $E(G_n) \setminus E(G''_n)$ such that every equivalence class consists of precisely two dummy edges which are parallel in $G_n$.
\end{enumerate} 

To build $G_{n+1}$, first obtain a graph $G^{**}_{n+1}$ by displaying all dummy edges of $G_n$ such that \ref{ddagger1} and \ref{ddagger2} are satisfied, and so that $\varrho_n \colon G^{**}_{n+1} \to G''_n$ is a contraction map (when ambiguous, make an arbitrary choice). Note in particular that \ref{ddagger2} and the fact that $G''_{n+1}$ was even imply that $G^{**}_{n+1}$ is an even graph.

To obtain a connected even graph $G_{n+1}$ from $G^{**}_{n+1}$, first of all, for each $r \in \script{R}(n)$, let us pick a spanning tree $S_r$ for the intersection graph formed by the cover $\set{P_{r'}}:{r' \in r^+}$ on $P_r$. Next, for each edge $P_sP_{s'}$ of $S_r$ fix an arbitrary point $y_{ss'} \in P_s \cap P_{s'}$. We now add type-C \emph{dummy edges} to $G^{**}_{n+1}$ according to the following rule:
\begin{enumerate}
\item[(C)] Fix a vertex $v_{r,A} \in V_n$ with $A \in \script{A}_r$. Let $\script{B}$ denote the finite partition of $V$ which is the least common refinement of the family $\set{\script{A}_{r'}}:{r' \in r^+}$. Pick a vertex $x_B \in B$ for each $B \in \script{B}$. Now for every $x_B$ and every edge $P_sP_{s'}$ of $S_r$, we add two parallel type-C dummy edges $d_1 \sim d_2$ with the same associated point $\eta_{n+1}(d_1) = \eta_{n+1}(d_2) := (x_B,y_{ss'}) \in V \times P$ between the two vertices $v_{s,A_s}$ and $v_{s',A_{s'}}$ where $A_s$ and $A_{s'}$ are the unique elements of $\script{A}_s$ and $\script{A}_{s'}$ respectively with $B \subset A_s$ and $B \subset A_{s'}$. Finally, we extend the map $\varrho_n$ to these newly inserted edges by defining $\varrho_n(d_1) = \varrho_n(d_2) := v_{r,A}$. This arrangement for $d_1$ and $d_2$ satisfies \ref{ddagger1} and \ref{ddagger2}.
\end{enumerate}

\begin{theorem}
\label{thm_EulerdecompProductCase}
Each $G_{n+1}$ is a finite Eulerian multi-graph, $\script{D}_n = (G_n,\eta_n)$ is an Euler decomposition of $X$, and $\varrho_n \colon G_{n+1} \to G_n$ is an $\eta$-compatible edge-contraction map. Thus,  $\Sequence{\script{D}_n}:{n \in \N}$ is an approximating sequence of Eulerian decompositions for $X$.
\end{theorem}
\begin{proof}
We first show that $\varrho_n \colon G_{n+1} \to G_n$ is an edge-contraction map, i.e.\ that it has connected fibres, see \ref{Q5} of Definition~\ref{defn_edgecontraction}. Interpreted as a continuous map, this translates to the fact that $\varrho_n$ is monotone. In particular, this will imply inductively that each $G_n$ is connected: Indeed, $G_0$ is trivially connected, and if $G_n$ is connected, then it follows from the fact that since $\varrho_n \colon G_{n+1} \to G_n$ is a continuous, monotone surjective map from a compact spaces onto a connected space, then also the domain $G_{n+1}$ must be connected, see e.g.\ \cite[Theorem~6.1.29]{Engelking}.

To see that $\varrho_n$ has connected fibres, fix some $v_{r,A} \in V_n$, and consider $H := \varrho_n^{-1}(v_{r,A})$, a subgraph of $G_{n+1}$. By definition, the vertex set of $H$ is precisely the set 
$$V_H = \set{v_{s,A'}}:{s \in r^+, \, A' \in \script{A}_s}.$$
Let $C \subset V_H$ be the vertex set of a component of the graph $H$. We have to show $C = V_H$. For this, note that if $v_{s,A'} \in C$ and $v_{t,A''} \in V_H$ with $A' \cap A'' \neq \emptyset$, then $v_{t,A''} \in C$. Indeed, let $P \subset S_r$ denote the unique $P_sP_t$ path in the tree $S_r$. Fix $x_B \in B \subset A' \cap A''$. Then the dummy edges in $\eta_n^{-1}\p{\set{(x_B,y_{uu'})}:{uu' \in E(P)}} \subset H$ which have been added according to rule (C) witness connectivity between $v_{s,A'}$ and $v_{t,A''}$. 

Therefore,  
$$A_C := \bigcup \set{A'}:{v_{s,A'} \in C} \quad \text{ and } \quad A_{\neg C} := \bigcup \set{A'}:{v_{s,A'} \in V_H \setminus C} $$
gives rise to a clopen bipartition $(A_C, A_{\neg C})$ of $A$. We claim that $A_{\neg C} = \emptyset$. This would imply that $C = V_H$, proving that $H$ is connected. Otherwise, $(A_C, A_{\neg C})$ is a non-trivial clopen bipartition of $A$, and so since $X_{r,A}=\p{A \times P_r} \cup E_{r,A}$ was a Peano continuum by (TILE)(a), it follows that $E_{r,A}(A_C, A_{\neg C})$ is a non-empty edge cut of $X_{r,A}$. Pick $f$ in $E_{r,A}(A_C, A_{\neg C})$ arbitrarily. Then $f \in F_{r,A} \subset F_{n+1}$ by (EDGE)(b), and hence $f \in E(H)$. However, it now follows from \ref{E2} that $f \in E_H(C,V_H \setminus C)$, witnessing that $C$ was not maximally connected, a contradiction.

That the $\varrho_n$ are $\eta$-compatible is easily verified, and so it follows from Lemma~\ref{lem_equivalenceExtend} that $\Sequence{\script{D}_n}:{n \in \N}$ is indeed an approximating sequence of Eulerian decompositions for $X$. Note that $w(\p{D_n,\eta_n}) \to 0$ follows from COVER(b), CUT(b), and the fact that we assumed that $X$ contained no loops, implying that $\diam{X_{r,A}} \to 0$ as $|r| \to \infty$.
\end{proof}

The proof of our main result is now complete:


\begin{proof}[Proof of Theorem~\ref{thm_ProductSpaceRemainders}]
Let $X$ be a Peano continuum with $\ground{X} = V\times P$. We may assume that $X$ is a Peano graph without loops with the even-cut property, such that $P$ is non-trivial. Then by Theorem~\ref{thm_EulerdecompProductCase}, the space $X$ has an approximating sequence of Eulerian decompositions, and hence $X$ is Eulerian by Theorem~\ref{thm_MainEquivalence}.
\end{proof}

\chapter{One-Dimensional Spaces}
\label{chapter_1DimRemainders}

\section{Overview}

The purpose of this final chapter is to prove the following theorem. 

\begin{theorem}
\label{thm_blablabla}
A one-dimensional Peano continuum is Eulerian if and only if it satisfies the even-cut condition.
\end{theorem}

More precisely, using $\ref{romaniii} \Rightarrow \ref{romani}$ of Theorem~\ref{thm_MainEquivalence}, what we will show here is that every one-dimensional Peano continuum satisfying the even-cut condition admits an approximating sequence of Eulerian decompositions.


Let us briefly remark that for $n \geq 1$, the dimension of a Peano continuum $X$ is $n$ if and only if the ground space $\ground{X}$ has dimension $n$. This is a consequence of the well-known \emph{sum theorem for dimension}, \cite[Thm.\ 1.5.2]{engelkingdimension}, by applying it to $X$ considered as a countable union of $\ground{X}$ and one-cells $\closure{e}$ for $e \in E(X)$. In particular, Theorem~\ref{thm_crossingfinitearcs}\ref{THMC} is indeed equivalent to Theorem~\ref{thm_blablabla}.


\subsection{Proof strategy}

Consider a one-dimensional Peano continuum $X$ for which we aim to construct an approximating sequence of Eulerian decompositions. As described in the Blueprint~\ref{obs_blueprint}, any Peano partition $\script{U}$ for $X$ into standard subspaces gives rise to a corresponding Eulerian decomposition for $X$, provided that $X$ satisfies the even-cut condition. Note that the even-cut assumption on $X$ is a necessary one, for if $\script{U}$ displays an odd edge cut of $X$, then no such corresponding Eulerian decomposition can exist. Now if we could find a Peano partition $\script{U}$ such that each partition element $U \in \script{U}$ individually still has the even-cut property, we could continue this procedure recursively to construct an extending sequence of Eulerian decompositions (cf.\ Definition~\ref{def_extendingEulerDecomp}). 

Recall, however, that there is a second objective for constructing an approximating sequence of Eulerian decompositions: Not only should the Eulerian decompositions extend each other (property~\ref{A1} of Definition~\ref{def:approximating}), but their widths should also decrease to zero (property~\ref{A2} of Definition~\ref{def:approximating}). This second requirement, however, is at odds with our earlier idea that partition elements of $\script{U}$ individually always continue to have the even-cut property, as the even-cut property generally prohibits single edges to be displayed (cf.\ Blueprint~\ref{obs_blueprint}), and so the width of our recursively constructed decompositions will be bounded from below by the diameter of the largest edge. 

We resolve these issues by the following approach: given $X$, we construct in Theorem~\ref{thm_arrangecuts} a Peano partition $\script{U}$ into standard subspaces of $X$ such that each partition element $U \in \script{U}$ individually still has the even-cut property, and so that each $U$ contains a finite set of edges $F_U$ such that each component of $U - F_U$ has somewhat smaller diameter than $X$. Then the partition $\script{U}'$ consisting of the components of $U-F_U$ for $U \in \script{U}$ and individual edges in $\bigcup_{U \in \script{U}} F_U$ gives rise to an Eulerian decomposition of smaller width as desired. And the fact that each $U$ satisfied the even-cut condition leaves enough traces in $U-F_U$ (almost all vertices of $U_\sim - F_U$ have even degree) so that we may continue the recursive construction, see Theorem~\ref{sec_54}.

{We remark that the assumption that $X$ is one-dimensional is essentially only used in a black box result due to Bing, namely that in this case there exists a decreasing sequence of
$1/n$-Peano partitions $\set{\script{U}_n}:{n \in \N}$ such that $\partial U$ is zero-dimensional for all $U \in \script{U}_n$ and all $n$, Theorem~\ref{thm_nicebrickpart}.}

Before we come to these results, we gather in Section~\ref{sec_52} a number of auxiliary results whose purpose is first to set up the language for arranging the even-cut property in terms of inverse limits, and second to deal with the fact that edges of some partition element $U \in \script{U}$ are not a priori edges of $X$, which requires us to generalise our concept of ground space and edges.

\section{Admissible Vertex Sets and Combinatorial Alignment}
\label{sec_52}
\index{admissible vertex set|(}

\subsection{Admissible vertex sets}
In the introduction, we stated in Sections~\ref{subsec_conj} and \ref{edge_cuts_degree} the even-cut condition for the class of Peano continua $X$ in terms of their ground spaces $\ground{X}$. For this chapter we generalise these notions in two directions: first, we generalise the notion of ground space to that of \emph{admissible vertex sets}, and second we extend the class of spaces $X$ we consider from Peano continua to a broad class of (metrisable) compacta --~which we call \emph{component-wise aligned compacta}.

To justify our first generalisation, recall that there is a standard fuzziness in the transition between combinatorial and topological graphs in the sense that degree-two vertices in combinatorial graphs are disregarded in the corresponding topological graph. {Indeed, the ground space of a finite graph $G=(V,E)$, according to our definition, will consist precisely of those vertices in $V$ that have degree at least $3$.}
This fuzziness {between vertices  and ground spaces} is even more pronounced in the case of graph-like spaces: {Recall that \emph{graph-like compactum} is an object $X=(V,E)$ where $X$ is a (metrizable) compact space, $V \subseteq X$ is a closed zero-dimensional subset such that $X \setminus V = \bigoplus_{e \in E} (0,1)$ is homeomorphic to a topological sum of intervals. Clearly, for continua this more combinatorial definition from \cite{euleriangraphlike,thomassenvella} describes the same class of spaces as our earlier definition of a graph-like continuum $X$ which required  $\ground{X}$ to be zero-dimensional -- however, similar to the situation just seen above for graphs, the combinatorial definition allows vertices of $V$ to subdivide free arcs of $X$, i.e.\ the inclusion in $\ground{X} \subseteq V$ might be proper.} For example, both $V=\Set{0,1}$ and $V$ the middle third Cantor set can function as vertex set of a graph-like continuum homeomorphic to the unit interval $I$.  {For additional information about graph-like spaces and their inverse limit descriptions see Section~\ref{sec_recapgraphlike}.}

{So far, taking the topological viewpoint of ground spaces was sufficient and even beneficial for our purposes. However, note that forgetting this extra combinatorial information has some undesirable side effects:} if $H = (V_H,E_H)$ and $G=(V_G,E_G)$ are combinatorial graphs such that $H$ is a subgraph of $G$, then their combinatorial structures are naturally aligned in the sense that $V_H \subset V_G$ and $E_H \subset E_G$. However, viewing $H$ and $G$ as topological spaces, the free arcs of $H$ might be strict supersets of the free arcs of $G$, with the consequence that $E(H)$ might not be a subset of $E(G)$ according to our topological definition. 
In this chapter, we set up the language for eliminating this imprecision.

 
\begin{defn}[Admissible vertex set]
A compact subset $V \subset X$ of a Peano continuum $X$ is an \emph{admissible vertex set}\index{admissible vertex set|textbf} provided that
 $\ground{X} \subset V$ and $V \setminus \ground{X}$ is zero-dimensional. For an admissible vertex set $V$, the space $X \setminus V$ is homeomorphic to a disjoint sum of open intervals, which we call the edges of $X$ associated with $V$, written $E(X{:}V)$.\index{E(X:V)@$E(X{:}V)$|textbf}
\end{defn}

This definition is equivalent to saying that $\ground{X}$ is a subset of $V$, and for every free arc $e$ of $X$, we have that $\closure{e}$ is a graph-like space homeomorphic to an interval {or simple closed curve} with zero-dimensional vertex set $(V \cap \closure{e})$.

{We also write $(X{:}V)$ for a Peano continuum with a fixed admissible vertex set $V$ in mind.} Note that the edges $E(X{:}V)$ are the connected components of $X \setminus V$. Since $\ground{X} \subset V$ and $V$ is closed, it follows that every edge is homeomorphic to an open interval. Moreover, if $X$ is a Peano graph (so $E(X)$ is dense in $X$), then also $(X{:}V)$ is a Peano graph in the sense that the edges $E(X{:}V)$ are dense in $X$. Moreover, we may generalise the notion of edge cuts from $(X{:}\ground{X})$ to $(X{:}V)$: an edge cut of $(X{:}V)$ is the set of edges {$E(A,B)=E_X(A,B)\subseteq E(X{:}V)$} crossing a clopen partition $V = A \oplus B$. It is straightforward to check that all results about edge cuts from Section~\ref{edge_cuts_degree} still apply in this slightly generalised setting.
{Let us denote by $(X{:}V)_\sim$ the graph-like continuum obtained from $X$ by contracting each connected component of $V$ to a point and choosing them as the vertex set for the resulting graph-like continuum.} Finally, we also extend  Definition~\ref{def_standardsubspace} of a standard subspace \index{standard subspace|textbf} to this generalised setting, and call a subspace $Y \subset X$  \emph{standard in $(X{:}V)$} if, for every $e \in E(X{:}V)$, whenever $e$ meets $Y$ then $e$ is a subset of $Y$.


\begin{lemma}
\label{lem_EvenCutIndAdmissible}
	Let $X$ be a Peano continuum and $V \subset X$ be any admissible vertex set. 
{Then the following are equivalent:
\begin{enumerate}
    \item $X$ has the even-cut property,
    \item $(X{:}V)$ has the even-cut property,
    \item $(X{:}\ground{X})_\sim$ has the even-cut property,
    \item $(X{:}V)_\sim$ has the even-cut property, and
    \item $X_\sim$ has the even-cut property.
\end{enumerate}
}
\end{lemma}

\begin{proof}
   {Note that the graph-like continua listed in (3), (4) and (5) are all homeomorphic as topological spaces. In particular, if one of them is Eulerian, then all of them are. Since for graph-like continua, we already know that being Eulerian is equivalent to having the even-cut property, \cite{euleriangraphlike}, it follows that $(3)-(5)$ are equivalent. }
   
   {Next, observe that $X$ and $(X{:}\ground{X})_\sim$ have the same edge set and the same edge cuts, so (1) and (3) are equivalent. Similarly, $(X{:}V)$ and $(X{:}V)_\sim$ have the same edge set and the same edge cuts, so (2) and (4) are equivalent. }
\end{proof}

\begin{lemma}
\label{lem_EvenCutImpliesStandard}
If $X$ is a Peano continuum and $V \subset X$ an admissible vertex set for $X$, then any non-trivial Peano subcontinuum $Y \subset X$ satisfying the even-cut condition is standard in $(X{:}V)$.
\end{lemma}

\begin{proof}
Note first that if $Y$ satisfies the even-cut condition, then any free arc of $Y$ lies on a simple closed curve of $Y$ (cf.\ \cite[Lemma~16]{euleriangraphlike}), and second, that any simple closed curve in $X$ is necessarily a standard subspace of $X$ (cf.\ \cite[Lemma~5]{euleriangraphlike}).
\end{proof}

\subsection{Combinatorial alignment}
\index{combinatorial alignment|(}

To facilitate comparing edges across different spaces, from now on we will work with admissible vertex sets instead of ground sets. 

\begin{defn}[Combinatorial alignment]
Suppose that $Y \subset X$ are Peano continua, and that $V_X$ and $V_Y$ are admissible vertex sets for $X$ and $Y$ respectively. We say that $(Y{:}V_Y)$ is \emph{combinatorially aligned}\index{combinatorial alignment|textbf} in $(X{:}V_X)$ if for every $e \in E(Y{:}V_Y)$, either $e \in E(X{:}V_X)$ or $e \subset V_X$. In this situation, write $E(Y{:}V_Y) = E^{\text{real}}_Y \sqcup E^{\text{fake}}_Y$ with $E^{\text{real}}_Y := E_Y \cap E(X{:}V_X)$ for the bipartition into real and fake edges. Finally, we say a combinatorially aligned continuum $(Y{:}V_Y) \subset (X{:}V_X)$ is \emph{faithfully}\index{faithful alignment|textbf} aligned if $E(Y{:}V_Y) \subset E(X{:}V_X)$, i.e.\ if $E^{\text{fake}}_Y = \emptyset$.
\end{defn}

As an example for combinatorial alignment, consider again the two simple closed curves $C_1$ and $C_2$ inside the hyperbolic tree $Y$ from Figure~\ref{fig:Euleriantraces} in Chapter~\ref{chapter_Eulerianmaps}.
In both cases, the red simple closed curves enter and leave the hyperbolic boundary circle fairly often, so need to be subdivided accordingly, in order to ensure that their combinatorial structure matches up. {With a suitable admissible vertex set, we see that $C_1$ will be combinatorially aligned but not faithfully aligned (as non-trivial segments of $C_1$ run inside the ground space of $Y$), whereas $C_2$ will even be faithfully aligned.}
In the same vein, note that $\ground{Y} \cap C_1$ is not an admissible vertex set for $C_1$.

\begin{lemma}
\label{lem_StandardImpliesCombAligned}
Suppose $X$ is a Peano continuum and $V \subset X$ an admissible vertex set for $X$. Suppose $Y \subset X$ is a standard Peano subcontinuum. Then there is an admissible vertex set $W$ for $Y$ such that $(Y{:}W)$ is combinatorially aligned in $(X{:}V)$.
\end{lemma}

\begin{proof}
Consider an edge $e \in E(Y{:}\ground{Y})$, that is to say, a free arc in $Y$. We show that we can subdivide $\closure{e}$ by a compact zero-dimensional vertex set $W_e$ such that every segment of $\closure{e} \setminus W_e$ is either an edge of $(X{:}V)$ or is completely contained in $V$. 

Consider $I_e:=\set{f \in E(X{:}V)}:{f \cap e \neq \emptyset} = \set{f \in E(X{:}V)}:{f  \subset e }$, by the fact that $Y$ is standard in $(X{:}V)$. So $I_e$ is a collection of disjoint open intervals of $e$. Define $W_e = \Set{e(0),e(1)} \cup \closure{\bigcup I_e} \setminus \bigcup I_e$. It is easy to verify that $W_e$ is as desired.

Finally, let $W:= \ground{Y} \cup \bigcup \set{W_e}:{e \in E(Y{:}\ground{Y})}$. Since $\set{W_e}:{e \in E(Y{:}\ground{Y})}$ is a zero-sequence of closed sets all intersecting the closed set $\ground{Y}$, it follows from standard arguments (see, for example, the proof of \cite[A.11.6]{mill}) that $W$ is closed in $Y$, hence compact. By the  sum theorem of dimension, \cite[Thm.\ 1.5.2]{engelkingdimension}, $W \setminus \ground{Y} \subset \bigcup \set{W_e}:{e \in E(Y{:}\ground{Y})}$ is zero-dimensional, and so $W$ is admissible.
\end{proof}

\begin{cor}
\label{cor_EvencutImpliesCombAligned}
Suppose $X$ is a Peano continuum and $V \subset X$ an admissible vertex set for $X$. Suppose $Y \subset X$ is a non-trivial Peano subcontinuum satisfying the even-cut condition. Then there is an admissible vertex set $W$ for $Y$ such that $(Y{:}W)$ is combinatorially aligned in $(X{:}V)$.
\end{cor}

\begin{proof}
Combine Lemmas~\ref{lem_EvenCutImpliesStandard} and \ref{lem_StandardImpliesCombAligned}.
\end{proof}

Finally, we prove a lemma giving a necessary condition when the even-cut condition is preserved under unions. This lemma can be seen as the dual statement to Lemma~\ref{lem_pastingedgewiseEulermaps}. A word of explanation and warning  about the term `edge-disjoint'. Given a Peano continuum $(X{:}V)$ and two combinatorially aligned subspaces $(Y{:}V_Y)$ and $(Z{:}V_Z)$ of $X$, we say that $(Y{:}V_Y)$ and $(Z{:}V_Z)$ are \emph{edge-disjoint}, or more precisely \emph{$E(X)$-edge-disjoint}, if $E^{\text{real}}_Y \cap E^{\text{real}}_Z = \emptyset$, that is to say if each edge of $(X{:}V)$ is contained in at most one of $Y$ or $Z$. In particular, note it may happen that \emph{fake} edges of $Y$ and $Z$ meet non-trivially.

\begin{lemma}
\label{lem_evencutUnions}
Let $(X_n)_{n \in \N}$ be a sequence of non-trivial $E(P)$-edge disjoint Peano subcontinua of a Peano continuum $P$ such that $\bigcup_{n \in \N} X_n$ {covers all edges of $P$}. If each $X_n$ satisfies the even-cut condition, then so does $P$.
\end{lemma}

\begin{proof}
By Corollary~\ref{cor_EvencutImpliesCombAligned}, we may assume without loss of generality that each $X_n$ is combinatorially aligned with $(P{:}\ground{P})$. We claim that 
	\begin{align}
		\label{eq1}
		E(P) = \bigsqcup_{n \in \N} E^{real}(X_n).
	\end{align}
Since the $X_n$ are pairwise $E(P)$-edge disjoint, the sets in $\set{E^{real}(X_n)}:{n \in \N}$ are pairwise disjoint, and hence the union is a disjoint union. The inclusion $\supseteq$ is immediate from the definition of being combinatorially aligned. For the reverse inclusion, consider any edge $e \in E(P)$. Since ${E(P) \subset} \bigcup_{n \in \N} X_n$ we may assume without loss of generality that $e \cap X_0 \neq \emptyset$, and so $e \subset X_0$, and so $e$ has non-trivial intersection with an edge $e' \in E(X_0)$. But since $X_0$ was combinatorially aligned with $P$, it follows that $e = e'$. 

Next, note that quite similarly, one obtains
	\begin{align}
		\label{eq2}
		\ground{X_n} \subset \ground{P}
		\end{align}
for all $n \in \N$. Indeed, the previous argument shows that if $x$ is an interior point of some edge $e \in E(P)$ and $x \in X_n$ then $e \in E(X_n)$. 
		
	Now in order to show that also $P$ satisfies the even-cut condition, consider an arbitrary separation $A \oplus B$ of $\ground{P}$. Our task is to show that $E_{P}(A,B)$ is even. First, note that by (\ref{eq2}), the separation $A \oplus B$ induces separations of $\ground{X_n}$ for each $n \in \N$. Moreover, since $|E_{P}(A,B)|$ is finite, it follows from (\ref{eq1}) that there is $N \in \N$ such that
	$$E_{P}(A,B) = E^{\text{real}}_{X_1}(A,B) \sqcup E^{\text{real}}_{X_2}(A,B) \sqcup \cdots \sqcup E^{\text{real}}_{X_N}(A,B).$$
	Next, we claim that $E^{\text{real}}_{X_n}(A,B) = E_{X_n}(A,B)$ for all $n \in \N$, {i.e.\ that the only edges of $X_n$ crossing the topological partition $A \oplus B$ are real edges of $X_n$.} Indeed, since any fake edge $d \in E\p{X_n}$ is a subset of $\ground P$ by the property of being combinatorially aligned, it follows from connectedness that $d$ is contained completely on one side of the separation $A \oplus B$ of $\ground{P}$, and so $d \notin  E_{X_n}(A,B)$, establishing the claim. Thus, we have
	$$E_{P}(A,B) = E_{X_1}(A,B) \sqcup E_{X_2}(A,B) \sqcup \cdots \sqcup E_{X_N}(A,B),$$
and so $E_{P}(A,B)$ is the disjoint union of finitely many sets of even cardinality, and hence is an even edge cut. (Recall that by Lemma~\ref{lem_EvenCutIndAdmissible}, the even-cut property is independent of the choice of admissible vertex sets.) Since $E_{P}(A,B)$ was arbitrary, we have established that $P$ satisfies the even-cut condition.
\end{proof}

\subsection{Combinatorially aligned spanning trees.}

From Lemma~\ref{lem_findinggraphlikeswithtargets} we know that in a Peano continuum $X$, for every zero-dimensional compact set $Y \subset \ground{X}$, there exists a standard graph-like continuum $Z \subset X$ with $Y \subset Z$. Suppose $V$ is an admissible vertex set of $X$. Then the same proof shows that for every zero-dimensional compact set $Y \subset V$, there exists a standard graph-like continuum $Z \subset (X{:}V)$ with $Y \subset Z$. 

A natural question is whether there also is a faithfully aligned graph-like continuum $Z=(V_Z,E_Z)$ with $Y \subset Z$. To see that this is not always possible, consider a Peano graph $X$ consisting of a dense zero-sequence of loops attached to ground space $I$. If $Y = \Set{0,1} \subset I = \ground{X}$ say, then it is not possible to find a graph-like continuum $Z=(V_Z,E_Z)$ with $Y \subset Z$ and $E_Z \subset E(X)$. However, if we only insist on combinatorially aligned, then the answer is in the affirmative.

\begin{lemma}
\label{lem_findinggraphlikeswithtargets2}
Suppose $X$ is a Peano continuum and $V \subset X$ an admissible vertex set for $X$. For every zero-dimensional compact set $Y\subset V$, there exists a combinatorially aligned graph-like tree $T=(V_T,E_T)$ such that $Y \subset V_T$.
\end{lemma}

\begin{proof}
By Lemma~\ref{lem_findinggraphlikeswithtargets}, there exists at least one standard graph-like continuum in $X$ covering $Y$. Take an inclusion-minimal such graph-like continuum $T$ -- by Lemma~\ref{lem_spanningtrees}, this will be a standard graph-like tree. 
By Lemma~\ref{lem_StandardImpliesCombAligned}, for the standard subspace $T$ there is an admissible vertex set $V_T$ such that $(T{:}V_T)$ is combinatorially aligned with $(X{:}V)$. Note that in this case we necessarily have $Y \subset V_T$.
\end{proof}

\subsection{Component-wise aligned compacta and sparse edge sets}
We now come to the second of our extensions where we extend the class of space we consider from Peano continua to so-called component-wise aligned compacta. Observe that the ground space $\ground{X} := X - E(X)$ defined as the complement of all free arcs is well-defined for an arbitrary (metrisable) compactum $X$.

\begin{defn}
A compact space $X$ is said to be \emph{component-wise aligned}\index{component-wise aligned|textbf} if the components of $X$ form a null-family of Peano continua, and $V_Y := \ground{X} \cap Y$ is an admissible vertex set for every component $Y$ of $X$. 
\end{defn}

For a component-wise aligned compactum $X$, note that by definition, we have $E(X) = \bigsqcup \set{E(Y{:}V_Y)}:{Y \text{ a component of } X}$. In particular, we have $\ground{X} = \bigcup V_Y$. Next, the definition of an admissible vertex set generalises naturally to component-wise aligned compacta $X$: $V \subset X$ is admissible if $\ground{X} \subset V$ and $V \setminus \ground{X}$ is zero-dimensional. As before, this allows us to define edge-cuts for $(X{:}V)$ in terms of edges crossing a clopen partition of $V$ for all component-wise aligned compacta $X$ and admissible vertex sets $V$ of $X$. It is straightforward to check that all results about edges and edge-cuts from Section~\ref{edge_cuts_degree} still apply in this slightly generalised setting. In particular, it follows from the fact that each $E(Y{:}V_Y)$ is a zero-sequence and the fact that the components $Y$ of $X$ form a null-family, that $E(X)$ is a zero-sequence, and so all edge-cuts in a component-wise aligned compactum are finite. 

\begin{lemma}
\label{lem_compnentwiseevencut}
A component-wise aligned compactum $X$ has the even-cut property if and only if every component of it has the even-cut property.
\end{lemma}

\begin{proof}
The forward implication follows from Lemma~\ref{lem_evencutUnions} {by considering the countably many components containing an edge of $X$}. 

Conversely, suppose that $X$ is a component-wise aligned compactum which has the even-cut property and let $Y$ be a {non-trivial} component of $X$. So let $(A,B)$ be a clopen partition of $\ground{Y}$ and consider the corresponding finite edge cut $D=E_Y(A,B)$. Then $X[A]=Y[A]$ and $X[B]=Y[B]$ are disjoint compact subsets of $X-D$, and each a union of components of $X-D$. Using the \v{S}ura-Bura Lemma {(see \cite[Theorem~6.1.23]{Engelking})}, there is a clopen partition $U \oplus W$ of $X-D$ such that $X[A] \subset U$ and $X[B] \subset W$. But this means that $D = E_X[U \cap \ground{X},W \cap \ground{X}]$, and so $D$ is even by assumption on $X$.
\end{proof}

Finally, let us see three natural examples of component-wise aligned compacta $X$. 

\begin{lemma}
\label{lem_componentwisealignedgraph-like}
Every graph-like compactum is component-wise aligned. 
\end{lemma}

\begin{proof}
The fact that the components of a graph-like compactum form a null-sequence is tantamount to saying that graph-like continua are \emph{finitely Souslian}, which is well-known, cf.~\cite[\S2.2]{euleriangraphlike}. Moreover, since the ground-space of a compact graph-like continuum is zero-dimensional, each $V_Y$ is zero-dimensional, and it follows readily that $(Y{:}V_Y)$ is a graph-like continuum with vertex set $V_Y$.
\end{proof}

\begin{lemma}
\label{lem_componentwisealignedlocconn}
Every locally connected compactum is component-wise aligned. 
\end{lemma}

\begin{proof}
{In this case, there are only finitely many components.}
\end{proof}

Recall that an edge set is sparse if it induces a graph-like subspace.\index{edge set!sparse}

\begin{lemma}
\label{lem_evencutpropofcomponents}
Let $X$ be a Peano continuum with admissible vertex set $V$, and $F \subset E(X{:}V)$ be a sparse edge set. Then $Y=X-F$ is a component-wise aligned compactum. More precisely:
\begin{enumerate}
    \item $V$ is an admissible vertex set for $Y$, and $(Y{:}V)$ is faithfully aligned in $(X{:}V)$, and
    \item for every component $Z$ of $Y$, we have that $(Z{:}V_Z)$, for $V_Z:=V\cap Z$, is faithfully aligned in $(Y{:}V)$, and hence in $(X{:}V)$.
\end{enumerate}
\end{lemma}

\begin{proof}
(1) 
Clearly, we have $\ground{Y} \subset \ground{X} \subset V$. Hence, it remains to show that $V \cap \closure{e}$ is compact zero-dimensional for every $e \in E(Y)$. Suppose not. Then there is a free arc $e\in E(Y)$ such that $\closure{e} \cap V$ is not zero-dimensional, so there exists a non-trivial subarc $\alpha \subset e \cap V$. Since $F$ is sparse, $\closure{F} \cap V$ is zero-dimensional, $\alpha \setminus \closure{F}$ is an open subset of $X$ consisting of intervals. But then any such interval is open in $X$ but completely contained in $V$, contradicting the fact that $V$ is admissible for $X$.

In particular, $E(Y{:}V) = E(X{:}V) \setminus F$, and hence $(Y{:}V)$ is faithfully aligned in $(X{:}V)$.

(2) Let $Z$ be a component of $Y$. The argument that $V_Z = V \cap Z$ is an admissible vertex set for $Z$ is analogous to the previous case. To see that each $(Z{:}V_Z)$ is faithfully aligned in $(Y{:}V)$, consider an edge $e \in E(Z{:}V_Z)$. We need to show that $e$ is open in $Y$. Otherwise, there is a sequence of points $z_n \in Y \setminus Z$ such that $z_n \to z \in e$. Without loss of generality, we may assume that $z_n \in Z_n$ is contained in  components $Z_n$ of $Y$ which are pairwise distinct. Let $x_n \in V \cap Z_n$ arbitrary. Since by Lemma~\ref{lem_sparseproperties}\ref{lem_removingzerosequences} the non-trivial components of $Y$ form a zero-sequence, it follows that $x_n \to z$ as well. However, since $z \notin V$, this contradicts the fact that $V$ is closed.
\end{proof}

\subsection{Circle decompositions}
 Recall that the edge set of a Peano continuum $X$ can be \emph{decomposed into edge-disjoint circles}\index{decomposition into edge-disjoint circles} if there is a collection of edge-disjoint copies of $S^1$ contained in $X$ such that each edge of $X$ is contained in precisely one such circle. We stress that this collection of copies of $S^1$ is not required to cover all of $X$, as this may be impossible even for graph-like continua, see \cite[Example~4]{euleriangraphlike}. This example also shows that any two circles in such a circle decomposition may be disjoint in $X$.

Applying the results previously obtained in this section, we are now ready to prove the following result announced in Section~\ref{subsec_cycledecomp} of the introduction:

\begin{theorem}
\label{thm_circledecomposition}
A Peano continuum has the even-cut property if and only if its edge set can be decomposed into edge-disjoint circles.
\end{theorem}

\begin{proof}
Our proof generalises the corresponding proof for countable graphs due to Nash-Williams \cite{NW}. For the reverse  implication, let $\set{S_n}:{n \in \N}$ be a collection of edge-disjoint simple closed curves in $X$ together covering all edges of $X$, each of which we may assume to be combinatorially aligned in $X$ by Corollary~\ref{cor_EvencutImpliesCombAligned}. Then each $S_n$ satisfies the even-cut condition, and the assertion now follows by Lemma~\ref{lem_evencutUnions}.

For the forward implication, fix an enumeration of the edge set of $X$ which is possible by Lemma~\ref{lem_removing edges}(c). We will find the circle decomposition recursively in countably many steps. Suppose inductively that we have already selected edge-disjoint, combinatorially aligned simple closed curves $S_1,\ldots,S_n$ in $X$ so that the first $n$ edges in our enumeration of $E(X)$ are covered. Since $F_n = \bigcup_{i \in [n]} E^{\text{real}}(S_i)$ is {a finite union of sparse sets, so itself} sparse, the space $X-F_n$ is a component-wise aligned compactum by Lemma~\ref{lem_evencutpropofcomponents}. Now consider the first edge $e$ in our enumeration of $E(X)$ not already covered by the previously selected simple closed curves (if there is no such edge, we are done). Then $e$ is an edge of some faithfully aligned component $Z$ of $X-F_n$. Since each $S_i$ for $i \in [n]$ meets each edge cut of $X$ in an even number of edges, it follows from {Lemma~\ref{lem_evencutpropofcomponents}} that $X-F_n$ has the even-cut property, and hence so does the Peano continuum $Z$ by Lemma~\ref{lem_compnentwiseevencut}. Therefore, removing $e$ does not disconnect $Z$, and we may select an $e(0)-e(1)$-arc $\alpha_e$ in $Z-e$. Then $S_{n+1} = \alpha_e \cup e$ is a simple closed curve covering $e$, which we may assume to be combinatorially aligned in $X$ by Corollary~\ref{cor_EvencutImpliesCombAligned}. Moreover, $S_{n+1}$ is edge-disjoint to all previously selected simple closed curves, completing the induction step. After countably many steps no uncovered edges of $X$ remain, giving a circle decomposition of $X$.
\end{proof}
\index{admissible vertex set|)}
\index{combinatorial alignment|)}

\section{Ensuring the Even-Cut Condition}

\subsection{Inverse limit representations of graph-like compacta}
\label{sec_recapgraphlike}

In this section, we briefly recall inverse limit techniques to deal with graph-like compacta \index{graph-like space|(} and the even-cut condition from \cite{euleriangraphlike}. For an extensive discussion of inverse limits of finite multi-graphs, the reader may consult \cite[\S 8.8]{Diestel} and \cite{euleriangraphlike}.

For general background on inverse limits of compact Hausdorff spaces over directed sets, see \cite[\S2.5 and 3.2.13ff]{Engelking}. For an introduction to inverse limit sequences, that is to say, inverse limits where the underlying directed set is $(\N,<)$, see \cite[Chapter II]{Nadler}. 

Let $X$ be a component-wise aligned compactum with admissible vertex set $V$. By subdividing edges once, if necessary, we may assume that every edge of $X$ has two distinct endpoints in $V$, so that $X$ is \emph{simple}. A \emph{clopen partition}\index{clopen partition} of $V$ is a partition $\mathcal{U}=\{U_1, U_2, \ldots , U_n\}$ of $V$ into pairwise disjoint clopen sets. Write 
$$E(\script{U}) = \bigcup_{i \in [n]} E(U_i,V \setminus U_i)$$ 
for \index{E(U)@$E(\script{U})$} the (finite) set of all cross edges of the finite partition $\script{U}$. Recall that $X[U_i]$\index{X[U]@$X[F]$ \& $X[U]$} denotes the space $U_i$ together with all edges from $X$ that have both their endpoints in $U_i$.

Next let $\Pi=\Pi(V)$\index{Pi(V)@$\Pi(V)$} be the set of all clopen partitions of $V$. The refinement relation naturally turns $(\Pi,\preccurlyeq)$ into a directed set. Now given $(X{:}V)$ and $\script{U} \in \Pi(V)$, the \emph{multi-graph associated with $\mathcal{U}$}\index{multi-graph associated with $\mathcal{U}$|textbf} for some $\script{U} \in \Pi$ is the finite graph $X_\mathcal{U}$
with vertex set $\script{U}$ and edge set $E(\script{U})$ of all cross edges of the finite partition with the natural edge-vertex incidence. Formally, we set  $X_\mathcal{U} =X / \{X[U] \colon U \in \mathcal{U}\}$. If $\pi_\mathcal{U} \colon X \to X_\mathcal{U}$ denotes the quotient mapping from $X$ to the multi-graph associated with $\mathcal{U}$, then $\pi_\mathcal{U}$ is a contraction map (however, if some $X[U_i]$ is not connected, then $\pi_\mathcal{U}$ is not an edge--contraction map).

Whenever $p \geq q \in \Pi(V)$, there are natural bonding maps $f_{pq} = \pi_q \circ \pi^{-1}_p \colon X_p \to X_q$. These maps send vertices of $X_p$ to the vertices of $X_q$ that contain them as subsets; they are the identity on the edges of $X_p$ that are also edges of $X_q$; and they send any other edge of $X_p$ to that vertex in $X_q$ containing both its endpoints. In other words, each bonding map is a contraction map. Also, these maps are compatible in the inverse limit sense (whenever $p\geq q \geq r$ then $f_{pr} = f_{pq} \circ f_{qr}$), and hence $\Sequence{X_p}:{ p \in \Pi}$ forms an inverse system. 

We now have the following facts (compare to \cite[Theorem 13]{euleriangraphlike}.) 

\begin{itemize}
    \item For any component-wise aligned compactum $X$, we have $X_\sim \cong \varprojlim \Sequence{X_p}:{ p \in \Pi}$.
  \item $X$ (or equivalently $X_\sim$) satisfies the even-cut condition if and only if every $X_p$ satisfies the even-cut condition if and only if every $X_p$ is an even graph. \index{even graph}
\end{itemize}

Indeed, to see this, note that for any admissible vertex set $V$ of $X$ there is a natural surjection $f \colon X \to  Y:= \varprojlim\Sequence{X_p}:{ p \in \Pi(V)}$ defined by $f(x) := (\pi_p(x) \colon p \in \Pi(V))$. By \cite[3.2.11]{Engelking}, it follows that $Y$ is homeomorphic to the quotient $X / \set{f^{-1}(y)}:{y \in Y}$. But the non-trivial fibres of $f$ correspond precisely to the non-trivial components of $\ground{X}$, and hence $X_\sim \cong \varprojlim \Sequence{X_p}:{ p \in \Pi}$ as desired.

We conclude this brief recap with an alternative description for component-wise aligned compacta $X$ with only finitely many components (which is equivalent to saying they are locally connected). So let $X$ be a locally connected compactum, and $V$ an admissible vertex set for $X$. Let $\script{E} = ([E(X{:}V)]^{< \infty}, \subset)$ denote the collection of finite edge sets of $(X{:}V)$, directed by inclusion. For $F \in \script{E}$, the space $X - F$ has finitely many components, listed as say $V_F= \Set{C_1,\ldots,C_k}$ by Lemma~\ref{lem_removing edges}. The \emph{contraction of $X$ onto $F$}\index{contraction of $X$ onto $F$|textbf}, denoted by $X.F$\index{X.F@$X.F$|see {contraction of $X$ onto $F$|textbf}}, is the finite multi-graph with vertex set $V_F$ and edge set $F$, where an edge in $F$ goes between those components in $V_F$ that contain its endpoints in $X$. Formally, $X.F =X / V_F$ 
is defined as the topological quotient of $X$ into the finitely many closed sets of $V_F$ and points of $\bigcup F$. Note that if $\pi_F \colon X \to X.F$ denotes the quotient mapping from $X$ to the multi-graph $X.F$, then $\pi_F$ is an edge-contraction map. The notation $X.F$ is taken from the same concept in matroid theory, see for example \cite[Chapter 3]{oxley}. Contrary to the graphs $X_\script{U}$ from above, the graphs $X.F$ may  contain loops.

\begin{itemize}
    \item For any locally connected compactum $X$, we have
    $X_\sim = \varprojlim \Sequence{X.F}:{ F \in \script{E}}$.
\item $X$ (or equivalently $X_\sim$) satisfies the even-cut condition if and only if every $X.F$ satisfies the even-cut condition if and only if every $X.F$ is an even graph.
\end{itemize}

The proof of the first fact can be derived from the previous inverse limit description as follows: if $X$ is locally connected, and $V$ an admissible vertex set for $U$, then pick a cofinal, refining sequence $\Sequence{\mathcal{U}_n}:{n \in \N}\subset \Pi(V)$ such that $X[U]$ is connected for all $U \in \script{U}_n$ and $n \in \N$. Then $\Sequence{E(\mathcal{U}_n)}:{n \in \N} $ is cofinal in $\script{E}$, and furthermore, it is clear from the definitions that $X_{\script{U}_n} = X.E(\mathcal{U}_n)$ and that the bonding maps agree. Thus, using the fact that inverse limits of cofinal subsystems agree, it follows that for locally connected compacta $X$, we have
$${X_\sim} = \varprojlim \Sequence{X_p}:{ p \in \Pi} = \varprojlim \Sequence{X_{\mathcal{U}_n}}:{n \in \N} = \varprojlim  \Sequence{X.E(\mathcal{U}_n)}:{n \in \N} = \varprojlim \Sequence{X.F}:{ F \in \script{E}}.$$
When $X$ is a locally connected compactum {with admissible vertex set $V$}, and $E(X{:}V) = \Set{e_1,e_2,\ldots}$ is any enumeration of its edges, then for $E_n= \Set{e_1,e_2,\ldots,e_n}$ we obviously have that $\Sequence{E_n}:{n \in \N}$ is cofinal in $\script{E}$. Hence, also $\varprojlim (X.E_n \colon n \in \N)$ is a compact graph-like space homeomorphic to $X_\sim$. 
\index{graph-like space|)}

\subsection{Inverse limits and sparse edge sets}


It will be important to understand how the even-cut condition changes when deleting or adding certain edge sets. For this, we shall need the following lemma, which says that the inverse limit operation commutes with deletion of edges. 

\begin{lemma}
\label{lem_evencutsplusinverselimits}
Let $X$ be a Peano continuum with admissible vertex set $V$, and $E(X{:}V) = \Set{e_1,e_2,\ldots}$ be any enumeration of its edges, {and let $E_n:= \Set{e_1,e_2,\ldots,e_n}$.} For sparse $F \subset E(X{:}V)$ write $F_n := F \cap E_n$. Then  $(X - F)_\sim  = \varprojlim \p{(X.E_n) - F_n}$. 
In particular, if $F$ is such that each $(X.E_n) - F_n$ is an even graph, then $X-F$ is a component-wise aligned compactum that has the even-cut property.
\end{lemma}

\begin{proof}
Consider a sparse edge set $F \subset E(X{:}V)$. By Lemma~\ref{lem_evencutpropofcomponents} we know that $Y = X-F$ is a component-wise aligned compactum with admissible vertex set $V$. 
Now for any $D \in \script{E}$, let us write $F_D := F \cap D$ (so $F_n = F_{E_n}$) and consider the inverse limit $\script{Y} = \varprojlim \Sequence{X.D - F_D}:{D \in \script{E}}$. Now clearly,  $\Sequence{E_n}:{n \in \N}$ is cofinal in $\script{E}$, and we have $\script{Y} = \varprojlim \p{(X.E_n) - F_n}$. 

At the same time, for any cofinal sequence $\Sequence{\script{U}_n}:{n \in \N}$ for $\Pi(V)$ we have $\script{Y} = \varprojlim \Sequence{X_{\script{U}_n} - F_{E(\script{U}_n)}}:{ n \in \N}$. However, given any clopen partition $\script{U} \in \Pi(V)$, we have $Y_\script{U} = X_\script{U} - F_{E(\script{U})}$. Therefore, we have $$ Y_\sim = \varprojlim \Sequence{Y_{\script{U}_n}}:{ n \in \N} = \varprojlim \Sequence{X_{\script{U}_n} - F_{E(\script{U}_n)}}:{ n \in \N} = \script{Y} = \varprojlim \p{(X.E_n) - F_n},$$ and the first assertion of the lemma is proven.

The second part now follows now from the previous discussion about inverse limits and the even-cut property: if $(X.E_n) - F_n$ is even for each $n \in \N$, then $\script{Y}$, and hence $ Y_\sim$, have the even-cut property, too.
\end{proof}

\subsection{Bipartite Peano partitions}

Recall {Definition~\ref{def_Bingpartition} for Peano covers and Peano partitions, and} Definition~\ref{def_intersectiongraph} for the definition of an intersection graph.

\begin{defn}[Bipartite Peano cover, zero-dimensional overlap]
A Peano cover / partition $\script{U}$ is called \emph{bipartite}\index{bipartite Peano cover|textbf}, if its intersection graph $G_\script{U}$ is bipartite. 

For a bipartite Peano cover $\script{U}$ we also write $\script{U}=\Set{K_1,K_2,\ldots, K_\ell,U_1,U_2,\ldots,U_k}$ and mean that the $K$'s form one partition class, and the $U$'s form the other partition class of the bipartite graph $G_\script{U}$. Even briefer, we say that $\p{K,U}$\index{(K,U)@$\p{K,U}$|see {bipartite Peano cover} \textbf} forms a bipartite Peano cover of some Peano continuum $X$ if $X = K \cup U$ and both $K$ and $U$ are locally connected compacta 
(note that this is indeed a bipartite cover).

Finally, a bipartite Peano cover $\p{K,U}$ is said to have \emph{zero-dimensional overlap}\index{zero-dimensional overlap (of bipartite Peano cover)|textbf} if $K \cap U$ is zero-dimensional.
\end{defn}

\begin{lemma}
\label{lem_picksmallclopenpartitionPeanoGraph}
Let $X$ be a Peano continuum with admissible vertex set $V$. Then for every $\varepsilon>0$ there is finite edge set $F \subset E(X{:}V)$ such that for each component $D$ of $X - F$ there is a component $C$ of $V$ with $D \subseteq B_\varepsilon(C)$.
\end{lemma}

{In this case we say that the components of $X- F$ are \emph{$\varepsilon$-close to the components of $V$}.}

\begin{proof}
Suppose for a contradiction the assertion is false for some $\varepsilon>0$. Enumerate $E(X{:}V) = \Set{e_1,e_2,e_3,\ldots}$ and let $F_n = \Set{e_1,\ldots,e_n}$. Then for each $n \in \N$, there is at least one \emph{bad} component $D$ of $X - F_n$ for which there is no component $C$ of $V$ with $D \subseteq B_\varepsilon(C)$. Further, every bad component of $X - F_{n+1}$ is contained in a bad component of $X - F_n$. Since $X - F_n$ has only finitely many components, Lemma~\ref{lem_removing edges}, it follows from K\"onig's Infinity Lemma \cite[Lemma~8.1.2]{Diestel} that there is a decreasing sequence $\Sequence{D_n}:{n \in \N}$ of bad components $D_n$ of $X - F_n$. 

{By \cite[Theorem 1.8]{Nadler}, the intersection $C:= \bigcap_n D_n$ is a continuum, and} since $\bigcup_n F_n = E(X{:}V)$, it follows that $C$ is a component of $V$. However, since all $D_n$ are closed in $X$ and $\bigcap_n D_n \subset B_\varepsilon(C)$, it follows from {\cite[Proposition 1.7]{Nadler}} that there is $N \in \N$ with $D_{N} \subseteq B_\varepsilon(C)$, contradicting  the fact that $D_N$ was chosen to be bad.
\end{proof}

\begin{theorem}
\label{thm_sunflowerfiniteconversionCube}
Let $X$ be a Peano continuum, and suppose that $X = K \cup U$ is such that $K = K_1 \oplus K_2 \oplus \cdots \oplus K_\ell$ consists of finitely many Peano components and the non-trivial components of $U$ form a zero-sequence of Peano continua $U_1, U_2, , \ldots$. Suppose further that every edge of $K$ intersects at most one $U_i$. Let $V$ be an admissible vertex set of $K$.

Then for every $\varepsilon >0$ there is $N \in \N$ such that $K' = K \cup \bigcup_{n > N} U_n$ admits a finite edge set $F_{K} \subset E(K{:}V)$ so that for each component $D'$ of $K'-F_K$ there is a component $C$ of $V$ with $D \subseteq B_\varepsilon(C)$.
\end{theorem}

\begin{proof}
Apply Lemma~\ref{lem_picksmallclopenpartitionPeanoGraph} {individually to the components of $K$} to find $F_K \subset E(K{:}V)$ finite such that components of $K - F_K$ are $\varepsilon / 2 $-close to {the components of} $V$. The components of $K- F_K$ are finitely many disjoint closed subsets of $X$, so some pair has minimal distance from each other. Denote that minimal distance by $\delta > 0$. Let $\eta := \min \Set{\varepsilon/2,\delta/3}$.

Now choose $N \in \N$ large enough such that $\diam{U_n} < \eta$ and $U_n \cap \p{\bigcup F_K} = \emptyset$ for all $n \geq N$. We claim that $N$ is as desired. First, note that since $X$ is connected, every $U_n$ has non-empty intersection with $K$. Therefore, it follows  that $K' = K \cup \bigcup_{n > N} U_n$ still has at most $\ell$ components, which are all Peano by Lemma~\ref{lem_addingzerosequences}. 

Moreover, any two components of $K'-F_K$ have, by choice of $\eta$ and $N$, distance at least $\delta - 2 \eta > 0$. In particular, no two components of $K-F_K$ fuse together by adding $\bigcup_{n > N} U_n$. Hence, for any component $D'$ of $K'-F_K$ there is a component $D$ of $K-F_K$ such that $D' \subset B_\eta(D)$. And by choice of $F_K$, there is a component $C$ of $V$ such that $D \subset B_{\varepsilon/2}(C)$. Thus, $D' \subset B_{\eta + \varepsilon/2} (C) \subset B_{ \varepsilon} (C)$, which completes the proof.
\end{proof}

\subsection{Modifying Peano partitions with zero-dimensional boundaries}

Consider a Peano graph $X$ for which we have a bipartite Peano partition $(K,U)$ with zero-dimensional overlap. In this subsection, we demonstrate how to modify the elements of $K$ and $U$ to obtain a new bipartite partition $(K',U')$  so as to guarantee that the subspaces $K',U'$ satisfy the even-cut condition. Moreover, we will do these changes so that $K'$ and $U'$ are arbitrarily close to the original $K$ and $U$.

\begin{theorem}
\label{thm_arrangecuts}
Let $X$ be a Peano continuum satisfying the even-cut condition that has a bipartite Peano partition $\script{U}=\p{K,U}$ with zero-dimensional overlap. Then for every $\varepsilon > 0$ there is a bipartite Peano cover $\script{U}'=\p{K',U'}$ such that
\begin{enumerate}[label=$(A\arabic*)$]
	\item\label{arrange1} $K \subseteq K'$ and $U' \subseteq U$,
	\item\label{arrange3} there is a finite edge set $F_K \subset E(K')$, so that each component of $K'-F_K$ either has diameter ${<}\varepsilon$ or is $\varepsilon$-close to a component of $\ground{K}$, and
	\item\label{arrange2} all elements of $\script{U}'$ satisfy the even-cut condition. 
\end{enumerate}
\end{theorem}


\begin{proof}
Since $K \cap U$ is compact zero-dimensional, the set $V : = \ground{X} \cup \p{K \cap U}$ is an admissible vertex set for $X$. Then every element of $\script{U}$ with the naturally induced admissible vertex set is faithfully aligned with $(X{:}V)$.
Write $K = K_1 \oplus K_2 \oplus \cdots \oplus K_\ell$ and $U = U_1 \oplus U_2 \oplus \cdots \oplus U_k$ for the Peano components of the two sides $\p{K,U}$. Since $U_i \cap K \subset U_i$ is zero-dimensional and contained in the vertex set of $U_i$ for each $i \in [k]$, by Lemma~\ref{lem_findinggraphlikeswithtargets2} there are combinatorially aligned graph-like trees $T_i \subset U_i$ with $U_i \cap K \subset V(T_i)$. Define $T = \bigcup_{i \in [k]} T_i$, a graph-like forest {(i.e.\ a graph-like compactum whose components are trees)} with $k$ components. Note that $T$ is combinatorially aligned with $(X{:}V)$ but may contain fake edges (edges contained in the ground space of $X$). However, as $T \cap K = U \cap K \subset V(T)$, no edge of $T$ intersects $K$. 

In order to arrange for \ref{arrange2}, our aim is to find a subset $F \subset E(T)$ such that by adding $F$ to $K$, denoted by $K + F := K \cup T[F]$, and removing $F^\text{real}=F \cap E^\text{real}_T$ from $U$, denoted by $U - F^\text{real}$, we obtain an edge-disjoint cover $\Set{K + F, U-F^\text{real}}$ of $X$ such that both sides satisfy the even-cut condition. In order to find this set $F$, we use logical compactness as follows. First, let $E(X) \cup E(T) = \Set{e_1,e_2,e_3,\ldots}$ be an enumeration of the countably many edges of $(X{:}V)$ together with the fake edges of $T$. Put $E_n := \Set{e_1,\ldots,e_n}$. Define $K^*=K \cup T$, which is a Peano continuum.
Now define (using the notation $Y.F:=Y.(E(Y) \cap F)$, called \emph{contracting onto $F$}, as introduced in Section~\ref{sec_recapgraphlike} above)
\[ X_n := X .E_n, \enskip K^*_n := K^*.E_n, \enskip K_n := K. E_n, \enskip U_n  := U. E_n, \; \text{and} \; S_n := T. E_n.\] 
We reiterate that not all edges of $E_n$ are edges of $X$. So $X_n - E_n$ stands for $X - \p{E_n \cap E(X)}$, $X_n = X.E_n$ stands for $X . \p{E_n \cap E(X)}$, and so $E(X_n) = E_n \cap E(X)$, and similarly in the other cases. By the results from Section~\ref{sec_recapgraphlike}, we have $X_\sim = \varprojlim X_n$, and similarly in the other cases. Note also that since $X$ is connected and satisfies the even-cut condition, it follows that every finite graph $X_n$ is Eulerian.

\begin{defn}
Let $\kappa \colon K \to K^*$, $\sigma^* \colon T \to K^*$ and $\sigma \colon T \to U$ be the (injective) inclusion maps. For every $n \in \N$, let $\pi_n$ be the (surjective) projection maps corresponding to the operation of contracting onto the edge set $E_n$, and define 
\begin{itemize}
\item $\kappa_n := \pi_n \circ \kappa \circ \pi_n^{-1} \colon K_n \to K^*_n$,
\item $\sigma^*_n := \pi_n \circ \sigma^* \circ \pi_n^{-1} \colon S_n \to K^*_n$, and
\item $\sigma_n := \pi_n \circ \sigma \circ \pi_n^{-1} \colon S_n \to U_n$.
\end{itemize}
\end{defn}
We may visualise these maps in a commuting diagram as follows: 
 \begin{center}
 \begin{tikzcd}[row sep=huge]
K \arrow[r, hook, "\kappa"] \arrow[d,twoheadrightarrow, "\pi_n"]  & K^*  \arrow[d,twoheadrightarrow, "\pi_n"] & T \arrow[l,swap, hook', "\sigma^*"] \arrow[r, hook, "\sigma"]  \arrow[d, twoheadrightarrow, "\pi_n"]   & U \arrow[d, twoheadrightarrow, "\pi_n"] \\
K_n \arrow[r,"\kappa_n"] & K^*_n & S_n \arrow[l, swap, "\sigma^*_n"] \arrow[r,"\sigma_n"] & U_n
\end{tikzcd}
\end{center}

\begin{lemma}
\label{lemma_pathsandtrails}
The following statements about the above diagram are true:
\begin{enumerate}
	\item The maps $\kappa_n$, $\sigma^*_n$ and $\sigma_n$ are well-defined (i.e.\ single valued) contraction maps, and the diagram commutes.
	\item $\kappa_n \restriction E(K_n)$, $\sigma^*_n \restriction E(S_n)$ and $\sigma_n \restriction E^\text{real}(S_n)$ act as identity, whereas $\sigma_n(E^{\text{fake}}(S_n)) \subset V(U_n),$
	\item  $\kappa(K)$ and $\sigma^*(T)$ form a decomposition of $K^*$ into connected subgraphs, and hence $\kappa_n(K_n)$ and $\sigma^*_n(S_n)$ form a decomposition of $K^*_n$ into connected subgraphs,
	\item\label{trailssss} If $P \subset T$ is a standard arc with end-vertices $a$ and $b$, then 
		\begin{itemize}
			\item $Q = \pi_n(P)$ forms a path in $S_n$ with edge set $F:= E(P) \cap E(S_n)$,
			\item $\sigma^*_n(Q)$ forms a trail\footnote{A trail is a walk without repeated edges.} in $K^*_n$ with edge set $F$ from $\pi_n(\sigma^*(a))$ to $\pi_n(\sigma^*(b))$, 
			\item $\sigma_n(Q)$ forms a trail in $U_n$ with edge set $F^{real}_n$ from $\pi_n(\sigma(a))$ to $\pi_n(\sigma(b))$.
		\end{itemize}	
\end{enumerate}
\end{lemma}

\begin{proof}
(1) and (2). To see that $\kappa_n$ is a well-defined contraction map and acts as the identity on $E(K)$, note that since $E_n \cap E(K) \subset E_n \cap E(K^*)$, it follows that every edge $e \in E(K)$, we have $\pi^{-1}_n(e) = e$, and hence $\kappa_n(e) = \pi_n \circ \kappa \circ \pi_n^{-1} (e) = e$. For a vertex $v \in V(K_n)$, note that by definition $\pi^{-1}_n(v)$ is a connected component of $K - E_n$. Hence, $\kappa(\pi^{-1}_n(v))$ is a connected subspace of $K^* - E_n$, and hence belongs to a connected component of  $K^* - E_n$. Thus,  $\pi_n(\kappa(\pi^{-1}_n(v))) = \kappa_n(v)$ is a vertex of $K^*_n$.\footnote{Note, however, that distinct vertices $v\neq v' \in V(K_n)$ may be mapped onto the same vertex in $V(K^*_n)$, as $\pi^{-1}_n(v)$ and $\pi^{-1}(v')$ are distinct components of $K - E_n$, but as subspaces might belong to the same component of $K^*-E_n$.} The proof for $\sigma^*_n$ is the same. The third case of $\sigma_n$ is almost the same, with the difference that while $\sigma_n$ is the identity on real edges of $S_n$, for every fake edge $e$ of $S_n$, we have $\sigma (\pi_n^{-1}(e)) \subset \ground{U}$, and hence belongs to a connected component of $U - E_n$, so $\sigma_n(e) = \pi_n(\sigma (\pi_n^{-1}(e))) \in V(U_n)$. 

Next, assertion (3) is clear by construction and the fact that $\kappa_n \restriction E(K_n)$, $\sigma^*_n \restriction E(S_n)$ act as identity. Finally, (4) follows from the fact that since all maps are contraction maps, trails get mapped to trails. 
\end{proof}

Let us call a subset $F_n \subset E(S_n)$ \emph{semi-good} if $U_n - \sigma_n(F_n)=U_n-F^{\text{real}}_n$ is an even subgraph of $U_n$. A semi-good set is called \emph{good}, if also $\kappa(K_n) + \sigma^*_n(F_n)  = K^*_n[E(K_n) \cup F_n]$ is an even subgraph of $K^*_n$. 

\textbf{Main claim}: \emph{For each $n \in \N$ there exists at least one good subset of $E(S_n)$.}

We will prove our main claim in two steps, first constructing a semi-good set, which we modify in a second step to a good set.

\medskip
\textbf{Step 1: There exists a semi-good subset $F'_n \subset E(S_n)$.} To see this, note that each graph $U_n$ has precisely $k$ connected components, and by the handshaking lemma, the number of odd-degree vertices of $U_n$ inside each component is even, so come in pairs. Let $\approx$ denote the corresponding equivalence relation, where each equivalence class consists of one such pair. Now for each vertex $u \in V(U_n)$, the preimage $\pi^{-1}_n(u)$ induces a clopen subset of the vertex set $V \cap U$ of $U$. If $u$ has odd degree, then necessarily $\pi^{-1}_n(u) \cap K \neq \emptyset$, as otherwise the edge-cut of $\pi^{-1}_n(u)$ induced in $U$ equals the edge-cut of $\pi^{-1}_n(u)$ induced in $X$, contradicting the even-cut property of $X$. By construction of $T$, there is a point $v_u \in \pi^{-1}_n(u) \cap K \cap V(T)$, and this point must satisfy $u = \pi_n(\sigma(v_u))$. Next, for each pair $u \approx u'$ of odd-degree vertices of $U_n$, $v_u$ and $v_{u'}$ lie in the same connected component of $T$, so there exists a unique path $P_{v_u,v_{u'}}$ in $T$ from $v_u$ to $v_{u'}$. By Lemma~\ref{lemma_pathsandtrails}(\ref{trailssss}), if we let $Q_{u,u'} = \pi_n(P_{v_u,v_{u'}})$ be the corresponding path in $S_n$, then $\sigma_n(Q_{u,u'})$ is a trail in $U_n$ from $ \sigma_n(\pi_n(v_u))=\pi_n(\sigma(v_u))=u$ to $\sigma_n(\pi_n(v_{u'})) = \pi_n(\sigma(v_{u'}))=u'$, where the respectively first equalities hold since the above diagram commutes, and the respective second equalities hold by choice of $v_u$ and $v_{u'}$. In particular, all vertices, apart from the end-vertices have even degree in that trail. Define $F'_n := \sum_{u \approx u'} E(Q_{u,u'})$. Then $\sigma_n(F'_n) = \sum_{u \approx u'} \sigma_n(Q_{u,u'})$ is the mod-2 sum over these trails, and so it is precisely the odd degree vertices of $U_n$ that have odd parity in $U_n[\sigma_n(F'_n)]$. Thus, $U_n - \sigma_n(F'_n)$ is an even graph, and so $F'_n$ is semi-good. 

\medskip
\textbf{Step 2: There exists a good subset $F_n \subset E(S_n)$.} First, fix a semi-good subset $F'_n \subset E(S_n)$, let ${F'_n}^\complement =  E(S_n) \setminus F'_n$ and define $K'_n = K^*_n - \sigma^*_n({F'_n}^\complement)$ and $U'_n = U_n - \sigma_n(F'_n)$. As before, for each vertex $k \in V(K^*_n)=V(K'_n)$, the set $\pi^{-1}_n(k)$ is a connected component of $K^* - E_n$, and hence a subcontinuum of $X-E_n$. Similarly, for each vertex $u \in V(U_n)=V(U'_n)$, the set $\pi^{-1}_n(u)$ is a connected component of $U - E_n$, and hence also a subcontinuum of $X-E_n$. Hence, for $\script{U} =\set{\pi^{-1}_n(v)}:{v \in V(K^*_n) \sqcup V(U_n)}$ we may consider the intersection graph $G=G_\script{U}$ of $\script{U}$ in $X-E_n$. For ease of notation, relabel 
$$V(G) =V(K^*_n) \sqcup V(U_n) \; \text{ and } \; E(G) = \set{vw}:{ \pi^{-1}_n(v) \cap \pi^{-1}_n(w) \neq \emptyset}.$$
Observe that $G$ is a bipartite graph with vertex bipartition $V(G) = V(K^*_n) \sqcup V(U_n)$, as whenever $k \neq k'$ are distinct vertices in $K^*_n$, then $\pi^{-1}_n(k)$ and $\pi^{-1}_n(k')$ are distinct components of $K^* - E_n$, and hence do not intersect, and similarly for $u \neq u' \in V(U_n)$.

\begin{mysubclaim}
\label{subclaim1}
Whenever $ku \in E(G)$, then $\pi_n^{-1}(k) \cap \pi_n^{-1}(u) \cap V(T) \neq \emptyset$.
\end{mysubclaim}

\begin{proof}[Proof of Subclaim~\ref{subclaim1}]\renewcommand{\qedsymbol}{$\Diamond$}
 Since $K^* \cap U = (K \cap U) \cup T$, the fact that $ku \in E(G)$ implies $\pi_n^{-1}(k) \cap \pi_n^{-1}(u)  \subset (K \cap U) \cup T$. Since $K \cap U \subset V(T)$, we only have to consider the case where $\pi_n^{-1}(k) \cap \pi_n^{-1}(u)$ intersect in an edge $e$ of $E(T)$, in which case $e(0),e(1) \in \pi_n^{-1}(k) \cap \pi_n^{-1}(u) \cap V(T)$, as $\pi_n^{-1}(k)$ and $\pi_n^{-1}(u)$ are standard subcontinua, and if $e$ is a fake edge, then $\closure{e} \subset \ground{U}$, so contained in a single component of $U- E_n$.
 \end{proof}

Next, for every connected component $C$ of the graph  $G$, the set $\bigcup \pi^{-1}_n(C)$ is a subspace of $X-E_n$. Write $ \script{C}(G) : = \set{\bigcup \pi^{-1}_n(C)}:{C \text{ a connected component of } G}$.  

\begin{mysubclaim}
\label{subclaim2}
We have $\set{\pi^{-1}_n(x)}:{x \in V(X_n)} = \script{C}(G)$.
\end{mysubclaim}
\begin{proof}[Proof of Subclaim~\ref{subclaim2}]\renewcommand{\qedsymbol}{$\Diamond$}
 This will follow once we show that $\script{C}(G)$ forms a partition of $X- E_n$ into subcontinua. First, each $\pi^{-1}_n(C)$ is a subcontinuum of $X-E_n$. This follows easily by induction on $|C|$, since for every edge $ku \in E(G)$, the two subcontinua $\pi^{-1}_n(k)$ and $\pi^{-1}_n(u)$ intersect by definition, so $\pi^{-1}_n(k) \cup \pi^{-1}_n(u)$ is again a subcontinuum. Next, for components $C \neq C'$ of $A$, if $\bigcup \pi^{-1}_n(C) \cap \bigcup \pi^{-1}_n(C') \neq \emptyset$, there would be $v \in C$ and $w \in C'$ such that $\pi^{-1}_n(v) \cap \pi^{-1}_n(w) \neq \emptyset$, and so $vw \in E(G)$, contradicting that $v$ and $w$ belong to distinct components of $G$. Finally, $X-E_n \subset (K^* - E_n) \cup (H-E_n)$ yields that $\bigcup \pi^{-1}_n(V(G)) = X- E_n$.
\end{proof}

Now a component $C$ of $G$ can be viewed as a single vertex of $X_n$, and hence induces an edge cut in $X_n$. Similarly, by the nature of $G$, a component $C$ also induces edge cuts in $K'_n$ and in $U'_n$: write $E_{K'_n}(C, C^\complement)$ as shorthand for the edge cut of $K'_n$ with sides $V(K'_n) \cap C$ versus $V(K'_n) \setminus C$.

\begin{mysubclaim}
\label{subclaim3}
We have $E_{X_n}(C, C^\complement) = E_{K'_n}(C, C^\complement) \sqcup E_{U'_n}(C, C^\complement)$ for any component $C$ of $G$, and hence $E_{K'_n}(C, C^\complement)$ is always even.
\end{mysubclaim}
\begin{proof}[Proof of Subclaim~\ref{subclaim3}]\renewcommand{\qedsymbol}{$\Diamond$}
 To see this claim, note that $E_{K'_n}(C, C^\complement)$ cannot contain fake edges of $T$, as any such edge lies in $\ground{U}$, contradicting that $C$ is a component of $A$. Hence, all edge cuts are subsets of $E(X_n)$. The equality of sets now follows from that fact that $K'_n$ and $U'_n$ are $E(X_n)$-edge-disjoint, and together cover all edges of $X_n$. Now since $X_n$ and $U'_n$ were even graphs by assumption, and so have the even-cut property, it follows that $E_{K'_n}(C, C^\complement)$ is even for every component $C$ of $A$. 
\end{proof}

To complete the proof of the second step, and hence of our main claim, note that by Subclaim~\ref{subclaim3} and the handshaking lemma, for any connected component $C$ of $G$, the number of vertices of $K^*_n$ which have odd-degree in $K'_n$ in $C$ is always even. 
Hence, we can pair up odd degree vertices of $K'_n$ such that for every pair $k \approx k'$ there is a path $A_{k,k'}$ in $G$ say with vertices $k_0 u_1 k_1u_1 \ldots u_{j-1} k_j$ where $k=k_0$, $k'=k_j$, $k_i \in V(K^*_n)$, $u_i \in V(U_n)$ and edges $\Set{k_0u_1, u_1k_1, k_1u_2, \ldots, u_{j-1} k_j } \subset E(G)$, using that $G$ is bipartite. By Subclaim~\ref{subclaim1}, for every $i \in [j]$ we may pick a point $a_i \in \pi^{-1}_n(k_{i-1}) \cap \pi^{-1}_n(u_i) \cap V(T)$ and a point $b_i \in \pi^{-1}_n(u_i) \cap \pi^{-1}_n(k_i) \cap V(T)$, and let $P_i$ be the unique path from  $a_i$ to $b_i$ in the forest $T$, which exists as $\pi^{-1}_n(u_i)$ is contained in a unique component of $U$. 

Now arguing as in Step 1, if we let $Q_{i} = \pi_n(P_i)$ be the corresponding path in $S_n$, then $\sigma_n(Q_i)$ is a trail in $U_n$ from $\pi_n(\sigma(a_i))=u_i$ to $ \pi_n(\sigma(b_i))=u_i$, i.e.\ $\sigma_n(Q_i)$ is a closed trail, so all vertices of $U_n$ in $\sigma_n(Q_i)$ have even degree. Hence, $\sum_{i \in [j]} \sigma_n(Q_i)$ is an even subgraph of $U_n$. At the same time, however, every $\sigma^*_n(Q_i)$ is a trail in $K^*_n$ from $\pi_n(\sigma^*(a_i))=k_{i-1}$ to $\pi_n(\sigma^*(b_i))=k_i$, and so $\sum_{i \in [j]} \sigma^*_n(Q_i)$ induces a subgraph in $K^*_n$ in which all vertices, apart from $k=k_0$ to $k'=k_n$ have even degree. Thus, if we let $F_{k,k'} =  \sum_{i \in [j]} E(Q_i)$, then $\sigma_n(F_{k,k'})$ is an even subgraph of $U_n$, and in the subgraph induced by $\sigma^*_n(F_{k,k'})$ in $K^*_n$, all vertices have even parity apart from precisely $k$ and $k'$. Hence, $F_n := F'_n + \sum_{k \approx k'} F_{k,k'}$ is a good subset $F_n \subset E(S_n)$. This completes the proof of Step 2.

\smallskip

Recall that we set out to show the existence of a set $F \subset E(T)$ such that by adding $F$ to $K$ and removing $F^{\text{real}}=F \cap E^{\text{real}}_T$ from $U$, we obtain an edge-disjoint cover $\Set{K + F, U-F^{\text{real}}}$ of $X$ such that both sides satisfy the even-cut condition. We will now obtain such a set $F$ from the good edge sets of $E(S_n)$ as follows. Since $E(S_n)$ is finite, each $E(S_n)$ has only finitely many good subsets. Moreover, since $U_{n} = U_{n+1} / e_{n+1}$ and $K^*_n = K^*_{n+1} / e_{n+1}$ are obtained by edge-contraction, even subgraphs of $H_{n+1}$ and $K^*_{n+1}$ restrict to even subgraphs of $U_n$ and $K^*_n$. Thus, every good choice $F_{n+1} \subset E(S_{n+1})$ at step $n+1$ induces a good choice $F_n = F_{n+1} \cap E(S_n)$ at step $n$. So by K\"onig's Infinity Lemma \cite[Lemma~8.1.2]{Diestel}, there is a sequence of good sets $\Sequence{F_n}:{n \in \N}$ with $F_n \subset E(S_n)$ such that $F_{n+1} \cap E(S_n) = F_n$ for all $n \in \N$. Now given such a sequence $\Sequence{F_n}:{n \in \N}$, define $F = \bigcup_{n \in \N} F_n \subset E(T)$ and claim that $F$ is as desired, i.e.\ that $K +  T[F]$ and $U - F^\text{real}$ have the even-cut property. Indeed, since $F^{\text{real}} \cap E(U_n) = F^\text{real}_n$ it follows from Lemma~\ref{lem_evencutsplusinverselimits} that $(U - F^\text{real})_\sim = \varprojlim \p{U_n - F^\text{real}_n}$ has the even-cut property. Hence, $U - F^{\text{real}}$ has the even-cut property. Similarly, also $K \cup T[F]$ has the even-cut property, as $K^*_\sim[E(K) \cup F] =  \varprojlim \p{K^*_n[E(K_n) \cup F_n]}$ is the inverse limit of even graphs. 

Moreover, since $K''=K \cup T[F]$ satisfies the even-cut condition, every leaf of $T[F]$ must intersect $K$ (as otherwise, there would be a vertex in  $\p{K \cup T[F]}_\sim$ of degree $1$, contradicting the even-cut property), and hence $K \cup T[F]$ continues to have at most $\ell$ connected components. Moreover, since the non-trivial components of $T[F]$ form a zero-sequence of graph-like continua, Lemma~\ref{lem_componentwisealignedgraph-like}, each of the $\ell$ components of $K \cup T[F]$ remains a Peano continuum, Lemma~\ref{lem_addingzerosequences}. Since $F$ is sparse, $U''=U - F^\text{real}$ is a component-wise aligned compactum such that every component of $U''$ is faithfully aligned in $(X{:}V)$,  Lemma~\ref{lem_evencutpropofcomponents}. By Lemma~\ref{lem_compnentwiseevencut}, each component of $U''$ satisfies the even-cut condition. To complete the proof of the theorem, we would like $U''$ to have only finitely many components. We rectify this problem by reassigning all but finitely many of these components of $U''$ back to $K''$, without violating property~\ref{arrange3}. Indeed, we may construct $K'$ and $U'$ as desired by applying Theorem~\ref{thm_sunflowerfiniteconversionCube} providing a finite edge set $F_K$ so as to satisfy \ref{arrange3}. Moreover,  by Lemma~\ref{lem_evencutUnions}, this reassignment preserves the even-cut condition of $K''$, and so $K'$ and $U'$ satisfy \ref{arrange2}. That it satisfies \ref{arrange1} is clear from construction, since we only ever added edge sets to $K$. 
\end{proof}


\section{Eulerian Decompositions of One-Dimensional Peano Continua}
\label{sec_54}

\subsection{The decomposition theorem}
\label{sec_proofDecompTheorem}

\begin{theorem}[$2^{\text{nd}}$ decomposition theorem]
\label{thm_1DimDecomposition}
Every one-dimensional Peano continuum $X \subset [0,1]^3$ with admissible vertex set $V$ satisfying the even-cut condition admits a Peano cover $\Set{X_1,\ldots,X_s}$ into edge-disjoint standard combinatorially aligned Peano subcontinua $X_i$ with {admissible vertex} sets $V_i$ each satisfying the even-cut condition, such that for each $i \in [s]$ there is a {finite} edge  set $F_i \subset E(X_i{:}V_i)$ 
such that every component $C$ of $X_i - F_i$ either satisfies
	$C \subset [0,\tfrac23] \times [0,1] \times [0,1] \subset [0,1]^3$ or $C \subset [\tfrac13,1] \times [0,1] \times [0,1] \subset [0,1]^3.$
\end{theorem}

Our proof relies critically on the fact that  one-dimensional Peano continua have exceptionally nice Peano partitions (Def.~\ref{def_Bingpartition}) that reflect properties of dimension, announced by Bing in \cite[Theorem~11]{partitioning} and used crucially by Andersen as a step towards the topological characterisation of the Menger universal curve in \cite{anderson, anderson2}. See also \cite{Mengercurve} for a detailed account, including a published proof in the one-dimensional case.

\begin{theorem}[{\cite[Theorem 2.9]{Mengercurve}}]
\label{thm_nicebrickpart}
Every one-dimensional Peano continuum admits a decreasing sequence of $1/n$-Peano partitions $\set{\script{U}_n}:{n \in \N}$ with zero-dimensional boundaries.\footnote{The theorem proved in \cite[Thm.\ 2.9]{Mengercurve} is stronger, but we shall not need these additional properties.}
\end{theorem}

\begin{proof}[Proof of Theorem~\ref{thm_1DimDecomposition}]
For $i \in [3]$ let $\pi_i \colon [0,1]^3 \to [0,1]$ denote the projection map from the cube onto the $i$th coordinate. Let $\varepsilon = 1 / 6 $. Pick an $\varepsilon$-{Peano} partition $\script{U}$ of $X$ with zero-dimensional boundaries as in Theorem~\ref{thm_nicebrickpart}, and let $\script{U}_u \subset \script{U}$ be the sub-collection $\script{U}_u = \set{U \in \script{U}}:{ U \cap \pi_1^{-1}[2/3,1]\neq \emptyset}$ and let $\script{U}_\ell := \script{U} \setminus \script{U}_u$. Next, let $K= \bigcup \script{U}_u$, and similarly let $U = \bigcup \script{U}_\ell$, giving rise to a bipartite Peano partition $\script{U}=\p{K,U}$ of $X$ with zero-dimensional overlap by the sum theorem of dimension, \cite[Thm.\ 1.5.2]{engelkingdimension}. Apply Theorem~\ref{thm_arrangecuts} to $\script{U}$ with { the same $\varepsilon=1/6$} to obtain a bipartite Peano cover $\script{U}'=\p{K',U'}$ of $X$ with properties \ref{arrange1}, \ref{arrange3} and \ref{arrange2} of Theorem~\ref{thm_arrangecuts}. For later use, let $F_K$ denote the finite edge set of $K'$ witnessing \ref{arrange3}. We claim that $\script{U}'$ is as desired.

Clearly, by construction and property \ref{arrange2}, $\script{U}'=\Set{X_1,\ldots,X_s}$ is a finite decomposition of $X$ into edge-disjoint Peano subcontinua $X_i$ each satisfying the even-cut condition; {in particular,  by Lemma~\ref{lem_EvenCutImpliesStandard} and Corollary~\ref{cor_EvencutImpliesCombAligned} we may assume that they are standard and combinatorially aligned in $(X{:}V)$}. 

Next, note that by \ref{arrange1} we have $U' \subset U$, and by construction {of $\script{U}_\ell$} we have $U \subset [0,\frac23] \times [0,1] \times [0,1]$. {In particular, this establishes the last assertion of the theorem for all $X_i \subset U'$. For all remaining $X_i \subset K'$, we claim that $F_i :=F_K \cap E(X_i)$ is as desired}. Indeed, by \ref{arrange3} any component $C$ of $K' - F_K$ either has diameter $\diam{C} < \varepsilon$, in which case we  trivially have 
$$C \subset [0,\tfrac23] \times [0,1] \times [0,1] \subset [0,1]^3 \; \text{ or } \;  C \subset [\tfrac13,1] \times [0,1] \times [0,1] \subset [0,1]^3,$$
or $C$ is contained in $B_\varepsilon(D)$ for some component $D$ of $\ground{K}$. In this case, since by construction we have
$D \subset K \subset [\tfrac23-\varepsilon,1] \times [0,1] \times [0,1]$, the fact that $C \subset B_\varepsilon(D)$ implies that
$$C \subset [\tfrac23-2\varepsilon,1] \times [0,1] \times [0,1] = [\tfrac13,1] \times [0,1] \times [0,1],$$
completing the proof.
%
\end{proof}


\subsection{Eulerian decompositions of one-dimensional Peano continua}

In this section we finally prove Theorem~\ref{thm_1DimDecomposition}. Let us fix a one-dimensional Peano continuum $X$ which satisfies the even-cut condition. 
By N\"obling's embedding theorem \cite[1.11.4]{engelkingdimension}, every one-dimensional continuum embeds into the unit cube $[0,1]^3$, and so for our purposes we may assume that $X$ is given as a subspace $X \subseteq [0,1]^3$. The goal is to show how the decomposition theorem may be used to construct an approximating sequence of Eulerian decompositions \index{Eulerian decomposition} for $X$, thereby implying the Eulerianity conjecture for all one-dimensional Peano continua.

First, recall that by \cite[Thm.\ 1.8.13]{engelkingdimension}, since $X$ is one-dimensional, the complement of $X$ in $[0,1]^3$ is connected, and since it is open, it must then be path-connected. Therefore, given $X \subset [0,1]^3$, we may add any finite set of edges between specified points of $X$ in 3-space to obtain a Peano continuum $X'$ such such that $X \subset X' \subset [0,1]^3$.


\begin{defn}[Truncation]
Let $\script{D}=(G,\eta)$ be a decomposition of a Peano continuum $X$, and let $v\in V(G)$. The \emph{truncation of $\script{D}$ to $v$}\index{truncation (of a decomposition)|textbf}, denoted by $\tau(v)$\index{tau(v)@$\tau(v)$|see {truncation (of a decomposition)} \textbf}, is a Peano continuum with $\tau(v) \supseteq \eta(v)$ with additional edges $E(\tau(v)) \setminus E(\eta(v)) = \set{e \in E(G)}:{e \sim v}$ and ground set 
$$\ground{\tau(v)} = \begin{cases} \ground{\eta(v)} & \text {if } E_G(v,G-v) = \emptyset, \\  \ground{\eta(v)} \oplus \Set{\star} & \text{otherwise,} \end{cases} $$
where vertex-edge incidences for the new edges are given by
$$e_{\tau}(i) = \begin{cases} 	(\eta(e))(i) & \text{ if } e(i) = v  \\ 
	\star & \text{ otherwise.}\end{cases} 	$$
for $e \sim v$ in $G$ and $i \in \Set{0,1}$. 
\end{defn}

Truncating means first contracting the subgraph $G[V(G) - v]$ to a single vertex $\star$, and then blowing up the `vertex' $v$ to its associated tile $\eta(v)$, connecting all edges previously incident with $v$ in $G$ to their correct endpoints in $\eta(v)$. The case distinction ensures that if $\star$ was isolated, it is to be disregarded (there might still be loops attached to $v$ in $G$).

From the above discussion we deduce the next lemma.
\begin{lemma}
\label{lem_truncationIn3Space}
Let $\script{D}=(G,\eta)$ be a decomposition of a Peano continuum $X$. A truncation $\tau(v)$ is always a connected Peano graph, and if $\eta(v) \subset [0,1]^3$, then we may always assume that $\eta(v) \subset \tau(v) \subset [0,1]^3$ for all $v \in V(G)$.
\end{lemma}

As announced, let us see how the Decomposition Theorem~\ref{thm_1DimDecomposition} can be used to construct an approximating sequence of Eulerian decompositions. For an example of an approximating sequence of Eulerian decompositions that satisfies property~\ref{item_newEuleriantiles} in the next proof, consider once more the hyperbolic $4$-regular tree from Figure~\ref{figure_extendingEulerdecomp} in Chapter~\ref{chap_Eulerdecomp}.

\begin{proof}[Proof of Theorem~\ref{thm_blablabla}]
We construct a sequence $\Sequence{\p{G_n,\eta_n}}:{n \in \N}$ of Eulerian decompositions for $X$ with $\p{G_0,\eta_0} \preccurlyeq \p{G_1,\eta_1} \preccurlyeq \cdots$ by recursion on $n$, such that each Eulerian decomposition $\p{G_n,\eta_n}$ satisfies, besides its usual properties \ref{Eta1}--\ref{E3} from Definition~\ref{def:Eulerdecomp}, the following extra two requirements:


\begin{enumerate}[label=(E\arabic*)]\setcounter{enumi}{7}
    \item \label{EX} each tile $\eta_n(v)$ is combinatorially aligned with $X$, 
	\item \label{item_newEuleriantiles}\label{E8} each truncation $\tau_n(v)$ satisfies the even-cut condition for all vertices $v$ of $\p{G_n,\eta_n}$,
	\item \label{item_newmesh}\label{E9} for every vertex $v$ of $\p{G_n,\eta_n}$, the tile $\eta_n(v)$ is contained in a cube $I_v$ with
	$$\eta_n(v) \subseteq I_v = I_v^1 \times I_v^2 \times I_v^3 \subset [0,1]^3$$ 
	such that for $r = n \pmod 3$ we have
	$$\diam{I_v^k} = \begin{cases} \p{\frac23}^{\lfloor n/3 \rfloor + 1} & \text{ if } k \leq r \\ 
							\p{\frac23}^{\lfloor n/3 \rfloor} & \text{ otherwise. }
	\end{cases}
	$$
\end{enumerate}
For the base case, we can choose the trivial decomposition. So suppose for some $n \in \N$ we have an Eulerian decomposition $(G_n,\eta_n)$ with properties \ref{EX},\ref{E8} and \ref{E9}, and write $E(G_n) = F_n \sqcup D_n$ for the implicit partition into displayed and dummy edges. Our task is to construct an Eulerian decomposition $(G_{n+1},\eta_{n+1})$ with properties \ref{EX},\ref{E8} and \ref{E9}, so that  $(G_{n+1},\eta_{n+1})$ extends $(G_n,\eta_n)$. In order to satisfy \ref{item_newmesh} at step $n+1$, it is clear that we have to cut our tiles apart along the unique coordinate $i \in \Set{1,2,3}$ where $n+1 = 3m + i$ for some $m \in \N$; without loss of generality, we may assume in the following that $i=1$.

Consider $v \in V(G_n)$. For ease of notation, we rescale affinely in all coordinates so that $I_v = [0,1]^3$. By Lemma~\ref{lem_truncationIn3Space}, we may assume that $\eta_n(v) \subset \tau_n(v) \subset [0,1]^3$. Then in combination with property \ref{EX} and \ref{E8}, we are allowed to apply Theorem~\ref{thm_1DimDecomposition} to the truncation $\tau_n(v)$ and obtain a finite Peano cover
$$\script{S}_v=\Set{X_1,X_2,\ldots,X_{s(v)}}$$
of $\tau_n(v)$ such that
\begin{enumerate}[label=(\roman*)]
\item the elements are pairwise edge-disjoint,
\item\label{romanI} each element satisfies the even-cut condition, 
\item each element is combinatorially aligned with $\tau_n(v)$,
\item\label{romanII} for each $i \in [s(v)]$ there is a finite edge set $F_i \subset E(X_i)$ 
such that every component $C$ of $X_i - F_i$ either satisfies
	$C \subset [0,\tfrac23] \times [0,1] \times [0,1] \subset [0,1]^3 \; \text{ or } \;  C \subset [\tfrac13,1] \times [0,1] \times [0,1] \subset [0,1]^3.$\end{enumerate}
Write $E_v = E(\tau_n(v)) \setminus E(\eta_n(v))$ for the `artificial' edges of $\tau_n(v)$. Write $F'_i = F_i \setminus E_v$, $F_v:= \bigcup_{i \in [s(v)]} F'_i$, and let us write $X_{i1},\ldots, X_{i\ell_i}$ for the finitely many components of $X_i - (E_v \cup F'_i)$ other than $\star$ (Lemma~\ref{lem_removing edges}). Let us write $\script{V}_v$ for the collection of all these $X_{ik}$. We have obtained a decomposition $\script{P}_v = \script{V}_v \cup F_v$ of $\eta_n(v)$ into edge disjoint standard subspaces $\script{V}_v$ and newly displayed edges $F_v$.\footnote{Note that some $X_{ik}$ is allowed to consist of a single edge, which does not count as being displayed.} Repeat this procedure for each $v \in V(G_n)$.

Our next task is to turn these partitions into an Eulerian decomposition $(G_{n+1},\eta_{n+1})$ of $X$. For this, we first define an auxiliary decomposition $(G'_{n+1},\eta'_{n+1})$, where the underlying graph $G'_{n+1}$ has vertex and edge set $E(G'_{n+1}):=F_{n+1} \sqcup D_n$ as follows:
\begin{itemize}
\item $V(G_{n+1}) := \bigsqcup_{v \in V(G_n)} \script{V}_v$ and
\item $F_{n+1} := F_n \sqcup \bigsqcup_{v \in V(G_n)} F_v$.
\end{itemize}
For the map $\eta'_{n+1}$ we take the natural candidate: for $e \in  F_n \cup D_n$, define $\eta'_{n+1}(e) := \eta_{n}(e)$. And for $x \in \script{P}_v$ (vertices and newly displayed edges alike) define $\eta'_{n+1}(x) = x$. Next, note that the map $\varrho'_n \colon G'_{n+1} \to G_n$ defined by $\varrho'_n \restriction \p{F_n \cup D_n} := \operatorname{id}$ and $\varrho'^{-1}_n(v) := \script{P}_v$ is a surjective map satisfying \ref{Q1} and \ref{Q2} of a contraction map, cf.\ Definition~\ref{defn_edgecontraction}. As our next step, we need to define vertex-edge-incidences for $G'_{n+1}$ so that
\begin{enumerate}[label=(\alph*)]
\item \ref{E2} and \ref{E3} are satisfied, i.e.\  $(G'_{n+1},\eta'_{n+1})$ is indeed a decomposition of $X$ according to Definition~\ref{def:Eulerdecomp},
\item \ref{Q3} and \ref{Q4} are satisfied for $\varrho'_n$, i.e.\ $\varrho'_n$ is a contraction map from $G'_{n+1}$ to $G_n$ according to Definition~\ref{defn_edgecontraction}, and so that 
\item $\varrho'_n$ is $\eta$-compatible according to Definition~\ref{defn_etacompatible}.
\end{enumerate}
So let us consider an arbitrary edge $f \in E(G'_{n+1})$. Suppose first that $f \in F_n \cup D_n$. Then $f \in E(G_n)$ where it is incident to $f_{G_n}(0)=v$ and $f_{G_n}(1)=w$ say (not necessarily distinct). In order to define $f_{G_{n+1}}(0)$, note that $f \in \tau_n(v)$, and hence there is a unique $X_i \in \script{S}_v$ with $f \in E(X_i)$. Since $f \in E_v$, there is a unique component $X_{ik}$ of $X_i - \p{E_v \cup F'_i}$ such that $f(0) \in X_{ik}$, and so we may define $f_{G'_{n+1}} (0) := X_{ik}$. This assignment satisfies \ref{E2} or \ref{E3} respectively by construction, as well as \ref{Q3}. Suppose next that $f \in F_{n+1} \setminus F_n$. By definition of $F_{n+1}$, there is a unique $v \in V(G_n)$ such that $f \in F_v$. This means in turn, that $f \in E(X_i)$ for some $X_i \in \script{S}_v$, and so there are unique components $X_{ik},{X}_{ij}$ of $X_i - (E_v \cup F'_i)$ such that $f(0) \in X_{ik}$ and $f(1) \in X_{ij}$. Hence, by defining $f_{G'_{n+1}}(0)=X_{ik}$ and $f_{G'_{n+1}}(1)=X_{ij}$, we see that this assignment satisfies \ref{E2} as well as \ref{Q4}. Hence, we have verified (a) and (b), and now that $\varrho'_n$ is indeed a contraction map, if is clear that it also is $\eta$-compatible, for we have
$$\eta_n(x) = \bigcup \set{\eta'_{n+1}(y)}:{y \in \varrho'^{-1}_n(v)}
$$
for all $x \in V(G_n) \cup E(G_n)$ by construction.


This completes the construction of $G'_{n+1}$ and $\varrho'_n \colon G'_{n+1} \to G_n$. Next, we claim that every vertex in $G'_{n+1}$ has even degree: indeed, for every vertex $v$ of $G'_{n+1}$ with corresponding tile $\eta'_{n+1}(v) = X_{ik}$ with $X_{ik} \subset X_i \in \script{S}_{\varrho'_n(v)}$, we have that the edges $E_{G'_{n+1}}(v)$ incident with $v$ in $G'_{n+1}$ correspond precisely to the edges in $(E_v \cup F_v) \cap E(X_i)$ incident with the component $X_{ik}$. However, since $X_i$ satisfies the even-cut condition by \ref{romanI}, it follows that this is an even number of edges, and hence that $v$ has even degree in $G'_{n+1}$. 

For later use, note that it follows from \ref{romanII} that $(G'_{n+1},\eta'_{n+1})$ satisfies \ref{E9}. Moreover, $(G'_{n+1},\eta'_{n+1})$ also satisfies \ref{E8}: indeed, for every $w \in V_{n+1}$ with $\eta'_{n+1}(w) \subset X_i \in \script{S}_v$ it is easy to verify that $\tau'_{n+1}(w)$ is a contraction of $X_i$; since $X_i$ satisfied the even-cut condition by \ref{romanI}, so does $\tau'_{n+1}(w)$. 

To turn $G'_{n+1}$ into the final Eulerian multi-graph $G_{n+1}$, we now generously add parallel dummy edges in $D_{n+1} \setminus D_n$ in order to make the graph connected,\footnote{While dummy edges are introduced in parallel pairs when they emerge for the first time in $G_{n+1}$, we do not (and cannot) require them to remain parallel in $G_{n+2}$.} making sure that \ref{E3}, \ref{Q4} and \ref{Q5} hold for these new dummy edges. Indeed, to achieve connectedness of $G_{n+1}$ is it sufficient, since $G_n$ was connected, to arrange for \ref{Q5}, i.e.\ to show that $\varrho_n$ has connected fibres. Towards this, recall that every $\eta_n(v)$ for $v \in V(G_n)$ was connected by definition. Let $\script{U}_v$ be the family of components of $\set{Y - E_v}:{Y \in \script{S}_v}$. Then $\script{U}_v$ is a finite family of continua covering $\eta_n(v)$, and hence its intersection graph $G_{\script{U}_v}$ on $\eta_n(v)$ is connected. Pick a spanning tree $T_v$ for $G_{\script{U}_v}$. For every edge $g=ab \in E(T_v)$ pick a point $x_g \in a \cap b \neq \emptyset$ in the overlap of the corresponding sets and then add two parallel dummy edges $d^1,d^2$ to $G_{n+1}$ with associated point $\eta_{n+1}(d^1) = x_g = \eta_{n+1}(d^2)$ and incidences so that $d^1(0) = d^2(0) \subset a$ and $d^1(1) = d^2(1) \subset b$. 

Then it is clear that $G_{n+1}$ is connected, and since we added new dummy edges in pairs, $G_{n+1}$ is still even. Thus, we have verified that $G_{n+1}$ is Eulerian, and so $(G_{n+1},\eta_{n+1})$ is an Eulerian decomposition of $X$ extending $(G_{n},\eta_{n})$ and satisfying \ref{E9}. Finally, it remains to check that also \ref{E8} holds true for $(G_{n+1}, \eta_{n+1})$. But this now follows easily from the fact that $(G'_{n+1}, \eta'_{n+1})$ satisfied \ref{E8}: indeed, since new dummy edges only occur in pairs, it follows that for every $w \in V(G_{n+1}) = V(G'_{n+1})$, the truncations $\tau_{n+1}{(w)}$ and $\tau'_{n+1}(w)$ differ only by a finite family of edges, which come in parallel pairs between $\star$ and (pairwise) the same point on the ground set on $\eta_{n+1}(w)$. It is clear that the even-cut condition is unaffected by these changes.

But now, since \ref{E9} implies that that $w\p{G_n,\eta_n} \leq \p{\frac23}^{\lfloor n/3 \rfloor} \to 0$, it follows that \ref{A1} and \ref{A2} of Definition~\ref{def:approximating} are satisfied, i.e.\ $\Sequence{\p{G_n,\eta_n}}:{n \in \N}$ is an approximating sequence of Eulerian decompositions for $X$. 
This completes the proof.
\end{proof}

\section{Outlook}
The techniques introduced in this chapter for one-dimensional continua lead to an abstract framework and to a technical conjecture, the truth of which implies the truth of the Eulerianity conjecture.

\begin{defn}
The \emph{core-size} of a Peano continuum $X$ is the real number $\core{X} = \sup \set{\diam{C}}:{C \text{ a connected component of } \ground{X}}$. 
For a collection of Peano continua $\script{U}$, we write $\Gmesh{\script{U}} = \sup \set{\core{X}}:{X \in \script{U}}$.
\end{defn}

\begin{defn}
An \emph{even-cut decomposition} of a Peano continuum $X$ is a finite cover $\script{U}$ of $X$ consisting of edge-disjoint standard subcontinua each of which has the even-cut property. A class $\mathscr{C}$ of Peano continua is \emph{closed under even-cut decompositions} if every $X \in \script{A}$ satisfies the even-cut property and admits even-cut decompositions $\script{U}$ of arbitrarily small $\Gmesh{\script{U}}$ such that $U \in \mathscr{C}$ for all $U \in \script{U}$.
\end{defn}

The results of this Chapter~\ref{chapter_1DimRemainders} can then summarised as follows:

\begin{theorem}
\label{thm_abstractthm1}
The class of all one-dimensional Peano continua with the even-cut property is closed under even-cut decompositions. \qed
\end{theorem}

\begin{theorem}
\label{thm_abstractthm2}
If $\mathscr{C}$ is a class of Peano continua closed under even-cut decompositions, then the Eulerianity conjecture holds for every $X \in \mathscr{C}$. \qed
\end{theorem}

Indeed, Theorem~\ref{thm_abstractthm1} follows by iterative applications of Theorem~\ref{thm_1DimDecomposition}, and Theorem~\ref{thm_abstractthm2} follows as in the proof of  Theorem~\ref{thm_blablabla} above, noting that by Lemma~\ref{lem_picksmallclopenpartitionPeanoGraph}, for every Peano continuum $X$ and every $\varepsilon>0$ there is a finite edge set $F \subset E(X)$ such that $\diam{C} < \core{X} + \varepsilon$ for every component $C$ of $X-F$.

\begin{conj}
\label{conj_abstract}
The class $\mathscr{C}$ of all Peano continua with the even-cut property is closed under even-cut decompositions.
\end{conj}

In other words, we conjecture that every Peano continuum $X$ satisfying the even-cut condition admits, for every $\varepsilon>0$, a finite cover $\script{U}$ of edge-disjoint standard subcontinua of $X$ all satisfying the even-cut condition with $\Gmesh{\script{U}} < \varepsilon$.

By Theorem~\ref{thm_abstractthm2}, the truth of Conjecture~\ref{conj_abstract} implies the truth of Conjecture~\ref{conj_eulerian}.

\newpage

\printindex


\newpage




















\begin{thebibliography}{99}

\bibitem{Charat} H.\ Abobaker and W.J.\ Charatonik, \emph{Hereditarily irreducible maps}, preprint.

\bibitem{anderson} R.D.\ Anderson, \emph{A Characterization of the Universal Curve and a Proof of Its Homogeneity}, Annals of Mathematics, \textbf{67}(2) (1958), 313--324.

\bibitem{anderson2} R.D.\ Anderson, \emph{One-Dimensional Continuous Curves and A Homogeneity Theorem}, Annals of Mathematics, \textbf{68}(1), (1958) 1--16. 

\bibitem{standard} E.\ Berger and H.\ Bruhn, \emph{Eulerian edge sets in locally finite graphs}, Combinatorica \textbf{31} (2011), 21--38.

\bibitem{graphhistory} N.\ Biggs, E.K.\ Lloyd and J.W.\ Wilson, \emph{Graph Theory 1736-1936}, Oxford University Press, 1986.

\bibitem{Brickpartitions} R.H.\ Bing, \emph{Complementary domains of continuous curves}, Fund.\ Math.\ \textbf{36} (1949), 306--318.

\bibitem{partitioningold} R.H.\ Bing, \emph{Partitioning a set}, Bull. Amer. Math. Soc. \textbf{55}(12) (1949), 1101--1110.

\bibitem{partitioning} R.H.\ Bing, \emph{Partitioning continuous curves}, Bull. Amer. Math. Soc. \textbf{58} 1952, 536--556. 

\bibitem{infinitematroids} N.\ Bowler, J.\ Carmesin and R. Christian, \emph{Infinite graphic matroids}, Combinatorica \textbf{38}(2) (2018),  305--339.

\bibitem{Matroids} H.~Bruhn, R.~Diestel, M.~Kriesell, R.~Pendavingh and P.~Wollan, \emph{Axioms for infinite matroids}, Advances in Mathematics \textbf{239} (2013), 18--46.

\bibitem{BruhnStein} H.\ Bruhn and M.\ Stein, \emph{On end degrees and infinite cycles in locally finite graphs}, Combinatorica \textbf{27} (2007), 269--291.

\bibitem{Koenigsberg} W.\ Bula, J.\ Nikiel and E.D.\ Tymchatyn, \emph{The K\"onigsberg bridge problem}, Can. J. Math. \textbf{46} (1994), 1175--1187.

\bibitem{totallyregular} R.D.\ Buskirk, J.\ Nikiel and E.D.\ Tymchatyn, \emph{Totally regular curves as inverse limits}, Houston. J. Math. \textbf{18}(3) (1992), 319--327.


\bibitem{Chapel} C.E.\ Capel, \emph{Inverse limit spaces}, Duke Math. J. \textbf{21} (1954), 233--245.

\bibitem{charatonik} J.J.\ Charatonik, \emph{Monotone mappings of universal dendrites}, Topology Appl. \textbf{38} (1991), 163--187.

\bibitem{Champetier} C.\ Champetier, \emph{Propri\'et\'es statistiques des groupes de pr\'esentation finie}, Advances in Mathematics, \textbf{116}(2) (1995) 197--262.

\bibitem{graphlikeplanar}  R.\ Christian, R.B.\ Richter and B.\ Rooney, \emph{The Planarity Theorems of MacLane and Whitney for Graph-like Continua}, Electron. J. Combin. \textbf{17}:\#R12, 2010.

\bibitem{Diestel} R.\ Diestel, \emph{Graph Theory}, Springer 2016, 5th edition.

\bibitem{dqu} R.\ Diestel, \emph{The cycle space of an infinite graph}, Comb. Probab. Comput. 14 (2005), 59--79.

\bibitem{DSurv} R.\ Diestel, \emph{Locally finite graphs with ends: a topological approach I-III}, Discrete Math \textbf{311--312} (2010--11).

\bibitem{infinitecycles} R.\ Diestel and D.\ K\"uhn, \emph{On infinite cycles I}, Combinatorica \textbf{24} (2004), 68--89.

\bibitem{infinitecycles2} R.\ Diestel and D.\ K\"uhn, \emph{On infinite cycles II}, Combinatorica \textbf{24} (2004), 91--116.

\bibitem{engelkingdimension} R.\ Engelking, \emph{Dimension Theory}, Elsevier, Amsterdam, 1978.

\bibitem{Engelking} R.\ Engelking, \emph{General topology}, second ed., Sigma Series in Pure Mathematics, vol.~6, Heldermann Verlag, Berlin, 1989.

\bibitem{Erdos} P.\ Erdos, T.\ Grunwald, and E.\ Vazsonyi, \emph{\"Uber Euler-Linien unendlicher Graphen,} Journal of Mathematics and Physics \textbf{17}.1-4 (1938) 59--75.

\bibitem{EM} B. Espinoza and E. Matsuhashi, \emph{Arcwise increasing maps}, Topology Appl. \textbf{190} (2015) 74--92.

\bibitem{euleriangraphlike} B.\ Espinoza, P.\ Gartside and M.\ Pitz, \emph{Eulerian graph-like spaces}, J. Graph Theory \textbf{95}(2) (2020), 209-239.

\bibitem{agelos} A.\ Georgakopoulos, \emph{Topological circles and Euler tours in locally finite graphs}, Electronic J. Comb. \textbf{16}:\#R40 (2009).

\bibitem{agelosedgelength} A.\ Georgakopoulos, \emph{Graph topologies induced by edge lengths}, Discrete Math. \textbf{311} (special issue 2011), 1523--1542. 

\bibitem{halin} R.\ Halin, \emph{S-functions for graphs}, Journal of Geometry \textbf{8} (1976), 171--186.

\bibitem{harrold} O.G.\ Harrold, Jr, \emph{A note on strongly irreducible maps of an interval}, Duke Math. J. \textbf{6}(3) (1940), 750--752. 

\bibitem{harrold2} O.G.\ Harrold, Jr, \emph{A mapping characterization of Peano spaces}, Bull. Amer. Math. Soc. \textbf{48}(8) (1942), 561--566.

\bibitem{Hilbert} D.\ Hilbert, \emph{\"Uber die stetige Abbildung einer Linie auf ein Fl\"achenst\"uck}, Math. Annalen \textbf{38} (1891) 459--460.


\bibitem{Kapovich} M.\ Kapovich and B.\ Kleiner, \emph{Hyperbolic groups with low-dimensional boundary}, Annales scientifiques de l'\'Ecole Normale Sup\'erieure, S\'erie 4, \textbf{33}(5) (2000)  647--669. 

\bibitem{kra} J.\ Krasinkiewicz, \emph{On two theorems of Dyer}, Colloq. Math., \textbf{50} (1986) 201--208.

\bibitem{kuratowski} K.\ Kuratowski, \emph{Topology} Vol. 2, Academic Press, New York and London, PWN--Polish Scientific Publishers, Warszawa, 1968. 

\bibitem{Laviolette} F.\ Laviolette, \emph{Decompositions of infinite graphs: Part II circuit decompositions,} Journal of Combinatorial Theory, Series B \textbf{94}(2) (2005) 278--333.

\bibitem{Mengercurve} J.C.\ Mayer, L.\ Oversteegen and E.D.\ Tymchatyn, The Menger curve. Characterization and extension of homeomorphisms of non-locally-separating closed subsets. Diss. Math. \textbf{252}, 1986.

\bibitem{mill} J.\ van Mill, \emph{The Infinite-Dimensional Topology of Function Spaces}, Elsevier, 2001.

\bibitem{Nadler87} S.B.\ Nadler, Jr, \emph{Induced universal maps and some hyperspaces with the fixed point property}, Proc. Amer. Math. Soc. \textbf{100}(4) (1987), 749--754.

\bibitem{Nadler} S.B.\ Nadler, Jr, \emph{Continuum Theory: An Introduction}, Pure and Applied Mathematics Series, Vol. 158, Marcel Dekker, Inc., New York and
Basel, 1992.

\bibitem{NW} C.\ St.\ J.\ A.\ Nash-Williams, \emph{Decompositions of graphs into closed and endless chains}, Proc. London Math. Soc. \textbf{10} (3) (1960), 221--238.

\bibitem{NW2} C.\ St.\ J.\ A.\ Nash-Williams \emph{Decomposition of graphs into two-way infinite paths,} Can. J. Math. \textbf{15} (1963), 479--485.

\bibitem{Nikiel} J.\ Nikiel, \emph{Images of arcs---a non-separable version of the Hahn-Mazurkiewicz theorem,} Fund. Math.\ \textbf{129} (1988), 91--120.

\bibitem{inverselimitsnikiel} J.\ Nikiel, \emph{Locally connected curves viewed as inverse limits}, Fund. Math.\ \textbf{133}(2) (1989), 125--134. 

\bibitem{noebling33} G.\ N\"obeling, \emph{Regulare Kurven als Bilder der Kreislinie}, Fund.\ Math. \textbf{20} (1933), 30--46.

\bibitem{oxley} J.G.\ Oxley, \emph{Matroid theory,} Vol. 3. Oxford University Press, USA, 2006.

\bibitem{graphminors2} R.\ Robertson and P.D.\ Seymour, \emph{Graph minors. II. Algorithmic aspects of tree-width}, Journal of Algorithms, \textbf{7}(3) (1986), 309--322.

\bibitem{rothschild} B.\ Rothschild, \emph{The decomposition of graphs into a finite number of paths}, Can.\ J.\ Math. \textbf{17} (1965), 468--479.

\bibitem{Sabidussi} G.\ Sabidussi, \emph{Infinite Euler graphs} Canad.\ J.\ Math. \textbf{16} (1964), 821--838.

\bibitem{sternfeld} Y.\ Sternfeld, \emph{Mappings in dendrites and dimension}, Houst.\ J.\ Math.\ \textbf{19} (1993) 483--497.

\bibitem{thomassenvella} C. Thomassen and A. Vella, \emph{Graph-like continua, augmenting arcs, and Menger's Theorem}, Combinatorica \textbf{28} (5) (2008) 595--623.

\bibitem{finiteoscillation} L.B.\ Treybig, \emph{Mappings of finite oscillations at local separating points}. In: Topology and Order Structures II, Math. Centre Tracts \textbf{169}, Amsterdam, 1983, 81--89. 

\bibitem{treybig} L.B.\ Treybig and L.E.\ Ward, Jr.,  \emph{The Hahn-Mazurkiewicz problem}. In: Topology and Order Structures I, Math. Centre Tracts \textbf{142}, Amsterdam, 1981, 95--105.



\bibitem{ward2} L.E.\ Ward, Jr., \emph{A generalization of the Hahn-Mazurkiewicz Theorem}, Proc. Amer. Math. Soc. \textbf{58} (1976), 369-3-74.

\bibitem{ward} L.E.\ Ward, Jr., \emph{An irreducible Hahn-Mazurkiewicz theorem}, Houston J. Math. \textbf{3} (1977), 285--290.

\bibitem{Whyburn} G.T.\ Whyburn, \emph{Analytic topology}, 1948, American Mathematical Soc.\ (Vol. 28).

\bibitem{Whyburn2} G.T. Whyburn, \emph{Sets of local separating points of a continuum}, Bull. Amer. Math. Soc. \textbf{39} (1933), 97--100.

\end{thebibliography}
\end{document}